\documentclass[a4paper,12pt,reqno]{amsart}

\usepackage[T1]{fontenc}
\usepackage{lmodern}
\usepackage[utf8]{inputenc}

\usepackage[left=2.2cm,right=2.2cm,top=2.2cm,bottom=2.7cm]{geometry}
\usepackage{amsthm,amsmath,amsfonts,amssymb,mathtools,extarrows}
\usepackage[bbgreekl]{mathbbol}
	\DeclareSymbolFontAlphabet{\mathbb}{AMSb}
	\DeclareSymbolFontAlphabet{\mathbbl}{bbold}
\usepackage[numbers]{natbib}
\usepackage[
	colorlinks,
	hyperfootnotes=true,
	pagebackref,
	hyperindex,
	linkcolor=blue,
	]{hyperref}
\usepackage{xcolor} 

\allowdisplaybreaks

\numberwithin{equation}{section}

\theoremstyle{plain}
\newtheorem{theorem}{Theorem}[section]
\newtheorem{lemma}[theorem]{Lemma}

\newtheorem{corollary}[theorem]{Corollary}

\theoremstyle{definition}
\newtheorem{definition}[theorem]{Definition}

\theoremstyle{remark}
\newtheorem{remark}[theorem]{Remark}
\newtheorem{example}[theorem]{Example}

\newcommand{\N}{\mathbb{N}}
\newcommand{\Zz}{\mathbb{Z}}
\newcommand{\R}{\mathbb{R}}
\newcommand{\C}{\mathbb{C}}

\newcommand{\Pp}{\mathbb{P}}
\newcommand{\Un}{\mathbbl{1}}
\newcommand{\D}{\mathbb{D}}
\newcommand{\E}{\mathbb{E}}
\newcommand{\Var}{\operatorname{Var}}

\newcommand{\Normal}{\mathcal{N}}

\newcommand{\ensemble}[1]{ \left\lbrace #1 \right\rbrace } 
\newcommand{\prth}[1]{\!\left( #1 \right) }
\newcommand{\abs}[1]{\left| #1 \right|}  
\newcommand{\norm}[1]{\left\Vert #1 \right\Vert}  

\newcommand{\Unens}[1]{ \Un_{ \ensemble{#1} } }
\newcommand{\1}[1]{\Unens{#1}}

\newcommand{\G}{{G_0}}
\newcommand{\vertices}[1]{v_{#1}}
\newcommand{\edges}[1]{e_{#1}}
\newcommand{\aut}[1]{a_{#1}}
\newcommand{\Psimin}{\Psi_{\min}}
\newcommand{\induced}[1]{\operatorname{graph}(#1)}
\newcommand{\copies}{\mathcal{M}}
\newcommand{\copiesk}[1]{\mathcal{M}_{#1}}
\newcommand{\graph}{\Gamma}
\newcommand{\V}[3]{V_{#1,\graph_{#2},\graph_{#3}}}

\title[From Berry--Esseen to Moderate Deviation]{Moderate deviations \\ for functionals over infinitely many \\ Rademacher random variables}
\author[M. Butzek]{Marius Butzek}
\address{Faculty of Mathematics, Ruhr University Bochum, Germany.}
\email{marius.butzek@rub.de}
\author[P. Eichelsbacher]{Peter Eichelsbacher}
\address{Faculty of Mathematics, Ruhr University Bochum, Germany.}
\email{peter.eichelsbacher@rub.de}
\author[B. Redno\ss]{Benedikt Redno\ss}
\address{Faculty of Mathematics, Ruhr University Bochum, Germany.}
\email{benedikt.rednoss@rub.de}
\begin{document}
\begin{abstract}
In this paper, moderate deviations for normal approximation of functionals over infinitely many Rademacher random variables are derived. They are based on a bound for the Kolmogorov distance between a general Rademacher functional and a Gaussian random variable, continued by an intensive study of the behavior of operators from the Malliavin--Stein method along with the moment generating function of the mentioned functional. As applications,  subgraph counting in the Erd\H{o}s--R\'enyi random graph and infinite weighted 2-runs are studied.
\end{abstract}
\maketitle
%
%
\section{Introduction and applications}
\noindent The theory of moderate deviations goes back to H. Cram\'{e}r in 1938: Knowing that for an independently identically distributed (i.i.d.) sequence $(X_k)_{k \in \N}$ of random variables such that $\mathbb{E}[X_k] = 0, \mathbb{E}(X_k^2) = 1, W_n := n^{-1/2} (X_1 + ... + X_n)$ and $\Phi$ the standard normal distribution function, the statement
\begin{align}\label{MDCramer1}
\frac{\mathbb{P}(W_n > x)}{1 - \Phi(x)} \longrightarrow 1 \quad \text{for} \quad x = O(1)
\end{align}
is valid, he was asking what happens if $x$ depends on $n \in \N$ such that $x \rightarrow \infty$ for $n \rightarrow \infty$? Can we find an interval such that \eqref{MDCramer1} holds for $0 \leq x \leq I(n), I(n) \rightarrow \infty$? The answer was given by himself: Under the assumption $\mathbb{E}[e^{t_0 \abs{X_1}}] \leq c < \infty $ for some $t_0 >0$ and a constant $c$, 
\begin{align}\label{MDCramer2}
\frac{\mathbb{P}(W_n > x)}{1 - \Phi(x)} \leq 1 + A \, n^{-1/2}(1+x^3) \quad \text{for} \quad 0 \leq x \leq a \, n^{1/6},
\end{align}
and $A$ and $a$ are positive constants depending only on $t_0$ and $c$. The result is optimal (see e.g. \cite{CT18} and \cite{Pe75}). Reminiscent of \eqref{MDCramer2} for a sequence $(Y_n)_{n \in \N}$ of random variables, such that $Y_n \overset{d}{\rightarrow} Y$, the moderate deviation of Cram\'{e}r-type is given by
\begin{align*}
\frac{\mathbb{P}(Y_n > x)}{\mathbb{P}(Y > x)} = 1 + \text{error term} \rightarrow 1
\end{align*}
with range $0 \leq x \leq a_n$, where $a_n \rightarrow \infty$ for $n \rightarrow \infty$.
\\ In \cite{Z23} Zhang was able to prove Cram\'{e}r-type moderate deviations for unbounded exchangeable pairs. He developed them by stopping the proof of the corresponding Berry--Esseen-type inequalities, he had obtained before by Stein's method, at a certain point and continuing differently. Stein's method 
is a powerful tool by itself to derive upper bounds for differences of probability distributions, originally for the normal distribution and later extended to other distributions. Zhang rearranged the fragments of the so called Stein-equation and the bound of its solution, a technique that was already seen in \cite{CFS13}, \cite{FLS20}, \cite{R07} and \cite{SZZ21}.
\\ Our ambition is to prove a similar result for functionals over infinitely many independent Rademacher random variables taking values +1 and $-1$ only. This type of result intersects with \cite{FK22}, where the authors obtain Cram\'{e}r-type moderate deviations via $p$-Wasserstein bounds, and we will refer to that. For Rademacher functionals a Kolmogorov bound in the context of normal approximation was shown recently in \cite[Theorem~3.1]{ERTZ22} such that the bounding terms can be expressed in terms of operators of the so called Malliavin--Stein method. Normal approximation of functionals over infinitely many Rademacher random variables was derived
already in \cite{NPR10}, \cite{KRT17}, \cite{KT17} and \cite{DK19}. Theorem 3.1 in \cite{ERTZ22} will be our starting point.

\subsection{Application to infinite weighted 2-runs}
To begin with, we introduce some of the possible applications for our theorem. In what follows, we use the usual big-O notation $O(.)$ with the meaning that the implicit constant does not depend on the parameters in brackets. For a sequence $a = (a_i)_{i \in \Zz}$ and $p>0$ we write $\norm{a}_{l^p(\Zz)} := \left( \sum_{i \in \Zz} \abs{a_i}^p \right)^{1/p}$.

Due to their simple dependence structure, runs, and more generally weighted or incomplete $U$-statistics, lend themselves to normal approximations,
see \cite{RR97}, where an exchangeable pair coupling is employed for a normal approximation. In \cite{RR97} the authors studied even degenerate weighted 
$U$-statistics, where either weights are considered which ensures a weak dependence or kernel functions are considered which depend on the sample size $n$ in a specific way. See also \cite{NW88}, where subgraph counts in random graphs are considered, see Subsection 1.2. Here we consider infinite weighted 2-runs, where random variables are possibly depending on the whole infinite sequence of i.i.d. Rademacher random variables.

Let $X = (X_i)_{i \in \Zz}$ be a double-sided sequence of i.i.d. Rademacher random variables such that $\Pp(X_i = 1) = \Pp(X_i = -1) = \frac12$ and let for each $n \in \N, (a^{(n)}_i)_{i \in \Zz}$ be a double-sided summable sequence of real numbers. Usually 2-runs are definded with a square-summable sequence but this will be not enough.
\\ The sequence $(F_n)_{n \in \N}$ of standardized infinite weighted 2-runs is then defined as
\begin{align*}
F_n := \frac{G_n - \mathbb{E}[G_n]}{\sqrt{\Var(G_n)}}, \quad G_n := \sum_{i \in \Zz} a^{(n)}_i \xi_i \xi_{i+1}, \quad n \in \N,
\end{align*}
where $\xi_i := (X_i + 1)/2$ for $i \in \Zz$. More generally one can consider an infinite weighted $d$-run defined by
\begin{align*}
G_n(d) := \sum_{i \in \Zz} a^{(n)}_i \xi_i \cdots \xi_{i+d-1},
\end{align*}
which is a weighted degenerate $U$-statistic of degree $d$.
However, since the analysis for any $d$ is of the cost of a quite cumbersome notation, we will focus on the case where $d=2$ (2-runs).
 

For recent results on 2-runs combined with  Malliavin--Stein method see \cite{ERTZ22}, \cite{KRT16} and \cite{NPR10}. Our moderate deviation is given as follows. 

\begin{theorem}\label{Theorem:MDP_Tworuns}
Recall the definition of $F_n$ from above. Then 
\begin{align}\label{TwoRunsFinal}
\frac{\Pp(F_n > z)}{1 - \Phi(z)} = 1 + O(1) (1+z^2) \gamma_{n}(z),
\end{align}
for $0 \leq z \leq \min\{ C_{n}^{-1/3}, C_{n}^{-2},\Var(G_n)^{1/2} \}$ such that $(1+z^2) \gamma_{n}(z) \leq 1$, where O(1) is bounded by a constant only depending on the coefficient sequence $(a^{(n)}_i)_{i \in \Zz}$ and
\begin{align*}
	\gamma_{n}(z) & := e^{O(1)z(\Var(G_n))^{-1/2}} \prth{(1+z^{1/2}+z) C_{n}}, \quad C_{n} := \frac{\norm{a^{(n)}}_{l^4(\Zz)}^2}{ \Var(G_n)}.
\end{align*}
\end{theorem}
\begin{remark}\label{rem12}
The constant $C_n$ has an important meaning. It is the order of the corresponding Berry-Esseen bound of the Kolmogorov distance $\sup_{z \in \mathbb R} | \Pp(F_n \leq z) - \Phi(z) |$ between the distribution function of $F_n$ and the standard normal distribution function in \cite[Theorem~1.1]{ERTZ22}. Depending on the coefficient sequence, $C_n$ can behave differently: By \eqref{two_runs_variance} and \eqref{cons_simp_2} $C_n$ is in general bounded by a constant, but it can be a constant itself (see e.g. $a^{(n)}_i = \frac{1}{i^2}$). So, to make \eqref{TwoRunsFinal} tend to 0 and the range increase in $n$, the condition $C_{n} \rightarrow 0$ for $n \rightarrow \infty$ is sufficient. We give now examples, where this is the case and where the resulting rate is optimal.
\end{remark}
\begin{example}
We consider $a^{(n)}_i =\Unens{\abs{i}\leq n} \, \forall \, i \in \Zz$, which is obviously a summable sequence. Then $\norm{a^{(n)}}_{l^4(\Zz)}^2 = O\prth{n^{1/2}},\Var(G_n) = O\prth{n}$ and $ C_{n}= O\prth{n^{-1/2}}$. The moderate deviation we get is 
\begin{align}\label{TwoRunsExample}
\frac{\Pp(F_n > z)}{1 - \Phi(z)} = 1 + O(1) (1+z^2) \gamma_{n}(z),
\end{align}
for $0 \leq z \leq n^{1/6}$ such that $(1+z^2) \gamma_{n}(z) \leq 1$, where
\begin{align*}
	\gamma_{n}(z) & := e^{O(1)zn^{-1/2}} \prth{(1+z^{1/2}+z) n^{-1/2}}.
\end{align*}
In order to discuss the quality of this result, we use a lower bound of the Kolmogorov distance known from \cite[Theorem~1(c)]{En81}, which got later refined by \cite[Corollary~3.12]{Re23}.
Since $G_n$ is almost surely an integer between $-n$ and $n$, said results imply that the Kolmogorov distance for normal approximation of $F_n$ is bounded from below by $c_0 \cdot (\Var(G_n))^{-\frac12}$
for some constant $c_0 > 0$.
As $\Var(G_n)$ is of order $n$, we conclude that $C_n$ being of order $n^{-\frac12}$ is optimal.
\end{example}
\begin{example}
We generalize the previous example to $a^{(n)}_i = n^{-\beta} \Unens{\abs{i}\leq n^\alpha} \, \forall \, i \in \Zz, \alpha \in \R, \beta > 0$. Then $\norm{a^{(n)}}_{l^4(\Zz)}^2 = O\prth{n^{(\alpha - 4\beta)/2}},\Var(G_n) = O\prth{n^{\alpha - 2\beta}}$ and $ C_{n}= O\prth{n^{-\alpha/2}}$. If we choose $\beta = 0$ and $\alpha \geq 1$ the moderate deviation we get is of the same form as \eqref{TwoRunsExample} with range $0 \leq z \leq n^{\alpha/6}$ respectively $0 \leq z \leq (\Var(G_n) )^{\alpha/6}$.
Using the same argumentation as in the previous example, we see that the rate of $C_n$ is again optimal.
\end{example}

\subsection{Application to subgraph counts in the Erd\H{o}s--R\'{e}nyi random graph}

As a further application we derive a Cram\'{e}r-type moderate deviation result for the subgraph counting statistic
in the Erd\H{o}s--R\'{e}nyi random graph $G(n,p)$.

Consider a random graph on $n \in \N$ vertices.
Each possible edge between two vertices is included with probability $p \in (0,1)$
independently of all other edges,
where $p$ may depend on $n$ even though we will not make this visible in our notation.

Let $\G$ be a fixed graph with at least one edge.
A subgraph $H \subset G(n,p)$ is called a \emph{copy} of $\G$ in $G(n,p)$ if it is isomorphic to $\G$.
Note that we are calling two graphs $G_1=(V_1,E_1)$ and $G_2=(V_2,E_2)$ \emph{isomorphic}
if there is an edge-preserving bijection $f: V_1 \rightarrow V_2$ between their sets of vertices,
such that two vertices $v,w \in V_1$ are joined by an edge $\{v,w\} \in E_1$ in $G_1$ if and only if the vertices $f(v),f(w) \in V_2$ are joined by an edge $\{f(v),f(w)\} \in E_2$ in $G_2$.

We are interested in the standardized number $W$ of copies of $\G$ in $G(n,p)$.

It is well known under which necessary and sufficient assumption $W$ is asymptotically normal.
To state this condition and further results, we will use the notation $q := 1-p$ as well as
\begin{align*}
	\Psimin := \Psimin(\G)
	:= \min_{\substack{H \subset \G \\ \edges{H} \geq 1}} \{ n^{\vertices{H}} \cdot p^{\edges{H}} \},
\end{align*}
with $H \subset \G$ a subgraph of $\G$,
$\vertices{H}$ the number of vertices of $H$,
and $\edges{H}$ the number of edges of $H$.
Further, $\aut{H}$ denotes the number of automorphisms of $H$.
Some very important yet easy to prove bounds for $\Psimin$ are
\begin{align*}
	n^2 p^{\edges{\G}} \leq \Psimin \leq n^2p.
\end{align*}

Theorem~2 in \cite{R88} states that $W$ is asymptotically normal if and only if
\begin{align*}
	q \cdot \Psimin \xlongrightarrow{n\to\infty} \infty.
\end{align*}
The best known convergence rate regarding the bounded Wasserstein distance $d_{\mathrm{bWass}}$ is presented in Theorem~2 in \cite{BKR89}:
There is a constant $C_\G$ only depending on $\G$ so that
\begin{align*}
	d_{\mathrm{bWass}}(W,N) &\leq \frac{C_\G}{\sqrt{q\Psimin}},
\end{align*}
where $N \sim \Normal(0,1)$ is standard normal distributed, and 
$$
d_{\mathrm{bWass}}(W,N) = \sup_{h \in {\mathcal H}} \big| \E h(W) - \E h(N)  \big| 
$$
with ${\mathcal H}$ being the set of all bounded Lipschitz-functions with $\|h\|_{\infty} + \|h' \|_{\infty} \leq 1$.
The same bound holds true for the Wasserstein distance, where ${\mathcal H}$ is the collection of Lipschitz-functions with Lipschitz-constant 1. The result can be shown by only a slight modification of the original proof in \cite{BKR89}.
For the Kolmogorov distance, the same order of the bound can be obtained. This was a long standing problem that has been lately solved using different approaches by \cite{PrS20}, \cite{Z19}, \cite{ER22} and \cite{ERTZ22}.

A refined Cram\'{e}r-type moderate deviation result for the subgraph counting statistic in the Erd\H{o}s--R\'{e}nyi random graph is shown in Proposition~2.3 in \cite{FLS20}:
For any fixed constant $c_0 > 0$ there is a constant $C_{c_0,\G}$ only depending on $c_0$ and $\G$, so that
\begin{align*}
	\max
		\bigg\{
			\bigg\vert
				\frac{\Pp(W<-t)}{\Phi(-t) \exp(-\frac{\gamma t^3}{6})} -1
			\bigg\vert,
			\bigg\vert
				\frac{\Pp(W>t)}{(1-\Phi(t)) \exp(\frac{\gamma t^3}{6})} -1
			\bigg\vert
		\bigg\}
	\leq
		C_{c_0,\G} \cdot
		\bigg(
			\frac{\Psimin}{q \cdot p^{2\edges{\G}}}
		\bigg)^{\frac52}
		\cdot \frac{1+t^2}{n^6}
\end{align*}
for
\begin{align*}
	0
	\leq
	t
	\leq
		c_0 \cdot
		\bigg(
			\frac{q \cdot p^{2\edges{\G}}}{\Psimin}
		\bigg)^{\frac54}
		\cdot n^3
	,
\end{align*}
where $\Phi$ is the distribution function of
the standard normal distribution,
and $\gamma = \E[W^3]$.

Another Cram\'{e}r-type moderate deviation result is known from Theorem~3.1 in \cite{Z19}:
There is a constant $C_\G$ only depending on $\G$, so that
\begin{align}
	\bigg\vert
		\frac{\Pp(W>t)}{1-\Phi(t)} -1
	\bigg\vert
	\leq
		C_\G
		\cdot (1+t^2)
		\cdot b_n(p,t)
	\label{sg:ineq:ZhangResult}
\end{align}
with
\begin{align*}
	b_n(p,t)
	:={}&
	\begin{cases}
		\displaystyle
		\frac{1+t}{\sqrt{\Psimin}}, & 0 < p < \frac12,\\[1.5em]
		\displaystyle
		\frac{1+\frac{t}{\sqrt{q}}}{n\sqrt{p}}, & \frac12 < p < 1,
	\end{cases}
\end{align*}
for all
\begin{align*}
	0
	\leq
	t
	\leq
		\frac{q \cdot p^{\edges{\G}}}{\sqrt{\Psimin}}
		\cdot n^2
\end{align*}
that satisfy
\begin{align*}
	(1+t^2) b_n(p,t) \leq 1
	.
\end{align*}
For any subgraph $H \subset \G$ with at least one edge but more than two vertices, there is
$n^{\vertices{H}} p^{\edges{H}}
\geq n^2p \cdot \frac{n^{\vertices{H}-2}}{2^{\edges{H}-1}}
\geq n^2p \cdot \frac{n}{2^{\edges{\G}-1}}$
in case of $\frac12 < p < 1$.
Hence, we know that $\Psimin = n^2 p$
for $\frac12 < p < 1$ and $n \geq 2^{\edges{\G}-1}$.
We further know that
$\frac{1+t}{\sqrt{\Psimin}}
\leq \frac{1+\frac{t}{\sqrt{q}}}{\sqrt{q\Psimin}}
\leq 2 \cdot \frac{1+t}{\sqrt{\Psimin}}$
in case of $0<p<\frac12$.
Therefore, as long as $n \geq 2^{\edges{\G}-1}$,
\eqref{sg:ineq:ZhangResult} can be rewritten
without loss of sharpness in the resulting rate:
\begin{align}
	\bigg\vert
		\frac{\Pp(W>t)}{1-\Phi(t)} -1
	\bigg\vert
	&\leq
		2 \cdot C_\G
		\cdot \frac{(1+t^2)\big(1+\frac{t}{\sqrt{q}}\big)}{\sqrt{q\Psimin}}
	.
	\label{sg:ineq:ZhangSimplified}
\end{align}

Further, Theorem~5.2 in \cite{FK22} presents a moderate deviation result for bounded and locally dependent random variables.
Application to the context of subgraph counting yields
that there exist constants $c_\G$ and $C_\G$ so that
\begin{align}
	\bigg\vert
		\frac{\Pp(W>t)}{1-\Phi(t)} -1
	\bigg\vert
	&\leq
		C_\G \cdot (1+t) \cdot ( 1 + \vert \log(\Delta) \vert + t^2 ) \cdot \Delta
	\label{sg:ineq:FK22}
\end{align}
for $0 \leq t \leq \Delta^{-\frac13}$
with $ \Delta =
	\frac{\Psimin}{n^3 p^{2\edges{\G}} q}
	+
	\left( \frac{\Psimin}{n^3 p^{2\edges{\G}} q} \right)^{\frac32}
	\sqrt{n} \log(n)
	\leq c_\G$.
As one can show that $\frac{1}{\sqrt{q\Psimin}} = o \left( \Delta \right)$, result \eqref{sg:ineq:ZhangSimplified} from \cite{Z19} is stronger than \eqref{sg:ineq:FK22} from \cite{FK22}.

With our approach we are able to improve on the result of \cite{Z19}.
We will prove the following result.

\begin{theorem}[Subgraph counts in the Erd\H{o}s--R\'{e}nyi random graph --- general result]
	\label{sg:theo:main}
	Let $\G$ be a graph with at least one edge.
	Let $W$ be the standardized number of copies of $\G$
	in the Erd\H{o}s--R\'{e}nyi random graph $G(n,p)$.
	And assume that $n \geq 4\vertices{\G}^2$.
	Then for all $t \geq 0$ there is
	\begin{align*}
		\bigg\vert
			\frac{\Pp(W>t)}{1-\Phi(t)} -1
		\bigg\vert
		\leq
			50c_\G
			\cdot \exp\Big(
				c_\G \cdot t^2 \cdot s(t)
				\Big)
			\cdot (1+t^2)
			\cdot s(t)
		,
	\end{align*}
	with
	\begin{align}
		s(t) &=			
			\bigg(
				1 + \frac{t}{\min\{\sqrt{\Psimin},1\}}
			\bigg)
			\frac{1}{\sqrt{q\Psimin}}
			\cdot
			\exp\bigg(
				\frac{5Dt}{\sigma}
			\bigg)
		\label{sg:fct:s}
		\intertext{and}
		c_\G &= 
		\frac{ 2^{\frac92 + \frac{15}2\edges{\G}} \cdot (\vertices{\G}!)^4 \cdot \edges{\G}}{\aut{\G}^\frac32}
		.
		\label{sg:const:c}
	\end{align}
\end{theorem}

The result holds for all $t \geq 0$.
Similar to the approach by \cite{Z19},
the result can be simplified
by restricting $t$ to be smaller than a suitable bound:

\begin{corollary}[Subgraph counts in the Erd\H{o}s--R\'{e}nyi random graph --- bounded domain]
	\label{sg:cor:main}
	Let $\G$ be a graph with at least one edge.
	Let $W$ be the standardized number of copies of $\G$
	in the Erd\H{o}s--R\'{e}nyi random graph $G(n,p)$.
	Let $c_1, c_2 >0$ be positive numbers.
	Assume that
	\begin{align*}
		n &\geq 4\vertices{\G}^2,
	&
		0
		\leq
		t
		&\leq
			c_1 \cdot
			\frac{ n^2 \cdot p^{\edges{\G}} \cdot \sqrt{q}}{\sqrt{\Psimin}},
	&
		t^2 \cdot s(t) &\leq c_2,
	\end{align*}
	with $s(t)$ as given in \eqref{sg:fct:s}.
	Then
	\begin{align*}
		\bigg\vert
			\frac{\Pp(W>t)}{1-\Phi(t)} -1
		\bigg\vert
		\leq
			C_{c_1,c_2,\G} \cdot \frac{1+t^3}{\sqrt{q\Psimin}}
		,
	\end{align*}
	where
	\begin{align*}
		C_{c_1,c_2,\G}
		&= 
			100
			(1+c_1)
			\cdot c_\G
			\cdot e^{5c_1 \hat c_\G + c_2 c_\G}
		,
	\end{align*}
	with $\hat c_\G = \sqrt{2} \cdot \frac{ \sqrt{\vertices{\G}!} \cdot \vertices{\G}^2 \cdot \edges{\G} }{ \sqrt{\aut{\G}} }$
	and $c_\G$ as given in \eqref{sg:const:c}.
\end{corollary}
In case of $\liminf_{n \rightarrow \infty} q >0$,
our result is of the same order as the result of \cite{Z19} presented in \eqref{sg:ineq:ZhangSimplified}.
However, in case of $\liminf_{n \rightarrow \infty} q = 0$, our result yields the better rate.

We give an overview how the remaining parts of this paper are structured. In Section 2 we list important operators used in Malliavin calculus and give a short introduction to Stein's method. The new moderate deviations for general non-linear functionals of possibly infinite Rademacher random variables are presented in Section 3. The proof of Theorem~\ref{Theorem:MDP_Tworuns} is shown in Section 4 and the proofs of Theorem~\ref{sg:theo:main} and Corollary~\ref{sg:cor:main} follow in Section 5.

\section{Preliminaries}
\noindent In this section we list all the definitions and notions we deal with, in particular the operators from Malliavin calculus. Since this is just a summary, we refer to \cite{NP12} for details and to \cite{Webpage} for further results related to the topic. For the setting of Bernoulli processes see also \cite{Pr08}.
\\ Throughout the upcoming definitions we will need the following notations: Let $l^2(\N)$ be the space of real square-summable sequences. Moreover, by $l^2(\N)^{\otimes p}$ we mean the $p$th tensor product of $l^2(\N)$ for $p \in \N$. Relevant subsets are $l^2(\N)^{\circ p}$, the symmetric functions in $l^2(\N)^{\otimes p}$, and $l_0^2(\N)^{\circ p}$, the symmetric functions in $l^2(\N)^{\otimes p}$ which vanish on diagonals.
\\ We start with $X = (X_k)_{k \in \N}$, a sequence of Rademacher random variables, e.g. $\forall k \in \N$:
\begin{align*}
\Pp(X_k = 1) & = p_k \in (0,1), \\
\Pp(X_k = -1) & = q_k = 1 - p_k,
\end{align*}
and, if needed, the standardized random variable
\begin{align*}
Y_k = \frac{X_k - p_k + q_k}{2 \sqrt{q_k p_k}}.
\end{align*}
We are interested in random variables $F \in L^2(\Omega, \sigma(X), \Pp)$ and we will use that in our setting $F$ can be written as 
\begin{align*}
F = \mathbb{E}[F] + \sum_{n=1}^\infty J_n(f_n),
\end{align*}
where
\begin{align*}
	J_n(f) = \sum_{(i_1,...,i_n) \in \Delta^n} f(i_1,...,i_n)Y_{i_1}Y_{i_2}... Y_{i_n} 
\end{align*}
with $f \in l_0^2(\N)^{\circ n}$ and $\Delta_n := \{ (i_1,...,i_n) \in \N^n : i_j \neq i_k \, \text{for} \, j \neq k \}$. $J_n(f)$ is called the $n$th discrete multiple integral. 
For $F = f(X) = f(X_1,X_2,...) \in L^1(\Omega, \sigma(X), \Pp)$ we define the discrete gradient $D_kF$ of $F$ at $k$th coordinate:
\begin{align*}
D_kF & := \sqrt{p_k q_k} (F_k^+ - F_k^-), \\
DF & := (D_1F, D_2F,...),
\end{align*}
where $F_k^+ := f(X_1,...,X_{k-1},+1,X_{k+1},...)$ and $F_k^- := f(X_1,...,X_{k-1},-1,X_{k+1},...), \, k \in \N$. For the main result we will consider
\begin{align*}
F \in \D^{1,2} := \{F \in L^2(\Omega, \sigma(X), \Pp) \vert \, \mathbb{E}[\norm{DF}_{l^2(\N)}^2] < \infty \},
\end{align*}
where
\begin{align*}
\mathbb{E}[\norm{DF}_{l^2(\N)}^2] := \mathbb{E}\left[ \sum_{k \in \N} (D_k F)^2 \right].
\end{align*}
Next we define the divergence operator $\delta$, also known as Skorokhod operator, and its domain $Dom(\delta)$. For $u := (u_k)_{k \in \N} \in (L^2(\Omega))^\N$ with
\begin{align*}
	u_k := \sum_{n=1}^{\infty} J_{n-1}(f_{n}(\cdot,k)),
\end{align*}
where $f_{n} \in l_0^2(\N)^{\circ n-1} \otimes  l^2(\N)$ for $n \in \N$, we say that $u \in Dom(\delta)$, if
\begin{align*}
	\sum_{n=1}^\infty  n! \norm{\widetilde{f_{n}}\Un_{\Delta_{n}}}^2_{ l^2(\N)^{\otimes n}} < \infty.
\end{align*}
By $\tilde{f}(k_1,...,k_n) := \frac{1}{n!} \sum_{\sigma \in \mathcal{G}_n} f(k_{\sigma(1)},...,k_{\sigma(n)}) $ we mean the canonical symmetrization of a function $f$ in n variables such that $\mathcal{G}_n$ is the symmetric group on $\{1,...,n\}$. Then, for $u \in Dom(\delta)$, the operator $\delta$ is given by
\begin{align*}
	\delta(u) := \sum_{n=1}^\infty J_{n}\prth{\widetilde{f_{n}}\Un_{\Delta_{n}}}.
\end{align*}
Another way to characterize $\delta$ is by the duality
\begin{align} \label{Skorokhod}
	\mathbb{E}[\langle DF, u \rangle] = \mathbb{E}[F \delta(u)], \quad F \in \D^{1,2}, u \in Dom(\delta),
\end{align} 
such that we can identify $\delta$ as the adjoint operator of $D$. Furthermore we can rewrite its domain to
\begin{align*}
Dom(\delta) = \left \{ u \in L^2(\Omega, l^2(\N)) \bigg \vert \exists \, C_u > 0 \, \forall \, F \in \D^{1,2} : \abs{\mathbb{E}[\langle DF, u \rangle]} \leq C_u \sqrt{\mathbb{E}[F^2]} \right \}.
\end{align*}
For 
\begin{align*}
	F \in Dom(L) = \left \{ F = \mathbb{E}[F] + \sum_{n=1}^\infty J_n(f_n) \in  L^2(\Omega, \sigma(X), \Pp) \bigg \vert \sum_{n=1}^\infty n^2 n! \norm{f_n}^2_{ l^2(\N)^{\otimes n}}  < \infty \right \} 
\end{align*}
we define by 
\begin{align*}
LF := \sum_{n=1}^\infty -n J_n(f_n), \quad L^{-1}F := \sum_{n=1}^\infty - \frac{1}{n}J_n(f_n)
\end{align*}
the Ornstein--Uhlenbeck operator $L$ and the pseudo-inverse Ornstein--Uhlenbeck operator $L^{-1}$. You can show that $F \in Dom(L)$ is equivalent to $F \in \D^{1,2}$ and $DF \in Dom(\delta)$; in this case, it holds that
\begin{align} \label{Malliavin_operators}
L = -\delta D.
\end{align}
At last we recall the main ideas of Stein's method for normal approximation, starting with the important characterisation
\begin{align}\label{Steinlemma}
	Z \sim \mathcal{N}(0,1) \Leftrightarrow \mathbb{E}[f^\prime(Z) - Z f(Z)] = 0
\end{align}
for all continuous differentiable $f$ such that the appearing expectations exist. So if a random variable is in some sense close to $\mathcal{N}(0,1)$, it is likely that the expectation in \eqref{Steinlemma} is close to 0.  This motivates the Stein-equation, written in the case of Kolmogorov distance, namely
\begin{align*}
f_z^\prime(F) - F f_z(F) = \Unens{F \leq z} - \Phi(z),
\end{align*}
respectively
\begin{align*}
\mathbb{E}[f_z^\prime(F) - F f_z(F)] = \Pp(F \leq z) - \Phi(z).
\end{align*}
The solution to this equation is given, see Lemma 2.1 in \cite{BC05}, by
\begin{align}\label{Stein_solution}
f_z(w) = \begin{cases}
\frac{\Phi(w) (1-\Phi(z))}{p(w)} & w \leq z, \\
\frac{\Phi(z) (1-\Phi(w))}{p(w)} & w > z,
\end{cases}
\end{align}
where $p(w) = e^{-w^2/2} / \sqrt{2 \pi}$ is the density of $\mathcal{N}(0,1)$. Later on we will use the following bounds, see Lemma 2.3 in \cite{CGS11}:
\begin{align}\label{Stein_prop_1}
	\frac{1-\Phi(w)}{p(w)} \leq \min\left\{\frac{1}{w}, \frac{\sqrt{2\pi}}{2} \right\},  \quad w > 0,
\end{align}
\begin{align}\label{Stein_prop_2}
\abs{w f_z(w)} \leq 1, \quad w \in \R,
\end{align}
and
\begin{align}\label{Stein_prop_3}
\abs{w f_z(w)} \leq 1 - \Phi(z), \quad w < 0 \leq z.
\end{align}
Note that \eqref{Stein_prop_3} follows with \eqref{Stein_solution} and \eqref{Stein_prop_1} by writing
\begin{align*}
	\abs{w f_z(w)} & = \abs{w} \sqrt{2\pi} e^{w^2/2} \Phi(w) (1-\Phi(z)) \\
	& = (1-\Phi(z))  (1-\Phi(\abs{w})) \sqrt{2\pi} \abs{w}e^{\abs{w}^2/2} \\
	& \leq (1-\Phi(z)) ,
\end{align*}
where we also used the symmetry of $\Phi$. We will need this more precise bound of $\abs{w f_z(w)}$ for the main result in section 3, where we distinguish different cases for $w \in \R$.

\section{Main Results}
\noindent Now we present the main theorem of our paper.
\begin{theorem}[Moderate deviations for Rademacher functionals]\label{Theorem:MD_Rademacher}
Let $F \in \D^{1,2}$ with $\mathbb{E}[F] = 0, \Var(F) = 1$, and
\begin{align*}
F f_z(F) + \Unens{F>z} \in \D^{1,2} \quad \forall \,  z \in \R,
\end{align*}
\begin{align*}
\frac{1}{\sqrt{pq}} DF \abs{D L^{-1} F}  \in Dom(\delta).
\end{align*}
Assume that there exists a constant $A > 0$ and increasing functions $\gamma_1(t),\gamma_2(t)$ such that $e^{tF} \in \D^{1,2}$ and
\makeatletter\tagsleft@true\makeatother
\begin{align}
\mathbb{E} \left[ \abs{ 1 - \langle DF , - DL^{-1} F \rangle } e^{tF}\right]
&\leq \gamma_1(t) \mathbb{E} \left[ e^{tF}\right],
\tag{A1}\label{A1}\\
\mathbb{E} \left[ \abs{  \delta\prth{ \frac{1}{\sqrt{pq}} DF \abs{D L^{-1} F} } } e^{tF}\right]
&\leq \gamma_2(t) \mathbb{E} \left[ e^{tF}\right],
\tag{A2}\label{A2}
\end{align}
\makeatletter\tagsleft@false\makeatother
for all $0 \leq t \leq A$.
For $d_0 \geq 0$, let
\begin{align*}
A_0(d_0) := \max\left\{0 \leq t \leq A : \frac{t^2}{2}\prth{\gamma_1(t) + \gamma_2(t)} \leq d_0 \right\}.
\end{align*}
Then, for any $d_0 \geq 0$,
\begin{align*}
\abs{ \frac{\Pp(F > z)}{1 - \Phi(z)} - 1} \leq 25 e^{d_0} (1+z^2) \prth{\gamma_1(z) + \gamma_2(z)}
\end{align*}
provided that $0 \leq z \leq A_0(d_0)$.
\end{theorem}
In consequence, the following result is achieved.
\begin{theorem}\label{Theorem:MD_Rademacher-short}
	Under the assumptions from Theorem~\ref{Theorem:MD_Rademacher}, there is
	\begin{align*}
		\abs{ \frac{\Pp(F > z)}{1 - \Phi(z)} - 1}
		&\leq 25 e^{\frac{z^2}{2}\prth{\gamma_1(z) + \gamma_2(z)}} (1+z^2) \prth{\gamma_1(z) + \gamma_2(z)}
	\end{align*}
	for all $0 \leq z \leq A$.
\end{theorem}

 As a first application we treat the i.i.d.-case: For our sequence $(X_i)_{i \in \N}$ of Rademacher random variables we consider the standardized $n$th partial sum
\begin{align*}
	F := F_n = \frac{1}{\sqrt{n}} \sum_{i = 1}^n Y_i := \frac{1}{\sqrt{n}} \sum_{i = 1}^n \frac{X_i - (2p-1)}{2 \sqrt{pq}}.
\end{align*}
The classical result can be received:
\begin{corollary}\label{Corollary:MDP_IID}
	Recall the definition of $F_n$ from above. Then 
	\begin{align}
		\frac{\Pp(F_n > z)}{1 - \Phi(z)} = 1 + O(1) (1+z^2) \gamma_{n}(z),
	\end{align}
	for $0 \leq z \leq n^{1/6}$ such that $(1+z^2) \gamma_{n}(z) \leq 1$, where O(1) is bounded by a constant and
	\begin{align*}
		\gamma_{n}(z) & := e^{O(1)z n^{-1/2}} \frac{(1 + z)}{\sqrt{n}}.
	\end{align*}
\end{corollary}

\begin{remark}
	We obtain the optimal range $0 \leq z \leq n^{1/6}$ from \eqref{MDCramer2}, but there is no $\log(n)$ in our error term compared to \cite[Corollary~2.2]{FK22}. The proof of  Corollary~\ref{Corollary:MDP_IID} is basically a shorter version of the proof of Theorem~\ref{Theorem:MDP_Tworuns}.
\end{remark}
For the proof of  Theorem~\ref{Theorem:MD_Rademacher} we will need two auxiliary lemmas.
\begin{lemma}[Bound for the moment generating function]\label{Lemma:BoundMomentGeneratingFunction}
	Under the assumptions of Theorem~\ref{Theorem:MD_Rademacher}, for $0 \leq t \leq A$, we have
	\begin{align}\label{BoundMomentGeneratingFunction1}
		\mathbb{E} \left[ e^{tF}\right] \leq \exp\left\{ \frac{t^2}{2}(1 + \gamma_1(t) + \gamma_2(t)) \right\}.
	\end{align}
	Then, for $0 \leq t \leq A_0(d_0)$,
	\begin{align}\label{BoundMomentGeneratingFunction2}
		\mathbb{E} \left[ e^{tF}\right] \leq e^{d_0} e^{t^2/2}.
	\end{align}
\end{lemma}

\begin{proof}
	Let $h(t) := \mathbb{E} \left[ e^{tF}\right]$. We recall that $\mathbb{E} \left[ e^{tF}\right] < \infty$ is implied by  $e^{tF} \in \D^{1,2}$ for $0 \leq t \leq A$, and so, by the continuity of the exponential funtion, we have $h^\prime(t) = \mathbb{E} \left[ F e^{tF}\right]$. It follows with \eqref{Malliavin_operators} and \eqref{Skorokhod} that
	\begin{align}\label{BoundMomentGeneratingFunction2_5}
		\mathbb{E} \left[ F e^{tF}\right] & = \mathbb{E} \left[ (L L^{-1} F) e^{tF}\right] \nonumber \\
		& = \mathbb{E} \left[ (- \delta D L^{-1} F) e^{tF}\right] \nonumber \\
		& = \mathbb{E} \left[ \langle D e^{tF}, - DL^{-1}F \rangle \right]. 
	\end{align}
	Now we consider the $k$-th component of $D e^{tF}$, which gives us
	\begin{align*}
		D_k e^{tF}  & = \sqrt{p_k q_k} \left[ e^{tF_k^+} - e^{tF_k^-} \right] \\
		& = t \sqrt{p_k q_k} \int_{F_k^-}^{F_k^+} e^{tu} du \\
		& = t \sqrt{p_k q_k} \int_{F_k^-}^{F_k^+} \left[ e^{tu} - e^{tF} \right] du + t e^{tF} D_k F\\
		& =: t R_k + t e^{tF} D_k F.
	\end{align*}
	If we define $R := (R_1,R_2,\ldots)$, we can go on from \eqref{BoundMomentGeneratingFunction2_5} by writing 
	\begin{align}\label{BoundMomentGeneratingFunction3}
		\mathbb{E} \left[ F e^{tF}\right] & = \mathbb{E} \left[ \langle t R, - DL^{-1}F \rangle \right] + \mathbb{E} \left[ \langle t e^{tF} DF, - DL^{-1}F \rangle \right] \nonumber \\
		& \leq t \mathbb{E} \left[ e^{tF}\right] + t \mathbb{E} \abs{ \langle R, - DL^{-1}F \rangle} + t \mathbb{E} \left[ \abs{ 1 - \langle DF , - DL^{-1} F \rangle } e^{tF}\right].
	\end{align}
	Without loss of generality $F_k^- \leq F \leq F_k^+$ since for the other case we  just have to change the sign. Then we can bound $R_k$ as follows. 
	\begin{align*}
		R_k  & = \sqrt{p_k q_k} \int_{F_k^-}^{F_k^+} \left[ e^{tu} - e^{tF} \right] du \\
		& \leq \sqrt{p_k q_k} \int_{F_k^-}^{F_k^+} \left[ e^{tF_k^+} - e^{tF_k^-} \right] du \\
		& = \sqrt{p_k q_k} \left[ e^{tF_k^+} - e^{tF_k^-} \right] \int_{F_k^-}^{F_k^+}  du \\
		& = D_k e^{tF} \cdot \frac{1}{ \sqrt{p_k q_k} } D_k F
	\end{align*}
	and by combining both cases
	\begin{align}\label{Rk_bound}
		\abs{R_k} & \leq \frac{1}{ \sqrt{p_k q_k} }  D_k e^{tF} \cdot D_k F.
	\end{align}
	By condition \eqref{A2} and \eqref{Rk_bound} we get
	\begin{align}\label{BoundMomentGeneratingFunction4}
		t \mathbb{E} \abs{ \langle R, - DL^{-1}F \rangle} & \leq t \mathbb{E} \left[ \langle \abs{R}, \abs{DL^{-1}F} \rangle \right] \nonumber \\
		& \leq t \mathbb{E} \left[ \langle  D e^{tF}, \frac{1}{ \sqrt{p q} } D F \abs{DL^{-1}F} \rangle \right] \nonumber \\
		& \leq t \mathbb{E} \left[ \abs{ \delta\prth{ \frac{1}{\sqrt{pq}} DF \abs{D L^{-1} F} } } e^{tF}\right] \nonumber \\
		& \leq t \gamma_2(t) \mathbb{E} \left[ e^{tF}\right].
	\end{align}
	By condition \eqref{A1}, for $0 \leq t \leq A$,
	\begin{align}\label{BoundMomentGeneratingFunction5}
		t \mathbb{E} \left[ \abs{ 1 - \langle DF , - DL^{-1} F \rangle } e^{tF}\right] \leq t \gamma_1(t) \mathbb{E} \left[ e^{tF}\right].
	\end{align}
	Combining \eqref{BoundMomentGeneratingFunction3}, \eqref{BoundMomentGeneratingFunction4} and \eqref{BoundMomentGeneratingFunction5}, we have for $0 \leq t \leq A$,
	\begin{align*}
		h^\prime(t) & = \mathbb{E} \left[ F e^{tF}\right] \\
		& \leq t h(t) + \{t (\gamma_1(t) + \gamma_2(t)) \} h(t) \\
		& = \{ 1+ \gamma_1(t) + \gamma_2(t) \} \, t \, h(t).
	\end{align*}
	Having in mind that $h(0) = 1$, and $\gamma_1$ and $\gamma_2$ are increasing, we complete the proof of \eqref{BoundMomentGeneratingFunction1} by solving the foregoing differential inequality:
	\begin{align*}
		\log(h(t)) & = \int_0^t \frac{h^\prime(s)}{h(s)} ds \\
		& \leq \int_0^t (1+ \gamma_1(s) + \gamma_2(s)) s ds \\
		& \leq \int_0^t (1+ \gamma_1(t) + \gamma_2(t)) s ds \\
		& = \frac{t^2}{2} \prth{1+ \gamma_1(t) + \gamma_2(t)},
	\end{align*}
	now we apply $\exp(.)$ on both sides. At last, \eqref{BoundMomentGeneratingFunction2} follows immediately from \eqref{BoundMomentGeneratingFunction1} by definition of $A_0(d_0)$.
\end{proof}

\begin{lemma}\label{Lemma:TermsTimesFfF}
	Under the assumptions of Theorem~\ref{Theorem:MD_Rademacher}, we have for $0 \leq z \leq A_0(d_0)$,
	\begin{align}
		\mathbb{E} \left[ \abs{ 1 - \langle DF , - DL^{-1} F \rangle } F e^{F^2/2} \Unens{0 \leq F \leq z} \right] \leq 6 e^{d_0} (1+z^2) \gamma_1(z)
	\end{align}
	and
	\begin{align}
		\mathbb{E} \left[ \abs{ \delta\prth{ \frac{1}{\sqrt{pq}} DF \abs{D L^{-1} F} } } F e^{F^2/2} \Unens{0 \leq F \leq z} \right] \leq 6 e^{d_0} (1+z^2) \gamma_2(z).
	\end{align}
\end{lemma}

\begin{proof}
	Same as \cite{Z23} we apply the idea in \cite[Lemma~5.2]{CFS13} for this proof. For $a \in \R$, denote $[a] = \max\{n \in \N : n \leq a\}$. Next, we define $H := 1 - \langle DF , - DL^{-1} F \rangle$.
	\begin{align*}
		\mathbb{E} \left[ \abs{ H } F e^{F^2/2} \Unens{0 \leq F \leq z} \right] & = \sum_{j=1}^{[z]} \mathbb{E} \left[ \abs{ H } F e^{F^2/2} \Unens{j-1 \leq F \leq j} \right] + \mathbb{E} \left[ \abs{ H } F e^{F^2/2} \Unens{[z] \leq F \leq z} \right].
	\end{align*}
	For the first term we get
	\begin{align*}
		\sum_{j=1}^{[z]} \mathbb{E} \left[ \abs{ H } F e^{F^2/2} \Unens{j-1 \leq F \leq j} \right]  & \leq \sum_{j=1}^{[z]} j e^{(j-1)^2/2 - j(j-1)} \mathbb{E} \left[ \abs{ H }  e^{jF} \Unens{j-1 \leq F \leq j} \right] \\
		& \leq 3 \sum_{j=1}^{[z]} j e^{-j^2/2} \mathbb{E} \left[ \abs{ H }  e^{jF} \Unens{j-1 \leq F \leq j} \right]
	\end{align*}
	and similarly, for the second
	\begin{align*}
		\mathbb{E} \left[ \abs{ H } F e^{F^2/2} \Unens{[z] \leq F \leq z} \right] & \leq z e^{[z]^2/2 - [z]z} \mathbb{E} \left[ \abs{ H }  e^{zF} \Unens{[z] \leq F \leq z} \right] \\
		& \leq 3 z e^{-z^2/2} \mathbb{E} \left[ \abs{ H }  e^{zF} \Unens{[z] \leq F \leq z} \right].
	\end{align*}
	For both terms, we used similar manipulations, namely for $j-1 \leq F \leq j$:
	\begin{itemize}
		\item $ e^{(j-1)^2/2 - j(j-1)} = e^{j^2/2 - j + 1/2 - j^2 + j} = e^{-j^2/2} e^{1/2} \leq 3 e^{-j^2/2}$.
		\item $ e^{(j-F)^2/2} \leq e^{1/2} \Leftrightarrow e^{F^2/2} \leq e^{-j^2/2 +jF + 1/2} \Leftrightarrow e^{F^2/2} \leq e^{(j-1)^2/2 - j(j-1)} e^{jF}$.
	\end{itemize}
	And for $[z] \leq F \leq z$:
	\begin{itemize}
		\item $ e^{[z]^2/2 - [z]z} = e^{[z]^2/2 - [z]z + z^2/2 - z^2/2} = e^{(z - [z])^2/2} e^{-z^2/2} \leq 3e^{-z^2/2}$.
		\item $ e^{(z - F)^2/2} \leq e^{(z - [z])^2/2} \Leftrightarrow e^{F^2/2} \leq e^{(z - [z])^2/2 + zF - z^2/2} \Leftrightarrow e^{F^2/2} \leq e^{[z]^2/2 - [z]z}e^{zF}$.
	\end{itemize}
	By condition \eqref{A1} and \eqref{BoundMomentGeneratingFunction2}, and recalling that $\gamma_1$ is increasing, for any $0 \leq x \leq z \leq A_0(d_0)$
	\begin{align*}
		e^{-x^2/2} \mathbb{E} \left[ \abs{ 1 - \langle DF , - DL^{-1} F \rangle } e^{xF} \right] & \leq \gamma_1(x) \mathbb{E}[e^{xF - x^2/2}] \\
		& \leq e^{d_0} \gamma_1(x) \leq e^{d_0} \gamma_1(z).
	\end{align*}
	By the foregoing inequalities,
	\begin{align*}
		\mathbb{E} \left[ \abs{ H } F e^{F^2/2} \Unens{0 \leq F \leq z} \right]  & \leq 3 e^{d_0} \gamma_1(z) \prth{\sum_{j=1}^{[z]} j + z} \leq 6 e^{d_0} (1+z^2) \gamma_1(z).
	\end{align*}
	The other statement of the lemma can be shown analogously.
\end{proof}
Now we are ready to prove our two theorems.
\begin{proof}[Proof of Theorem~\ref{Theorem:MD_Rademacher}]
We note at first that
\begin{align*}
\abs{ \Pp(F > z) - (1 - \Phi(z))} & = \abs{ 1 - \Pp(F \leq z) - (1 - \Phi(z))} = \abs{ \Phi(z) - \Pp(F \leq z)}.
\end{align*}
By Stein's method and the proof of \cite[Theorem~3.1]{ERTZ22} we have for $z \in \R$
\begin{align*}
\abs{ \Pp(F > z) - (1 - \Phi(z))} & = \abs{\mathbb{E}[f_z^\prime(F) - Ff_z(F)]} \leq J_1 + J_2 
\end{align*}
with
\begin{align*}
J_1 & := \mathbb{E} \abs{f_z^\prime(F) \prth{ 1 - \langle DF , - DL^{-1} F \rangle }}, \\
J_2 & := \mathbb{E}\left[(F f_z(F) + \Unens{F>z} )\delta\prth{ \frac{1}{\sqrt{pq}} DF \abs{D L^{-1} F} }\right].
\end{align*}
For the upcoming estimation we can split $J_2$ into two terms, namely
\begin{align*}
\abs{J_2} & \leq \mathbb{E} \left[ \abs{  \delta\prth{ \frac{1}{\sqrt{pq}} DF \abs{D L^{-1} F} } } \abs{ F f_z(F) + \Unens{F>z}}\right] \leq  J_{21} + J_{22}
\end{align*}
with
\begin{align*}
J_{21} & := \mathbb{E} \left[ \abs{  \delta\prth{ \frac{1}{\sqrt{pq}} DF \abs{D L^{-1} F} }  } \abs{ F f_z(F)}\right], \\
J_{22} & := \mathbb{E} \left[ \abs{  \delta\prth{ \frac{1}{\sqrt{pq}} DF \abs{D L^{-1} F} }  } \Unens{F>z}\right].
\end{align*}
Using the same arguments as in the proof of \cite[Proposition~4.1]{Z23}, in particular \eqref{Stein_solution}, \eqref{Stein_prop_2} and  \eqref{Stein_prop_3}, we have
\begin{align*}
J_{21} \leq J_{23} + J_{24} + J_{25}
\end{align*}
with
\begin{align*}
J_{23} & := (1 - \Phi(z)) \cdot \mathbb{E} \left[ \abs{  \delta\prth{ \frac{1}{\sqrt{pq}} DF \abs{D L^{-1} F} } } \Unens{F<0}\right], \\
J_{24} & := \sqrt{2\pi} \cdot (1 - \Phi(z)) \cdot \mathbb{E} \left[ \abs{ \delta\prth{ \frac{1}{\sqrt{pq}} DF \abs{D L^{-1} F} } } F e^{F^2/2} \Unens{0 \leq F \leq z} \right], \\
J_{25} & := \mathbb{E} \left[ \abs{ \delta\prth{ \frac{1}{\sqrt{pq}} DF \abs{D L^{-1} F} } } \Unens{F>z}\right] = J_{22}.
\end{align*}
Thus,
\begin{align}\label{MDP_Rademacher1}
\abs{J_2} & \leq J_{23} + J_{24} + 2 \cdot J_{25}.
\end{align}
For $J_{23}$, by condition \eqref{A2} with $t=0$ and noting that $\gamma_2$ is increasing,
\begin{align}\label{MDP_Rademacher2}
\mathbb{E} \left[ \abs{  \delta\prth{ \frac{1}{\sqrt{pq}} DF \abs{D L^{-1} F} } } \Unens{F<0}\right] & \leq \mathbb{E} \left[ \abs{ \delta\prth{ \frac{1}{\sqrt{pq}} DF \abs{D L^{-1} F} } } \right] \nonumber \\
                   & \leq \gamma_2(0) \leq e^{d_0} \gamma_2(z).
\end{align}
For $J_{24}$, by Lemma~\ref{Lemma:TermsTimesFfF}, we have
\begin{align}\label{MDP_Rademacher3}
\mathbb{E} \left[ \abs{ \delta\prth{ \frac{1}{\sqrt{pq}} DF \abs{D L^{-1} F} } } F e^{F^2/2} \Unens{0 \leq F \leq z} \right] \leq 6 e^{d_0} (1+z^2) \gamma_2(z).
\end{align}
For $J_{25}$, by condition \eqref{A2} and \eqref{BoundMomentGeneratingFunction2}, for $0 \leq z \leq A_0(d_0)$,
\begin{align*}
J_{25} & = \mathbb{E} \left[ \abs{ \delta\prth{ \frac{1}{\sqrt{pq}} DF \abs{D L^{-1} F} } } \Unens{F>z} e^{zF} e^{-zF} \right] \nonumber \\
            & \leq \mathbb{E} \left[ \abs{ \delta\prth{ \frac{1}{\sqrt{pq}} DF \abs{D L^{-1} F} } } \Unens{F>z} e^{zF} \right] e^{-z^2} \nonumber \\
            & \leq  \gamma_2(z) \mathbb{E} \left[ e^{zF}\right] e^{-z^2} \nonumber \\
            & \leq e^{d_0} \gamma_2(z) e^{-z^2/2}.
\end{align*}
We recall that for $z>0$
\begin{align*}
e^{-z^2/2} & \leq \sqrt{2\pi} \cdot (1+z) \cdot (1 - \Phi(z)) \leq \frac{3\sqrt{2\pi}}{2} \cdot (1+z^2) \cdot (1 - \Phi(z)).
\end{align*}
Then, for $0 \leq z \leq A_0(d_0)$,
\begin{align}\label{MDP_Rademacher4}
\mathbb{E} \left[ \abs{ \delta\prth{ \frac{1}{\sqrt{pq}} DF \abs{D L^{-1} F} } } \Unens{F>z}\right] & \leq \frac{3 e^{d_0}\sqrt{2\pi}}{2} (1+z^2) \gamma_2(z) (1 - \Phi(z)).
\end{align}
Therefore, combining ``\eqref{MDP_Rademacher1} -- \eqref{MDP_Rademacher4}'', for $0 \leq z \leq A_0(d_0)$, we have
\begin{align*}
\abs{J_2} & \leq \prth{1 + 6\sqrt{2\pi} + 3\sqrt{2\pi}} e^{d_0} (1+z^2) \gamma_2(z) (1 - \Phi(z)) \leq 25 e^{d_0} (1+z^2) \gamma_2(z) (1 - \Phi(z)).
\end{align*}
For the remaining term $J_1$ we have a similar approach after using again Stein's equation:
\begin{align*}
\abs{J_1} & =  \mathbb{E} \abs{f_z^\prime(F) \prth{ 1 - \langle DF , - DL^{-1} F \rangle }} \\
                 & \leq  \mathbb{E} \left[ \abs{f_z^\prime(F)} \abs{ 1 - \langle DF , - DL^{-1} F \rangle } \right] \\
                 & \leq J_{11} + J_{12} + J_{13}
\end{align*}
with
\begin{align*}
J_{11} & := \mathbb{E} \left[ \abs{F f_z(F)} \abs{ 1 - \langle DF , - DL^{-1} F \rangle } \right], \\
J_{12} & := \mathbb{E} \left[ (1 - \Phi(z)) \abs{ 1 - \langle DF , - DL^{-1} F \rangle } \right], \\
J_{13} & := \mathbb{E} \left[ \Unens{F>z} \abs{ 1 - \langle DF , - DL^{-1} F \rangle } \right].
\end{align*}
From here on we can identify any of these terms with a corresponding term from the first part of the proof, namely $J_{21} - J_{25}$. Therefore, combining these modified estimations, for $0 \leq z \leq A_0(d_0)$, we have
\begin{align*}
\abs{J_1} & \leq \prth{1 + 1 + 6\sqrt{2\pi} + 3\sqrt{2\pi}} e^{d_0} (1+z^2) \gamma_1(z) (1 - \Phi(z)) \leq 25 e^{d_0} (1+z^2) \gamma_1(z) (1 - \Phi(z)).
\end{align*}
All in all, we have shown, for $0 \leq z \leq A_0(d_0)$,
\begin{align*}
\abs{ \Pp(F > z) - (1 - \Phi(z))} & \leq 25 e^{d_0} (1+z^2) (\gamma_1(z) + \gamma_2(z)) (1 - \Phi(z))
\end{align*}
or equivalently
\begin{align*}
\abs{ \frac{\Pp(F > z)}{1 - \Phi(z)} - 1} & \leq 25 e^{d_0} (1+z^2) (\gamma_1(z) + \gamma_2(z)).
\qedhere
\end{align*}
\end{proof}

\begin{proof}[Proof of Theorem~\ref{Theorem:MD_Rademacher-short}]
	Let $0 \leq z_0 \leq A$ be fixed.
	Choose $d_0 = \frac{z_0^2}{2}\prth{\gamma_1(z_0) + \gamma_2(z_0)}$.
	Per definition there is $0 \leq z_0 \leq A_0(d_0)$.
	Hence, we may apply Theorem~\ref{Theorem:MD_Rademacher}, which then implies
	\begin{align*}
		\abs{ \frac{\Pp(F > z_0)}{1 - \Phi(z_0)} - 1}
		&\leq 25 e^{d_0} (1+z_0^2) \prth{\gamma_1(z_0) + \gamma_2(z_0)} \\
		&= 25 e^{\frac{z_0^2}{2}\prth{\gamma_1(z_0) + \gamma_2(z_0)}}
			(1+z_0^2) \prth{\gamma_1(z_0) + \gamma_2(z_0)}.
			\qedhere
	\end{align*}
\end{proof}
%
%
%
%
\section{Proofs I: Infinite weighted 2-runs}
\begin{proof}[Proof of Theorem~\ref{Theorem:MDP_Tworuns}] In what follows, we show that all the assumptions of Theorem~\ref{Theorem:MD_Rademacher} are verified. Note that although here the Rademacher random variables are indexed by $\Zz$ instead of $\N$, Theorem~\ref{Theorem:MD_Rademacher} can be fully carried to this setting. Since the coefficient sequence $(a^{(n)}_i)_{i \in \Zz}$ is in $l^1(\Zz)$ we have $\norm{a^{(n)}}_{l^p(\Zz)} < \infty \, \forall \, p \in \N$. By definition $\mathbb{E}[F] = 0$ and $\Var(F) = 1$, and the rewritten random variable 
\begin{align*}
F := F_n = \frac{1}{\sqrt{\Var(G_n)}} \sum_{i \in \Zz} a^{(n)}_i \left[\xi_i \xi_{i+1} - \frac{1}{4}\right] = \frac{1}{4 \sqrt{\Var(G_n)}} \sum_{i \in \Zz} a^{(n)}_i \left[X_i  + X_i X_{i+1} + X_{i+1}\right]
\end{align*}
is bounded. In particular we will use $\mathbb{E}\abs{G_n} \leq \norm{a^{(n)}}_{l^1(\Zz)}$ and
\begin{align}\label{two_runs_variance}
\frac{1}{16} \norm{a^{(n)}}_{l^2(\Zz)}^2 \leq \Var(G_n) = \frac{3}{16} \sum_{i \in \Zz} (a^{(n)}_i)^2 + \frac{1}{8} \sum_{i \in \Zz} a^{(n)}_i a^{(n)}_{i+1} \leq \frac{5}{16} \norm{a^{(n)}}_{l^2(\Zz)}^2.
\end{align}
Mostly with the summability of $(a^{(n)}_i)_{i \in \Zz}$ , $\abs{e^x - 1} \leq \abs{x} e^{\abs{x}} \, \forall \, x \in \R$ and equation (12.2) in \cite{Pr08} we see that $F \in \D^{1,2}$ and $e^{tF} \in \D^{1,2} \, \forall \, t \in \R$.
Regarding the assumptions that 
$\frac{1}{\sqrt{pq}} DF \abs{D L^{-1} F}  \in Dom(\delta)$ and
$F f_z(F) + \Unens{F>z} \in \D^{1,2} \, \forall \, z \in \R$,
we follow the argumentation of \cite{ERTZ22}, see in particular Remark~3.5 in there.
$F$ is an element of the sum of the first and second Rademacher chaos,
see the beginning of the proof of Theorem~1.1 in \cite{ERTZ22},
and by hypercontractivity we find that $F \in L^4(\Omega)$.
Following the calculations in the proof of Lemma~3.7 in \cite{DK19}
with $u_k = \frac{1}{\sqrt{p_kq_k}} D_kF \abs{D_k L^{-1} F}$ for $k \in \Zz$,
it can be shown that assumption~(2.14) in Proposition~2.2 in \cite{KRT17} is satisfied.
This implies that
$u=\frac{1}{\sqrt{pq}} DF \abs{D L^{-1} F}  \in Dom(\delta)$.
Further, it implies that
$\E[ (F f_z(F) + \Unens{F>z} ) \delta(u)]
= \E[\langle D(F f_z(F) + \Unens{F>z}),u \rangle]$,
which is why we do not need to verify whether
$F f_z(F) + \Unens{F>z}$ is an element of $\D^{1,2} \, \forall \, z \in \R$.
Now we start to compute the terms appearing in \eqref{A1} and \eqref{A2}. By definition
\begin{align*}
F_{k}^+ & = \frac{1}{4 \sqrt{\Var(G_n)}} \prth{\sum_{\substack{i \in \Zz \\ i \neq k-1,k}} a^{(n)}_i \left[X_i  + X_i X_{i+1} + X_{i+1}\right] + a^{(n)}_{k-1} (2X_{k-1} + 1) + a^{(n)}_k (2X_{k+1} + 1) }, \\
F_{k}^- & = \frac{1}{4 \sqrt{\Var(G_n)}} \prth{\sum_{\substack{i \in \Zz \\ i \neq k-1,k}} a^{(n)}_i \left[X_i  + X_i X_{i+1} + X_{i+1}\right] - a^{(n)}_{k-1} - a^{(n)}_k },
\end{align*}
and we get (for fixed $n \in \N$)
\begin{align*}
D_k F & = \frac{1}{2} \prth{F_{k}^+ - F_{k}^-} \\
                                   & = \frac{1}{4 \sqrt{\Var(G_n)}} \prth{ a^{(n)}_{k-1} (X_{k-1} + 1) + a^{(n)}_k (X_{k+1} + 1)}.
\end{align*}
Further we obtain
\begin{align*}
-L^{-1} F = \frac{1}{4 \sqrt{\Var(G_n)}} \sum_{i \in \Zz} a^{(n)}_i \left[X_i  + \frac{1}{2} X_i X_{i+1} + X_{i+1}\right]
\end{align*}
and so
\begin{align*}
-D_k L^{-1} F = \frac{1}{8 \sqrt{\Var(G_n)}} \prth{ a^{(n)}_{k-1} (X_{k-1} + 2) + a^{(n)}_k (X_{k+1} + 2)}.
\end{align*}
With these expressions we can compute the scalar product
\begin{align*}
\langle DF, -DL^{-1}F \rangle =  \frac{1}{32 \Var(G_n)} \bigg(\sum_{k \in \Zz} & (a^{(n)}_{k-1})^2 \left[X_{k-1}^2  + 3 X_{k-1} + 2\right] \\
     & + (a^{(n)}_{k})^2 \left[X_{k+1}^2  + 3 X_{k+1} + 2\right] \\
     & + a^{(n)}_{k-1} a^{(n)}_{k}\left[3 X_{k-1} + 2 X_{k-1} X_{k+1} + 3 X_{k+1} + 4\right]\bigg).
\end{align*}
According to \cite[(2.13)]{KRT17} it holds that
\begin{align*}
1 = \mathbb{E}[\langle DF, -DL^{-1}F \rangle] = \frac{1}{32 \Var(G_n)} \sum_{k \in \Zz}  3 (a^{(n)}_{k-1})^2 + 4 a^{(n)}_{k-1} a^{(n)}_{k} + 3 (a^{(n)}_{k})^2.
\end{align*}
We use the Cauchy--Schwarz inequality for 
\begin{align*}
\mathbb{E} \left[ \abs{ 1 - \langle DF , - DL^{-1} F \rangle } e^{tF}\right] \leq  \prth{ \mathbb{E} \left[ \prth{ 1 - \langle DF , - DL^{-1} F \rangle }^2 e^{tF}\right]}^\frac{1}{2} \prth{\mathbb{E} \left[ e^{tF}\right]}^\frac{1}{2}
\end{align*}
and deal with the double sum resulting from the square of
\begin{align*}
\langle DF , - DL^{-1} F \rangle - 1  =  \frac{1}{32 \Var(G_n)} \bigg(\sum_{k \in \Zz} & (a^{(n)}_{k-1})^2 3 X_{k-1} + (a^{(n)}_{k})^2 3 X_{k+1} \\
     & + a^{(n)}_{k-1} a^{(n)}_{k}\left[3 X_{k-1} + 2 X_{k-1} X_{k+1} + 3 X_{k+1} \right]\bigg).
\end{align*}
Then we can write $\prth{ 1 - \langle DF , - DL^{-1} F \rangle }^2 = B_1 + \ldots + B_9$ with
\begin{align*}
B_1 & = \frac{9}{1024 (\Var(G_n))^2} \sum_{k,l \in \Zz} (a^{(n)}_{k-1})^2 (a^{(n)}_{l-1})^2 X_{k-1} X_{l-1}, \\
B_2 & = \frac{9}{1024 (\Var(G_n))^2} \sum_{k,l \in \Zz} (a^{(n)}_{k})^2 (a^{(n)}_{l})^2 X_{k} X_{l}, \\
B_3 & = \frac{9}{1024 (\Var(G_n))^2} \sum_{k,l \in \Zz} (a^{(n)}_{k-1})^2 (a^{(n)}_{l})^2 X_{k-1} X_{l}, \\
B_4 & = \frac{9}{1024 (\Var(G_n))^2} \sum_{k,l \in \Zz} (a^{(n)}_{k})^2 (a^{(n)}_{l-1})^2 X_{k} X_{l-1}, \\
B_5 & = \frac{3}{1024 (\Var(G_n))^2} \sum_{k,l \in \Zz} (a^{(n)}_{k-1})^2 (a^{(n)}_{l-1}) (a^{(n)}_{l}) X_{k-1} \left[3 X_{l-1} + 2 X_{l-1} X_{l+1} + 3 X_{l+1}, \right] \\
B_6 & = \frac{3}{1024 (\Var(G_n))^2} \sum_{k,l \in \Zz} (a^{(n)}_{k})^2 (a^{(n)}_{l-1}) (a^{(n)}_{l} )X_{k} \left[3 X_{l-1} + 2 X_{l-1} X_{l+1} + 3 X_{l+1} \right], \\
B_7 & = \frac{3}{1024 (\Var(G_n))^2} \sum_{k,l \in \Zz} (a^{(n)}_{k-1}) (a^{(n)}_{k}) (a^{(n)}_{l-1})^2 X_{l-1} \left[3 X_{k-1} + 2 X_{k-1} X_{k+1} + 3 X_{k+1} \right], \\
B_8 & = \frac{3}{1024 (\Var(G_n))^2} \sum_{k,l \in \Zz} (a^{(n)}_{k-1}) (a^{(n)}_{k}) (a^{(n)}_{l})^2 X_{l} \left[3 X_{k-1} + 2 X_{k-1} X_{k+1} + 3 X_{k+1} \right],
\end{align*}
where $B_1 = B_2 = B_3 = B_4$ and $B_5 = B_6 = B_7 = B_8$ by symmetry and change of variables. The last missing term is given by
\begin{align*}
B_9 = \frac{1}{1024 (\Var(G_n))^2} \bigg(\sum_{k,l \in \Zz} (a^{(n)}_{k-1}) (a^{(n)}_{k}) (a^{(n)}_{l-1}) (a^{(n)}_{l}) & \cdot \left[3 X_{k-1} + 2 X_{k-1} X_{k+1} + 3 X_{k+1} \right] \\
& \cdot \left[3 X_{l-1} + 2 X_{l-1} X_{l+1} + 3 X_{l+1} \right]\bigg).
\end{align*}
So basically we have to deal with three classes of subterms in total. Since they will be multiplied with $e^{tF}$, we have to study
\begin{align*}
\mathbb{E}[X_i e^{tF}], \quad \mathbb{E}[X_i X_j e^{tF}], \quad \mathbb{E}[X_i X_j X_k e^{tF}], \quad \mathbb{E}[X_i X_j X_k X_l e^{tF}]
\end{align*}
for $i \neq j \neq k \neq l$ --- if two or more indices are equal, it is just one of the terms from before or immediately $\mathbb{E}[e^{tF}]$.  This is done in the following lemma and we will refer to it, in particular the inequalities shown in its proof.
%
%
%
%
%
%
\begin{lemma}\label{Lemma:XsTimesMGF}
	In the setting of Theorem~\ref{Theorem:MDP_Tworuns} we have for $i \neq j \neq k \neq l \in \Zz$
	\begin{align}\label{1X_Lemma}
		\abs{\mathbb{E}[X_i e^{tF}]}  \leq & \; \; \; \frac{C t}{4 \sqrt{\Var(G_n)}} \cdot \prth{\abs{a^{(n)}_{i-1}} + \abs{a^{(n)}_{i}}} \cdot \mathbb{E}\left[  e^{tF}\right] e^{ct} \nonumber \\
		& + \frac{C t^2}{16 \Var(G_n)} \cdot \sum_{\substack{m_1 \in \{i-1,i\} \\ n_1 \in \{i-2,i+1\}}} \abs{a^{(n)}_{m_1}} \abs{a^{(n)}_{n_1}} \cdot \mathbb{E}\left[  e^{tF}\right] e^{ct} \nonumber \\
		& + \frac{C t^2}{32 \Var(G_n)} \cdot \sum_{\substack{m_2,n_2 \in \{i-1,i\}}} \abs{a^{(n)}_{m_2}} \abs{a^{(n)}_{n_2}} \cdot \mathbb{E}\left[  e^{tF}\right] e^{ct}.
	\end{align}
	\begin{align}\label{2X_Lemma}
		\abs{\mathbb{E}[X_i X_j e^{tF}]}  \leq & \; \; \; \frac{C t}{4 \sqrt{\Var(G_n)}} \cdot \abs{a^{(n)}_{\min(i,j)}} \cdot \mathbb{E}\left[  e^{tF}\right] e^{ct} \Unens{\abs{i-j}=1}\nonumber\\
		& + \frac{C t^2}{16 \Var(G_n)} \cdot \sum_{\substack{m_1 \in \{i-1,i,j-1\} \\ n_1 \in \{i-2,i+1,j-2,j+1\}}} \abs{a^{(n)}_{m_1}} \abs{a^{(n)}_{n_1}} \cdot \mathbb{E}\left[  e^{tF}\right] e^{ct} \nonumber \\
		& + \frac{C t^2}{32 \Var(G_n)} \cdot \sum_{\substack{m_2,n_2 \in \{i-1,i,j-1,j\}}} \abs{a^{(n)}_{m_2}} \abs{a^{(n)}_{n_2}} \cdot \mathbb{E}\left[  e^{tF}\right] e^{ct}.
	\end{align}
	\begin{align}\label{3X_Lemma}
		\abs{\mathbb{E}[X_i X_j X_k e^{tF}]}  \leq & \; \; \;  \; \frac{C t^2}{16 \Var(G_n)} \cdot \sum_{\substack{m_1 \in \{i-1,i,j-1,j,k-1,k\} \\ n_1 \in \{i-2,i+1,j-2,j+1,k-2,k+1\}}} \abs{a^{(n)}_{m_1}} \abs{a^{(n)}_{n_1}} \cdot \mathbb{E}\left[  e^{tF}\right] e^{ct} \nonumber \\
		& + \frac{C t^2}{32 \Var(G_n)} \cdot \sum_{\substack{m_2,n_2 \in \{i-1,i,j-1,j,k-1,k\}}} \abs{a^{(n)}_{m_2}} \abs{a^{(n)}_{n_2}} \cdot \mathbb{E}\left[  e^{tF}\right] e^{ct}.
	\end{align}
	\begin{align}\label{4X_Lemma}
		\abs{\mathbb{E}[X_i X_j X_k X_l e^{tF}]}  \leq & \; \; \;  \; \frac{C t^2}{16 \Var(G_n)} \cdot \qquad \qquad \sum_{\mathclap{\substack{m_1 \in \{i-1,i,j-1,j,k-1,k,l-1,l\} \\ n_1 \in \{i-2,i+1,j-2,j+1,k-2,k+1,l-2,l+1\}}}} \qquad \qquad \abs{a^{(n)}_{m_1}} \abs{a^{(n)}_{n_1}} \cdot \mathbb{E}\left[  e^{tF}\right] e^{ct} \nonumber \\
		& + \frac{C t^2}{32 \Var(G_n)} \cdot \qquad \qquad \sum_{\mathclap{\substack{m_2,n_2 \in \{i-1,i,j-1,j,k-1,k,l-1,l\}}}} \qquad \quad \abs{a^{(n)}_{m_2}} \abs{a^{(n)}_{n_2}} \cdot \mathbb{E}\left[  e^{tF}\right] e^{ct}.
	\end{align}
\end{lemma}
\begin{proof}
	The first key element of our strategy is to split $F$ into $F_a$, the summands that depend on the $X$'s multiplied with $e^{tF}$, and $F_u$, the summands that are independent. We should have in mind that $F_a$ and $F_u$ are not necessarily independent from each other. To get this dependency structure under control
	we will make use of several Taylor expansions of the exponential.
	Note that there are remainder functions $r_1,r_2 : \R \rightarrow \R$
	such that
	$e^x = 1 + x \cdot r_1(x) = 1 + x + \frac{x^2}{2} \cdot r_2(x)$
	with $\vert r_1 (x) \vert , \vert r_2 (x) \vert \leq e^{\max\{ 0, x \}} \leq e^{\vert x \vert}$
	for all $x \in \R$. 
	So, the second key element is  an iterated Taylor expansion on $e^{tF}$ according to the following scheme:
	For a finite index set $I$, let there be real numbers $(x_j)_{j\in I}$, $(y_j)_{j\in I}$ and $z$.
	Then by iterated Taylor expansion there is
	\begin{align}\label{ITE-formula}
		e^{z + \sum_{j\in I} x_j}
		&=
		1 \cdot e^{z} +
		\sum_{j\in I} x_j \cdot e^{z} +
		\frac12 \Big( \sum_{j\in I} x_j \Big)^2 r_2 \Big( \sum_{j\in I} x_j \Big) \cdot e^{z} \nonumber\\
		&=
		1 \cdot e^{z} +
		\sum_{j\in I} x_j \cdot 1 \cdot e^{z-y_j} +
		\sum_{j\in I} x_j \cdot y_j r_1(y_j) \cdot e^{z-y_j} +
		\frac12 \Big( \sum_{j\in I} x_j \Big)^2 r_2 \Big( \sum_{j\in I} x_j \Big) \cdot e^{z}.
	\end{align}
	We remind on the short notation $A_k := a^{(n)}_k \left[X_k  + X_k X_{k+1} + X_{k+1}\right]$ for the upcoming computations. In the case of $\mathbb{E}[X_i e^{tF}]$:
	\begin{align*}
		F_a = \frac{1}{4 \sqrt{\Var(G_n)}} \prth{A_{i-1} + A_i}
	\end{align*}
	and by independence
	\begin{align*}
		\mathbb{E}[X_i e^{tF}] = \mathbb{E}[X_i e^{tF_a} e^{tF_u}] & = \mathbb{E}[X_i e^{tF_u}] + t \mathbb{E}[X_i F_a e^{tF_u}] + t^2 \mathbb{E}[X_i F_a^2r_2(t F_a) e^{tF_u}/ 2] \\
		& = t \mathbb{E}[X_i F_a e^{tF_u}] + t^2 \mathbb{E}[X_i F_a^2 r_2(t F_a) e^{tF_u}/ 2],
	\end{align*}
	where we chose $z = F_u$ and $\sum_{j\in I} x_j  = F_a$ --- note that $I$ will increase with every case since the number of multiplied $X$'s increases. For the first order term we split $F_u$ in the same manner as before, $F_u = F_{u_a} + F_{u_u}$, such that $F_{u_a} = \prth{A_{i-2} + A_{i+1}} /4 \sqrt{\Var(G_n)}$ and use the iteration from \eqref{ITE-formula}. Then
	\begin{align*}
		X_i F_a = \frac{1}{4 \sqrt{\Var(G_n)}} \prth{a^{(n)}_{i-1} (X_{i-1} X_i + X_{i-1} + 1) + a^{(n)}_i (1 + X_{i+1} + X_i X_{i+1})}
	\end{align*}
	and
	\begin{align*}
		t \mathbb{E}[X_i F_a e^{tF_u}] & = t \mathbb{E}\left[ X_i F_a e^{tF_{u_a}} e^{tF_{u_u}}\right] \\
		& = t \mathbb{E}\left[ X_i F_a e^{tF_{u_u}}\right] + t^2 \mathbb{E}\left[ X_i F_a F_{u_a} r_1(t F_{u_a}) e^{tF_{u_u}}\right],
	\end{align*}
	so $ y_j = F_{u_a}$. From here on we get $e^{tF}$ back by bounding the difference of the independent part and $F$, e.g. 
	\begin{align*}
		\abs{F_{u_u} - F} = \abs{F_{u_a} + F_a} \leq c = O\prth{\frac{1}{\sqrt{\Var(G_n)}}} \rightarrow 0 \; \text{as} \; n \rightarrow \infty. 
	\end{align*}
	Since the exact constant is not important, we always write just $c$ if we use that type of estimation, and in the same way $C$ for prefactors. Thus
	\begin{align}\label{1X_firstorder}
		t \abs{\mathbb{E}[X_i F_a e^{tF_u}]}  \leq & \frac{C t}{4 \sqrt{\Var(G_n)}} \cdot \prth{\abs{a^{(n)}_{i-1}} + \abs{a^{(n)}_{i}}} \cdot \mathbb{E}\left[  e^{tF}\right] e^{ct} \nonumber \\
		& + \frac{C t^2}{16 \Var(G_n)} \cdot \sum_{\substack{m_1 \in \{i-1,i\} \\ n_1 \in \{i-2,i+1\}}} \abs{a^{(n)}_{m_1}} \abs{a^{(n)}_{n_1}} \cdot \mathbb{E}\left[  e^{tF}\right] e^{ct}.
	\end{align}
	For the second order term we just bound
	\begin{align}\label{1X_secondorder}
		t^2 \abs{\mathbb{E}[X_i F_a^2 r_2(t F_a) e^{tF_u}/ 2]} & \leq \frac{C t^2}{32 \Var(G_n)} \cdot \sum_{\substack{m_2,n_2 \in \{i-1,i\}}} \abs{a^{(n)}_{m_2}} \abs{a^{(n)}_{n_2}} \cdot \mathbb{E}\left[  e^{tF}\right] e^{ct}.
	\end{align}
	In the case of $\mathbb{E}[X_i X_j e^{tF}]$:
	\begin{align*}
		F_a  = \begin{cases}
			\frac{1}{4 \sqrt{\Var(G_n)}} \prth{A_i + A_{j-1} + A_j},  & i = j + 1, \\
			\frac{1}{4 \sqrt{\Var(G_n)}} \prth{A_{i-1} + A_i + A_{j-1} + A_j}, & \abs{i-j} \geq 2, \\
			\frac{1}{4 \sqrt{\Var(G_n)}} \prth{A_{i-1} + A_i + A_j},  & j = i + 1,
		\end{cases}
	\end{align*}
	and
	\begin{align*}
		\mathbb{E}[X_i X_j e^{tF}] & = t \mathbb{E}[X_i X_j F_a e^{tF_u}] + t^2 \mathbb{E}[X_i X_j F_a^2 r_2(t F_a) e^{tF_u}/2].
	\end{align*}
	For the first order term we compute as a preparation
	\begin{align*}
		X_i X_j F_a = a^{(n)}_{i-1} (X_{i-1} X_i X_j + X_{i-1} X_j + X_j) & + a^{(n)}_{i} (X_j + X_{i+1} X_j + X_i X_{i+1} X_j) \\
		& + a^{(n)}_{j-1} (X_{i} X_{j-1} X_j + X_i X_{j-1} + X_i) \\
		& + a^{(n)}_{j} (X_{i} + X_i X_{j+1} + X_i X_j X_{j+1}).
	\end{align*}
	In particular we have to consider the special case $\abs{i-j} = 1$ and assume $i = j + 1$. If not, we just have to swap $i$ and $j$. Under our assumption the last equation reduces to
	\begin{align*}
		X_i X_j F_a = a^{(n)}_{i-1} (X_i + 1 + X_j) & + a^{(n)}_{i} (X_j + X_{i+1} X_j + X_i X_{i+1} X_j) \\
		& + a^{(n)}_{j-1} (X_{i} X_{j-1} X_j + X_i X_{j-1} + X_i).
	\end{align*}
	From here on we assume that the indices apperaring in upcoming $F_a$'s and $F_{u_a}$'s are all different. If not, there is only an effect on the number of coefficients and so the constants, but not on the order of our bound. Having that in mind we split $F_u$ in the same manner as before, $F_u = F_{u_a} + F_{u_u}$, such that $F_{u_a} = \prth{A_{i-2} + A_{i+1} + A_{j-2} + A_{j+1}} /4 \sqrt{\Var(G_n)}$. Then
	\begin{align*}
		t \mathbb{E}[X_i X_j F_a e^{tF_u}] 
		& = t \mathbb{E}[X_i X_j F_a e^{tF_{u_u}}] + t^2 \mathbb{E}[X_i X_j F_a F_{u_a} r_1(t F_{u_a}) e^{tF_{u_u}}]
	\end{align*}
	and thus for $i= j + 1$
	\begin{align}\label{2X_firstorder_1}
		t \abs{\mathbb{E}[X_i X_j F_a e^{tF_u}]}  \leq & \frac{C t}{4 \sqrt{\Var(G_n)}} \cdot \abs{a^{(n)}_{i-1}} \cdot \mathbb{E}\left[  e^{tF}\right] e^{ct} \nonumber\\
		& + \frac{C t^2}{16 \Var(G_n)} \cdot \sum_{\substack{m_1 \in \{i-1,i,j-1\} \\ n_1 \in \{i-2,i+1,j-2,j+1\}}} \abs{a^{(n)}_{m_1}} \abs{a^{(n)}_{n_1}} \cdot \mathbb{E}\left[  e^{tF}\right] e^{ct}.
	\end{align}
	In the case $\abs{i-j} \geq 2$ the first term of the last inequality does not appear since all the indices in $X_i X_j F_a$ are different:
	\begin{align}\label{2X_firstorder_2}
		t \abs{\mathbb{E}[X_i X_j F_a e^{tF_u}]}  \leq \frac{C t^2}{16 \Var(G_n)} \cdot \sum_{\substack{m_1 \in \{i-1,i,j-1,j\} \\ n_1 \in \{i-2,i+1,j-2,j+1\}}} \abs{a^{(n)}_{m_1}} \abs{a^{(n)}_{n_1}} \cdot \mathbb{E}\left[  e^{tF}\right] e^{ct}.
	\end{align}
	For the second order term we just bound
	\begin{align} \label{2X_secondorder}
		t^2 \abs{\mathbb{E}[X_i X_j F_a^2 r_2(t F_a) e^{tF_u}/2]} & \leq \frac{C t^2}{32 \Var(G_n)} \cdot \sum_{\substack{m_2,n_2 \in \{i-1,i,j-1,j\}}} \abs{a^{(n)}_{m_2}} \abs{a^{(n)}_{n_2}} \cdot \mathbb{E}\left[  e^{tF}\right] e^{ct}.
	\end{align}
	In the case of $\mathbb{E}[X_i X_j X_k e^{tF}]$:
	\begin{align*}
		F_a = \frac{1}{4 \sqrt{\Var(G_n)}} \prth{A_{i-1} + A_i + A_{j-1} + A_j + A_{k-1} + A_k}
	\end{align*}
	and
	\begin{align*}
		\mathbb{E}[X_i X_j X_k e^{tF}] & = t \mathbb{E}[X_i X_j X_k F_a e^{tF_u}] + t^2 \mathbb{E}[X_i X_j X_k F_a^2 r_2(t F_a) e^{tF_u}/2].
	\end{align*}
	For the first order term we compute as a preparation
	\begin{align*}
		X_i X_j X_k F_a = & a^{(n)}_{i-1} (X_{i-1} X_i X_j X_k + X_{i-1} X_j X_k+ X_j X_k) \\
		& + a^{(n)}_{i} (X_j X_k + X_{i+1} X_j X_k + X_i X_{i+1} X_j X_k) \\
		& + a^{(n)}_{j-1} (X_{i} X_{j-1} X_j X_k + X_i X_{j-1} X_k + X_i X_k) \\
		& + a^{(n)}_{j} (X_{i} X_k + X_i X_{j+1} X_k + X_i X_j X_{j+1} X_k) \\
		& + a^{(n)}_{k-1} (X_{i} X_j X_{k-1} X_k + X_i X_j X_{k-1} + X_i X_j) \\
		& + a^{(n)}_{k} (X_{i} X_j + X_i X_j X_{k+1} + X_i X_j X_k X_{k+1}).
	\end{align*}
	And by our assumption $i \neq j \neq k$ in every summand at least one $X$ will remain. We split $F_u = F_{u_a} + F_{u_u}$, such that $F_{u_a} = \prth{A_{i-2} + A_{i+1} + A_{j-2} + A_{j+1} + A_{k-2} + A_{k+1}} /4 \sqrt{\Var(G_n)}$. Then
	\begin{align*}
		t \mathbb{E}[X_i X_j X_k F_a e^{tF_u}] & = t \mathbb{E}[X_i X_j X_k F_a e^{tF_{u_a}} e^{tF_{u_u}}] \\
		& = t \mathbb{E}[X_i X_j X_k F_a e^{tF_{u_u}}] + t^2 \mathbb{E}[X_i X_j X_k F_a F_{u_a} r_1(t F_{u_a}) e^{tF_{u_u}}] \\
		& = t^2 \mathbb{E}[X_i X_j X_k F_a F_{u_a} r_1(t F_{u_a}) e^{tF_{u_u}}] \\
	\end{align*}
	by independence and thus
	\begin{align}\label{3X_firstorder}
		t \abs{\mathbb{E}[X_i X_j X_k F_a e^{tF_u}]}  \leq \frac{C t^2}{16 \Var(G_n)} \cdot \sum_{\substack{m_1 \in \{i-1,i,j-1,j,k-1,k\} \\ n_1 \in \{i-2,i+1,j-2,j+1,k-2,k+1\}}} \abs{a^{(n)}_{m_1}} \abs{a^{(n)}_{n_1}} \cdot \mathbb{E}\left[  e^{tF}\right] e^{ct}.
	\end{align}
	For the second order term we just bound
	\begin{align} \label{3X_secondorder}
		t^2 \abs{\mathbb{E}[X_i X_j X_k F_a^2 r_2(t F_a) e^{tF_u}/2]} & \leq \frac{C t^2}{32 \Var(G_n)} \cdot \sum_{\substack{m_2,n_2 \in \{i-1,i,j-1,j,k-1,k\}}} \abs{a^{(n)}_{m_2}} \abs{a^{(n)}_{n_2}} \cdot \mathbb{E}\left[  e^{tF}\right] e^{ct}.
	\end{align}
	In the case of $\mathbb{E}[X_i X_j X_k X_l e^{tF}]$:
	\begin{align*}
		F_a = \frac{1}{4 \sqrt{\Var(G_n)}} \prth{A_{i-1} + A_i + A_{j-1} + A_j + A_{k-1} + A_k + A_{l-1} + A_l}
	\end{align*}
	and
	\begin{align*}
		\mathbb{E}[X_i X_j X_k X_l e^{tF}] & = t \mathbb{E}[X_i X_j X_k X_l F_a e^{tF_u}] + t^2 \mathbb{E}[X_i X_j X_k X_l F_a^2 r_2(t F_a) e^{tF_u}/2].
	\end{align*}
	For the first order term we compute as a preparation
	\begin{align*}
		X_i X_j X_k X_l F_a = & a^{(n)}_{i-1} (X_{i-1} X_i X_j X_k X_l + X_{i-1} X_j X_k X_l+ X_j X_k X_l) \\
		& + a^{(n)}_{i} (X_j X_k X_l + X_{i+1} X_j X_k X_l + X_i X_{i+1} X_j X_k X_l) \\
		& + a^{(n)}_{j-1} (X_{i} X_{j-1} X_j X_k X_l + X_i X_{j-1} X_k X_l + X_i X_k X_l) \\
		& + a^{(n)}_{j} (X_{i} X_k X_l + X_i X_{j+1} X_k X_l + X_i X_j X_{j+1} X_k X_l) \\
		& + a^{(n)}_{k-1} (X_{i} X_j X_{k-1} X_k X_l + X_i X_j X_{k-1} X_l + X_i X_j X_l) \\
		& + a^{(n)}_{k} (X_{i} X_j X_l + X_i X_j X_{k+1} X_l + X_i X_j X_k X_{k+1} X_l) \\
		& + a^{(n)}_{l-1} (X_{i} X_j X_k X_{l-1} X_l + X_i X_j X_k X_{l-1} + X_i X_j X_k) \\
		& + a^{(n)}_{k} (X_{i} X_j X_k + X_i X_j X_k X_{l+1} + X_i X_j X_k X_l X_{l+1}) .
	\end{align*}
	And by our assumption $i \neq j \neq k \neq l$ in every summand at least one $X$ will remain. $F_{u_a}$ is given by $F_{u_a} = \prth{A_{i-2} + A_{i+1} + A_{j-2} + A_{j+1} + A_{k-2} + A_{k+1} + A_{l-2} + A_{l+1}} /4 \sqrt{\Var(G_n)}$ this time. Then
	\begin{align*}
		t \mathbb{E}[X_i X_j X_k X_l F_a e^{tF_u}] & = t^2 \mathbb{E}[X_i X_j X_k X_l F_a F_{u_a} r_1(t F_{u_a}) e^{tF_{u_u}}] \\
	\end{align*}
	by independence and thus
	\begin{align}\label{4X_firstorder}
		t \abs{\mathbb{E}[X_i X_j X_k X_l F_a e^{tF_u}]}  \leq \frac{C t^2}{16 \Var(G_n)} \cdot \qquad \qquad \sum_{\mathclap{\substack{m_1 \in \{i-1,i,j-1,j,k-1,k,l-1,l\} \\ n_1 \in \{i-2,i+1,j-2,j+1,k-2,k+1,l-2,l+1\}}}} \qquad \qquad \abs{a^{(n)}_{m_1}} \abs{a^{(n)}_{n_1}} \cdot \mathbb{E}\left[  e^{tF}\right] e^{ct}.
	\end{align}
	For the second order term we just bound
	\begin{align} \label{4X_secondorder}
		t^2 \abs{\mathbb{E}[X_i X_j X_k X_l F_a^2 r_2(t F_a) e^{tF_u}/2]} & \leq \frac{C t^2}{32 \Var(G_n)} \cdot \qquad \qquad \sum_{\mathclap{\substack{m_2,n_2 \in \{i-1,i,j-1,j,k-1,k,l-1,l\}}}} \qquad \quad \abs{a^{(n)}_{m_2}} \abs{a^{(n)}_{n_2}} \cdot \mathbb{E}\left[  e^{tF}\right] e^{ct}.
	\end{align}
\end{proof}
\noindent Now we are ready to deal with all three classes of subterms and choose $B_2, B_6$ and $B_9$ as representatives: 
\\ \textbf{First class of  subterms, $B_1 - B_4$:}
\begin{align} \label{B2_equality}
\mathbb{E}[B_2 e^{tF}] = \frac{9}{1024 (\Var(G_n))^2} \Bigg(\sum_{k \in \Zz} (a^{(n)}_{k})^4  \mathbb{E}[e^{tF}] & + \sum_{\substack{k,l \in \Zz \\ \abs{k-l} = 1}} (a^{(n)}_{k})^2 (a^{(n)}_{l})^2 \mathbb{E}[X_{k} X_{l}  e^{tF}] \nonumber \\
& + \sum_{\substack{k,l \in \Zz \\ \abs{k-l} \geq 2}} (a^{(n)}_{k})^2 (a^{(n)}_{l})^2 \mathbb{E}[X_{k} X_{l}  e^{tF}] \Bigg).
\end{align}
The first one is the easiest by
\begin{align*}
	B_{21} := \frac{9}{1024 (\Var(G_n))^2} \sum_{k \in \Zz} (a^{(n)}_{k})^4  \mathbb{E}[e^{tF}] =  \frac{9}{1024 (\Var(G_n))^2} \norm{a^{(n)}}_{l^4(\Zz)}^4  \mathbb{E}[e^{tF}].
\end{align*}
By using \eqref{2X_firstorder_1}, \eqref{2X_firstorder_2} and \eqref{2X_secondorder} it remains to  bound 
\begin{align*}
	B_{22} & := \frac{9}{1024 (\Var(G_n))^2} \sum_{\substack{k,l \in \Zz \\ \abs{k-l} = 1}} (a^{(n)}_{k})^2 (a^{(n)}_{l})^2 \mathbb{E}[X_{k} X_{l}  e^{tF}] \\
                       & \leq \frac{9}{1024 (\Var(G_n))^2} \prth{B_{221} + B_{222} + B_{223}} \mathbb{E}\left[  e^{tF}\right] e^{ct},
\end{align*}
\begin{align*}
	B_{23} & := \frac{9}{1024 (\Var(G_n))^2} \sum_{\substack{k,l \in \Zz \\ \abs{k-l} \geq 2}} (a^{(n)}_{k})^2 (a^{(n)}_{l})^2 \mathbb{E}[X_{k} X_{l}  e^{tF}]  \\
                       & \leq \frac{9}{1024 (\Var(G_n))^2} \prth{ B_{231} + B_{232}} \mathbb{E}\left[  e^{tF}\right] e^{ct},
\end{align*}
such that 
\begin{align*}
B_{221} & := \frac{C t}{4 \sqrt{\Var(G_n)}} \cdot \sum_{\substack{k,l \in \Zz \\ \abs{k-l} = 1}} (a^{(n)}_{k})^2 (a^{(n)}_{l})^2  \abs{a^{(n)}_{\min(k,l)}},\\
B_{222} & := \frac{C t^2}{16 \Var(G_n)} \cdot \sum_{\substack{k,l \in \Zz \\ \abs{k-l} = 1}} (a^{(n)}_{k})^2 (a^{(n)}_{l})^2 \cdot \sum_{\substack{m_1 \in \{k-1,k,l-1,l\} \\ n_1 \in \{k-2,k+1,l-2,l+1\}}} \abs{a^{(n)}_{m_1}} \abs{a^{(n)}_{n_1}}, \\
B_{223} & := \frac{C t^2}{32 \Var(G_n)} \cdot \sum_{\substack{k,l \in \Zz \\ \abs{k-l} = 1}} (a^{(n)}_{k})^2 (a^{(n)}_{l})^2 \cdot \sum_{\substack{m_2,n_2 \in \{k-1,k,l-1,l\}}} \abs{a^{(n)}_{m_2}} \abs{a^{(n)}_{n_2}}, \\
B_{231} & := \frac{C t^2}{16 \Var(G_n)} \cdot \sum_{\substack{k,l \in \Zz \\ \abs{k-l} \geq 2}} (a^{(n)}_{k})^2 (a^{(n)}_{l})^2 \cdot \sum_{\substack{m_1 \in \{k-1,k,l-1,l\} \\ n_1 \in \{k-2,k+1,l-2,l+1\}}} \abs{a^{(n)}_{m_1}} \abs{a^{(n)}_{n_1}}, \\
B_{232} & := \frac{C t^2}{32 \Var(G_n)} \cdot \sum_{\substack{k,l \in \Zz \\ \abs{k-l} \geq 2}} (a^{(n)}_{k})^2 (a^{(n)}_{l})^2 \cdot \sum_{\substack{m_2,n_2 \in \{k-1,k,l-1,l\}}} \abs{a^{(n)}_{m_2}} \abs{a^{(n)}_{n_2}}.
\end{align*}
Then by the inequality of arithmetic and geometric means, from here on AM-GM inequality
\begin{align*}
B_{221} & = \frac{C t}{4 \sqrt{\Var(G_n)}} \prth{\sum_{\substack{k \in \Zz}} (a^{(n)}_{k})^2 (a^{(n)}_{k-1})^2  \abs{a^{(n)}_{k-1}} + \sum_{\substack{k \in \Zz}} (a^{(n)}_{k})^2 (a^{(n)}_{k+1})^2  \abs{a^{(n)}_{k}}} \\
              & = \frac{C t}{4 \sqrt{\Var(G_n)}} \prth{\sum_{\substack{k \in \Zz}} \sqrt[5]{\abs{a^{(n)}_{k}}^{10} \abs{a^{(n)}_{k-1}}^{10}  \abs{a^{(n)}_{k-1}}^5} + \sum_{\substack{k \in \Zz}} \sqrt[5]{\abs{a^{(n)}_{k}}^{10} \abs{a^{(n)}_{k+1}}^{10}  \abs{a^{(n)}_{k}}^5}} \\
              & \leq \frac{C t}{20 \sqrt{\Var(G_n)}} \prth{\sum_{\substack{k \in \Zz}} 2 \abs{a^{(n)}_{k}}^{5} + 3 \abs{a^{(n)}_{k-1}}^{5} + \sum_{\substack{k \in \Zz}} 3 \abs{a^{(n)}_{k}}^{5} + 2 \abs{a^{(n)}_{k+1}}^{5}} \\
              & \leq \frac{C t}{ \sqrt{\Var(G_n)}} \norm{a^{(n)}}_{l^5(\Zz)}^5.
\end{align*}
If we look at $B_{222} - B_{232}$, we can change the order of summation since all summands are non-negative. And we become even bigger if we add the missing indices: 
\begin{align*}
B_{223} & \leq \frac{C t^2}{32 \Var(G_n)} \sum_{\substack{k,l \in \Zz}} \sum_{\substack{m_2 \in \{k-1,k,l-1,l\}}} \sum_{\substack{n_2 \in \{k-1,k,l-1,l\}}} (a^{(n)}_{k})^2 (a^{(n)}_{l})^2 \abs{a^{(n)}_{m_2}} \abs{a^{(n)}_{n_2}}.
\end{align*}
From here on we treat different cases, but every time we can use the AM-GM inequality:
\\ Case 1: $m_2 \in \{k-1,k\}$ and $n_2 \in \{k-1,k\}$:
\begin{align*}
\sum_{\substack{k \in \Zz}} \sum_{\substack{l \in \Zz}} (a^{(n)}_{k})^2 (a^{(n)}_{l})^2 \abs{a^{(n)}_{m_2}} \abs{a^{(n)}_{n_2}} & = \sum_{\substack{k \in \Zz}} (a^{(n)}_{k})^2 \abs{a^{(n)}_{m_2}} \abs{a^{(n)}_{n_2}} \sum_{\substack{l \in \Zz}} (a^{(n)}_{l})^2 \\
                               & \leq C \norm{a^{(n)}}_{l^4(\Zz)}^4 \norm{a^{(n)}}_{l^2(\Zz)}^2.
\end{align*}
\\ Case 2: $m_2 \in \{l-1,l\}$ and $n_2 \in \{l-1,l\}$:
\begin{align*}
\sum_{\substack{l \in \Zz}} \sum_{\substack{k \in \Zz}} (a^{(n)}_{k})^2 (a^{(n)}_{l})^2 \abs{a^{(n)}_{m_2}} \abs{a^{(n)}_{n_2}} & = \sum_{\substack{l \in \Zz}} (a^{(n)}_{l})^2 \abs{a^{(n)}_{m_2}} \abs{a^{(n)}_{n_2}} \sum_{\substack{k \in \Zz}} (a^{(n)}_{k})^2 \\
                               & \leq C \norm{a^{(n)}}_{l^4(\Zz)}^4 \norm{a^{(n)}}_{l^2(\Zz)}^2.
\end{align*}
\\ Case 3: $m_2 \in \{k-1,k\}$ and $n_2 \in \{l-1,l\}$:
\begin{align*}
\sum_{\substack{k \in \Zz}} \sum_{\substack{l \in \Zz}} (a^{(n)}_{k})^2 (a^{(n)}_{l})^2 \abs{a^{(n)}_{m_2}} \abs{a^{(n)}_{n_2}} & = \sum_{\substack{k \in \Zz}} (a^{(n)}_{k})^2 \abs{a^{(n)}_{m_2}} \sum_{\substack{l \in \Zz}} (a^{(n)}_{l})^2 \abs{a^{(n)}_{n_2}} \\
                               & \leq C \norm{a^{(n)}}_{l^3(\Zz)}^6.
\end{align*}
\\ Case 4: $m_2 \in \{l-1,l\}$ and $n_2 \in \{k-1,k\}$:
\begin{align*}
\sum_{\substack{l \in \Zz}} \sum_{\substack{k \in \Zz}} (a^{(n)}_{k})^2 (a^{(n)}_{l})^2 \abs{a^{(n)}_{m_2}} \abs{a^{(n)}_{n_2}} & = \sum_{\substack{l \in \Zz}} (a^{(n)}_{l})^2 \abs{a^{(n)}_{m_2}} \sum_{\substack{k \in \Zz}} (a^{(n)}_{k})^2 \abs{a^{(n)}_{n_2}} \\
                               & \leq C \norm{a^{(n)}}_{l^3(\Zz)}^6.
\end{align*}
According to \eqref{two_runs_variance} in case 1 and 2 the norm $\norm{a^{(n)}}_{l^2(\Zz)}^2$ vanishes directly with the variance in the prefactor. Summarizing for $B_{223}$:
\begin{align*}
B_{223} & \leq \frac{C t^2}{32} \norm{a^{(n)}}_{l^4(\Zz)}^4 + \frac{C t^2}{32 \Var(G_n)}  \norm{a^{(n)}}_{l^3(\Zz)}^6.
\end{align*}
Analogously we get basically bounds of the same order for $B_{222},B_{231}$ and $B_{232}$. Combining our bounds for the subterms of \eqref{B2_equality} gives us
\begin{align*}
\mathbb{E}[B_2 e^{tF}] \leq C e^{ct} \prth{t^2 \frac{ \norm{a^{(n)}}_{l^3(\Zz)}^6}{(\Var(G_n))^3} + (1 + t^2) \frac{\norm{a^{(n)}}_{l^4(\Zz)}^4}{ (\Var(G_n))^2} + t \frac{\norm{a^{(n)}}_{l^5(\Zz)}^5}{ (\Var(G_n))^{5/2}} }  \mathbb{E}[e^{tF}].
\end{align*}
\textbf{Second class of  subterms, $B_5 - B_8$:} We write $\mathbb{E}[B_6 e^{tF}] = B_{61} + B_{62} + B_{63}$ such that
\begin{align*}
	B_{61} & := \frac{9}{1024 (\Var(G_n))^2} \sum_{k,l \in \Zz} (a^{(n)}_{k})^2 (a^{(n)}_{l-1}) (a^{(n)}_{l}) \mathbb{E}[X_{k} X_{l-1}  e^{tF}], \\
	B_{62} & := \frac{6}{1024 (\Var(G_n))^2} \sum_{k,l \in \Zz} (a^{(n)}_{k})^2 (a^{(n)}_{l-1}) (a^{(n)}_{l}) \mathbb{E}[X_{k} X_{l-1} X_{l+1} e^{tF}], \\
	B_{63} & := \frac{9}{1024 (\Var(G_n))^2} \sum_{k,l \in \Zz} (a^{(n)}_{k})^2 (a^{(n)}_{l-1}) (a^{(n)}_{l}) \mathbb{E}[X_{k} X_{l+1}  e^{tF}].
\end{align*}
$B_{61}$ and $B_{63}$ are analogous to $B_{22}$ and $B_{23}$ since they have the same structure: Two coefficients with $k$-index, two coefficients with $l$-index, one $X$ with $k$-index and one $X$ with $l$-index. And with that the arguments are the same. Looking at the remaining $B_{62}$ two indices of the $X$'s are equal if and only if $k=l+1$ or $k=l-1$. In the latter case $B_{62}$ reduces to
\begin{align*}
\frac{C}{(\Var(G_n))^2} \sum_{k \in \Zz} (a^{(n)}_{k})^3 (a^{(n)}_{k+1}) \mathbb{E}[X_{k+2} e^{tF}],
\end{align*}
so we can use \eqref{1X_firstorder} and \eqref{1X_secondorder} giving us upper bounds of order 
\begin{align*}
O\prth{t \frac{\norm{a^{(n)}}_{l^5(\Zz)}^5}{ (\Var(G_n))^{5/2}}} \quad \text{and} \quad O\prth{t^2 \frac{\norm{a^{(n)}}_{l^6(\Zz)}^6}{ (\Var(G_n))^{3}}}.
\end{align*}
And the same for $k=l+1$. At last, if neither $k=l-1$ nor $k=l+1$ we can use \eqref{3X_firstorder} and \eqref{3X_secondorder} giving us upper bounds of order 
\begin{align*}
O\prth{t^2 \frac{\norm{a^{(n)}}_{l^3(\Zz)}^6}{ (\Var(G_n))^{3}}} \quad \text{and} \quad  O\prth{t^2 \frac{\norm{a^{(n)}}_{l^4(\Zz)}^4}{ (\Var(G_n))^{2}}}.
\end{align*}
Combining our bounds for $B_{61},B_{62}$ and $B_{63}$ we get
\begin{align*}
\mathbb{E}[B_6 e^{tF}] \leq C e^{ct} \prth{t^2 \frac{ \norm{a^{(n)}}_{l^3(\Zz)}^6}{(\Var(G_n))^3} + t^2 \frac{\norm{a^{(n)}}_{l^4(\Zz)}^4}{ (\Var(G_n))^2} + t \frac{\norm{a^{(n)}}_{l^5(\Zz)}^5}{ (\Var(G_n))^{5/2}} + t^2 \frac{\norm{a^{(n)}}_{l^6(\Zz)}^6}{ (\Var(G_n))^{3}} }  \mathbb{E}[e^{tF}].
\end{align*}
\textbf{Third class of  subterms:} It consists only of $B_9$ and so we have to deal with $\mathbb{E}[B_9 e^{tF}]$. Multiplying all the $X$'s inside we get products of lengths two, three and four. The first two cases are already solved and a product of length four appears only one time, namely
\begin{align*}
		\frac{C}{ (\Var(G_n))^2} \sum_{k,l \in \Zz} (a^{(n)}_{k-1}) (a^{(n)}_{k}) (a^{(n)}_{l-1}) (a^{(n)}_{l}) \mathbb{E}[X_{k-1} X_{k+1} X_{l-1} X_{l+1} e^{tF}].
\end{align*}
We have two pairs of two equal indices of the $X$'s if and only if $k=l$ and then we are in the situation of $B_{21}$. Note that it is impossible that three or more indices are equal. If two indices are equal and two indices are different, e.g. $k-1 = l+1$ we are in the $2X$-case and can use \eqref{2X_firstorder_1}, \eqref{2X_firstorder_2} and \eqref{2X_secondorder}. At last, if all four indices are different, most of the work is done by \eqref{4X_firstorder} and \eqref{4X_secondorder} leading to upper bounds of order 
\begin{align*}
	O\prth{t^2 \frac{\norm{a^{(n)}}_{l^3(\Zz)}^6}{ (\Var(G_n))^{3}}} \quad \text{and} \quad  O\prth{t^2 \frac{\norm{a^{(n)}}_{l^4(\Zz)}^4}{ (\Var(G_n))^{2}}}.
\end{align*}
Combining our bounds from all the different cases we get 
\begin{align*}
	\mathbb{E}[B_9 e^{tF}] \leq C e^{ct} \prth{t^2 \frac{ \norm{a^{(n)}}_{l^3(\Zz)}^6}{(\Var(G_n))^3} + (1+t^2) \frac{\norm{a^{(n)}}_{l^4(\Zz)}^4}{ (\Var(G_n))^2} + t \frac{\norm{a^{(n)}}_{l^5(\Zz)}^5}{ (\Var(G_n))^{5/2}} + t^2 \frac{\norm{a^{(n)}}_{l^6(\Zz)}^6}{ (\Var(G_n))^{3}} }  \mathbb{E}[e^{tF}].
\end{align*}
Summarizing everything we have done so far a bound as in condition \eqref{A1} is obtained by
\begin{align*}
\mathbb{E} \left[ \abs{ 1 - \langle DF , - DL^{-1} F \rangle } e^{tF}\right] & \leq  \prth{ \mathbb{E} \left[ \prth{ 1 - \langle DF , - DL^{-1} F \rangle }^2 e^{tF}\right]}^\frac{1}{2} \prth{\mathbb{E} \left[ e^{tF}\right]}^\frac{1}{2} \\
                     & =  \prth{ \mathbb{E} \left[ \sum_{i=1}^9 B_i e^{tF}\right]}^\frac{1}{2} \prth{\mathbb{E} \left[ e^{tF}\right]}^\frac{1}{2} \\
                     & \leq \sum_{i=1}^9 \prth{ \mathbb{E} \left[ B_i e^{tF}\right]}^\frac{1}{2} \prth{\mathbb{E} \left[ e^{tF}\right]}^\frac{1}{2} \\        
                     &\leq\widetilde{\gamma_1}(t)\mathbb{E}\left[e^{tF}\right]                                                                                                 \end{align*}
such that 
\begin{align*}
	\widetilde{\gamma_1}(t) & = C e^{ct} \prth{t \frac{ \norm{a^{(n)}}_{l^3(\Zz)}^3}{(\Var(G_n))^{3/2}} + (1+t) \frac{\norm{a^{(n)}}_{l^4(\Zz)}^2}{ \Var(G_n)} + t^{1/2} \frac{\norm{a^{(n)}}_{l^5(\Zz)}^{5/2}}{ (\Var(G_n))^{5/4}} + t \frac{\norm{a^{(n)}}_{l^6(\Zz)}^3}{ (\Var(G_n))^{3/2}} } \\
                                                   & =: C e^{ct} \prth{t C_{n,1} + (1+t) C_{n,2} + t^{1/2} C_{n,3} + t C_{n,4} }.
\end{align*}
We now move on and show that a bound as in condition \eqref{A2} exists: Again, by the Cauchy--Schwarz inequality
\begin{align*}
\mathbb{E} \left[ \abs{  \delta\prth{ \frac{1}{\sqrt{pq}} DF \abs{D L^{-1} F} } } e^{tF}\right] \leq  \prth{ \mathbb{E} \left[ \prth{  \delta\prth{ \frac{1}{\sqrt{pq}} DF \abs{D L^{-1} F} } }^2 e^{tF}\right]}^\frac{1}{2} \prth{\mathbb{E} \left[ e^{tF}\right]}^\frac{1}{2}.
\end{align*}
By \cite[Corollary~9.9]{Pr08} it is $\delta(u) = \sum_{k=0}^\infty Y_k u_k$, where $Y_k$ is the $kth$ centered and standardized Rademacher random variable. In our case $Y_k = X_k$ and the corollary can be applied since $u_k := D_k F \abs{D_k L^{-1} F} / \sqrt{p_k q_k}$ does not depend on $X_k$. Then
\begin{align} \label{deltaPrivault}
\prth{ \mathbb{E} \left[ \prth{  \delta\prth{ u} }^2 e^{tF}\right]}^\frac{1}{2} & = \prth{ \sum_{k,l \in \Zz} \mathbb{E} \left[ u_k u_l X_k X_l  e^{tF}\right]}^\frac{1}{2} \nonumber \\
      & \leq \prth{ \sum_{k \in \Zz} \mathbb{E} \left[ u_k^2 e^{tF}\right]}^\frac{1}{2} + \prth{ \sum_{\substack{k,l \in \Zz \\ k \neq l}} \mathbb{E} \left[ u_k u_l X_k X_l  e^{tF}\right]}^\frac{1}{2}.
\end{align}
For the upcoming computations we recall
\begin{align*}
D_k F & = \frac{1}{4 \sqrt{\Var(G_n)}} \prth{ a^{(n)}_{k-1} (X_{k-1} + 1) + a^{(n)}_k (X_{k+1} + 1)}, \\
-D_k L^{-1} F & = \frac{1}{8 \sqrt{\Var(G_n)}} \prth{ a^{(n)}_{k-1} (X_{k-1} + 2) + a^{(n)}_k (X_{k+1} + 2)}.
\end{align*}
and so
\begin{align*}
 D_k F (-D_k L^{-1}F) =  \frac{1}{32 \Var(G_n)} & \Bigl((a^{(n)}_{k-1})^2 \left[ 3 X_{k-1} + 3\right] + (a^{(n)}_{k})^2 \left[3 X_{k+1} + 3\right] \\
     & + a^{(n)}_{k-1} a^{(n)}_{k}\left[3 X_{k-1} + 2 X_{k-1} X_{k+1} + 3 X_{k+1} + 4\right]\Bigl).
\end{align*}
The square of the righthandside is of a familiar form: Every summand consists of a product of length four of coefficients with index $k \pm ...$  multiplied with something bounded, and so as before we get immediately or by the AM-GM-inequality
\begin{align*}
\sum_{k \in \Zz} \mathbb{E} \left[ u_k^2 e^{tF}\right] \leq C \frac{\norm{a^{(n)}}_{l^4(\Zz)}^4}{ (\Var(G_n))^2} \mathbb{E} \left[ e^{tF}\right].
\end{align*}
For the remaining term of \eqref{deltaPrivault} we adapt the strategy that is used in the proof of Lemma~\ref{Lemma:XsTimesMGF} --- see its beginning for a detailed explanation. Set
\begin{align*}
	A_k := a^{(n)}_k \left[X_k  + X_k X_{k+1} + X_{k+1}\right]
\end{align*}
so that
\begin{align*}
F_a = \frac{1}{4 \sqrt{\Var(G_n)}} \prth{A_{k-2} + A_{k-1} + A_{k} + A_{k+1} + A_{l-2} + A_{l-1} + A_{l} + A_{l+1}}
\end{align*}
and by Taylor expansion
\begin{align*}
\mathbb{E}[u_k u_l X_k X_l e^{tF}] & = \mathbb{E}[u_k u_l X_k X_l e^{tF_a} e^{tF_u}] \\
                                       & = \mathbb{E}[u_k u_l X_k X_l e^{tF_u}] + t \mathbb{E}[u_k u_l X_k X_l F_a e^{tF_u}] + t^2 \mathbb{E}[u_k u_l X_k X_l F_a^2 r_2(t F_a) e^{tF_u}/2].
\end{align*}
\textbf{0-order-term:} By independence
\begin{align*}
 \mathbb{E}[u_k u_l X_k X_l e^{tF_u}] =  \mathbb{E}[u_k u_l X_k X_l]  \mathbb{E}[e^{tF_u}].
\end{align*}
Since by definition $X_k$ and $X_l$ respectively $u_k$ are independent, we just have to check whether the same goes for $X_k$ and $u_l$, which is leading to two cases.
\\ Case 1: $l \notin\{ k-1,k+1\}$
\begin{align*}
 \mathbb{E}[u_k u_l X_k X_l] =  \mathbb{E}[X_k]  \mathbb{E}[u_k u_l X_l] = 0.
\end{align*}
Case 2: $l \in\{ k-1,k+1\}$
\begin{align*}
\sum_{k \in \Zz} \mathbb{E}[u_k u_l X_k X_l] \mathbb{E}[e^{tF_u}] \leq C \frac{\norm{a^{(n)}}_{l^4(\Zz)}^4}{ (\Var(G_n))^2} \mathbb{E} \left[ e^{tF}\right] e^{ct},
\end{align*}
following our usual argumentation.
\\ \textbf{1st-order-term:} We split $F_u$ in the same manner as before, $F_u = F_{u_a} + F_{u_u}$, such that $F_{u_a} = \prth{A_{k-3} + A_{k+2} + A_{l-3} + A_{l+2}} /4 \sqrt{\Var(G_n)}$ and use another Taylor expansion of degree 1. Then
\begin{align} \label{deltafirstorder}
t \mathbb{E}[u_k u_l X_k X_l F_a e^{tF_u}] & = t \mathbb{E}[u_k u_l X_k X_l F_a e^{tF_{u_a}} e^{tF_{u_u}}] \nonumber \\
                                                                 & = t \mathbb{E}[u_k u_l X_k X_l F_a e^{tF_{u_u}}] + t^2 \mathbb{E}[u_k u_l X_k X_l F_a F_{u_a} r_1(t F_{u_a}) e^{tF_{u_u}}]. 
\end{align}
Note that $F_{u_u}$ is --- as part of $F_u$ --- independent of $X_k,X_l,u_k$ and $u_l$, but also independent of $F_a$ since we removed $F_{u_a}$, the depending part of $F_u$. As a consequence
\begin{align*}
	\mathbb{E}[u_k u_l X_k X_l F_a e^{tF_{u_u}}] =  \mathbb{E}[u_k u_l X_k X_l F_a]  \mathbb{E}[e^{tF_{u_u}}].
\end{align*}
Our next observation is $\mathbb{E}[u_k u_l X_k X_l F_a] = 0$ for $\abs{k-l}\geq 5$ --- in this case all appearing indices are different and the claim follows ultimately from independence. We treat the remaining case $\abs{k-l}\leq 4$ as four subcases $\abs{k-l} = i$ for $i \in \{1,2,3,4\}$, but here we just write down the first one $\abs{k-l} = 1$ as the others are analogous and so there outcome. In the mentioned case, if $l = k+1$, we receive
\begin{align*}
t \abs{u_k u_l X_k X_l F_a} \leq \frac{Ct}{ (\Var(G_n))^{5/2}} \quad \sum_{\mathclap{\substack{i_1,i_2 \in \{k-1,k\} \\ i_3,i_4 \in \{k,k+1\} \\ i_5 \in \{k-2,k-1,k,k+1,k+2\}}}} \quad \abs{a_{i_1}^{(n)}} \abs{a_{i_2}^{(n)}} \abs{a_{i_3}^{(n)}} \abs{a_{i_4}^{(n)}} \abs{a_{i_5}^{(n)}}
\end{align*}
and very similar for $l=k-1$. For both every summand consists of a product of five coefficients with index $k \pm ...$, and so we get 
\begin{align*}
\sum_{\substack{k,l \in \Zz \\ k \neq l \\ \abs{k-l}= 1}} t \abs{\mathbb{E}[u_k u_l X_k X_l F_a]  \mathbb{E}[e^{tF_{u_u}}]} & \leq  Ct \frac{\norm{a^{(n)}}_{l^5(\Zz)}^5}{ (\Var(G_n))^{5/2}} \mathbb{E}\left[ e^{tF}\right] e^{ct}.
\end{align*}
Having in mind that $\abs{r_1(t F_{u_a})} \leq e^{t \abs{F_{u_a}}}$ we can bound the second term of \eqref{deltafirstorder} by using
\begin{align*}
t^2 \abs{u_k u_l X_k X_l F_a F_{u_a}} \leq \frac{Ct^2}{ (\Var(G_n))^{3}} \quad \sum_{\mathclap{\substack{i_1,i_2 \in \{k-1,k\} \\ i_3,i_4 \in \{l-1,l\} \\ i_5 \in \{k-2,k-1,k,k+1,l-2,l-1,l,l+1\} \\ i_6 \in \{k-3,k+2,l-3,l+2\}}}} \quad \abs{a_{i_1}^{(n)}} \abs{a_{i_2}^{(n)}} \abs{a_{i_3}^{(n)}} \abs{a_{i_4}^{(n)}} \abs{a_{i_5}^{(n)}} \abs{a_{i_6}^{(n)}}
\end{align*}
and every summand consists of a product of length six of either three coefficients with index $k \pm ...$  and three coefficients with index $l \pm ...$, or four coefficients with index $k \pm ...$  and two coefficients with index $l \pm ...$or the other way around. Combining the cases we get 
\begin{align*}
\sum_{\substack{k,l \in \Zz \\ k \neq l}} t^2 \abs{\mathbb{E}[u_k u_l X_k X_l F_a F_{u_a} r_1(t F_{u_a}) e^{tF_{u_u}}] } & \leq Ct^2 \prth{\frac{\norm{a^{(n)}}_{l^4(\Zz)}^4}{ (\Var(G_n))^{2}} + \frac{\norm{a^{(n)}}_{l^3(\Zz)}^6}{ (\Var(G_n))^{3}}}\mathbb{E}\left[ e^{tF}\right] e^{ct}.
\end{align*}
\textbf{2nd-order-term:} Finally, having in mind that $\abs{r_2(t F_{a})} \leq e^{t \abs{F_{a}}}$ we can bound the last term of our original Taylor expansion by using
\begin{align*}
t^2 \abs{u_k u_l X_k X_l F_a^2} \leq \frac{Ct^2}{ (\Var(G_n))^{3}} \quad \sum_{\mathclap{\substack{i_1,i_2 \in \{k-1,k\} \\ i_3,i_4 \in \{l-1,l\} \\ i_5,i_6 \in \{k-2,k-1,k,k+1,l-2,l-1,l,l+1\} }}} \quad \abs{a_{i_1}^{(n)}} \abs{a_{i_2}^{(n)}} \abs{a_{i_3}^{(n)}} \abs{a_{i_4}^{(n)}} \abs{a_{i_5}^{(n)}} \abs{a_{i_6}^{(n)}}
\end{align*}
and get analogously the bound
\begin{align*}
\sum_{\substack{k,l \in \Zz \\ k \neq l}} t^2 \abs{\mathbb{E}[u_k u_l X_k X_l F_a^2 r_2(t F_{a}) e^{tF_{u}}/2 ] }& \leq Ct^2 \prth{\frac{\norm{a^{(n)}}_{l^4(\Zz)}^4}{ (\Var(G_n))^{2}} + \frac{\norm{a^{(n)}}_{l^3(\Zz)}^6}{ (\Var(G_n))^{3}}}\mathbb{E}\left[ e^{tF}\right] e^{ct}.
\end{align*}
Summarizing everything we have done so far a bound as in condition \eqref{A2} is obtained by
\begin{align*}
\mathbb{E} \left[ \abs{  \delta\prth{ u } } e^{tF}\right] & \leq  \prth{ \mathbb{E} \left[ \prth{  \delta\prth{ u } }^2 e^{tF}\right]}^\frac{1}{2} \prth{\mathbb{E} \left[ e^{tF}\right]}^\frac{1}{2} \\
                                                                                          & \leq \prth{ \prth{ \sum_{k \in \Zz} \mathbb{E} \left[ u_k^2 e^{tF}\right]}^\frac{1}{2} + \prth{ \sum_{\substack{k,l \in \Zz \\ k \neq l}} \mathbb{E} \left[ u_k u_l X_k X_l  e^{tF}\right]}^\frac{1}{2} } \prth{\mathbb{E} \left[ e^{tF}\right]}^\frac{1}{2} \\
                                                                                          & \leq \widetilde{\gamma_2}(t)\mathbb{E}\left[e^{tF}\right]  
\end{align*}
such that 
\begin{align*}
	\widetilde{\gamma_2}(t) & = C e^{ct} \prth{  t \frac{ \norm{a^{(n)}}_{l^3(\Zz)}^3}{(\Var(G_n))^{3/2}} + (1+t) \frac{\norm{a^{(n)}}_{l^4(\Zz)}^2}{ \Var(G_n)} + t^{1/2} \frac{\norm{a^{(n)}}_{l^5(\Zz)}^{5/2}}{ (\Var(G_n))^{5/4}} } \\
                                                   & =: C e^{ct} \prth{t C_{n,1} + (1+t) C_{n,2} + t^{1/2} C_{n,3} }.
\end{align*}
In a final step we want to simplify our bounds by comparing the constants $C_{n,i}$ with each other. To do so, we will use, for $m \geq 2$:
\begin{align}
\norm{a^{(n)}}_{l^m(\Zz)}^m & = \sum_{k \in \Zz} \prth{\abs{a_k}^{m-1} \abs{a_k} } \nonumber \\
                                                   & \leq \sqrt{\sum_{k \in \Zz} \prth{\abs{a_k}^{m-1}}^2 \sum_{k \in \Zz} \abs{a_k}^2 } \nonumber  \\
                                                   & = \sqrt{\sum_{k \in \Zz} \prth{\abs{a_k}^{m-1}}^2  } \cdot C \cdot (\Var(G_n))^{1/2} \label{cons_simp_1}\\
                                                   & \leq \sum_{k \in \Zz} \abs{a_k}^{m-1} \cdot C \cdot (\Var(G_n))^{1/2} \nonumber  \\
                                                   & = \norm{a^{(n)}}_{l^{m-1}(\Zz)}^{m-1} \cdot C \cdot (\Var(G_n))^{1/2} \label{cons_simp_2} \\ \nonumber 
\end{align}
by the Cauchy--Schwarz inequality. Then \eqref{cons_simp_1} and \eqref{cons_simp_2} imply
\begin{align*}
C_{n,1} & = \sum_{k \in \Zz} \abs{a_k}^{3} (\Var(G_n))^{-3/2} \leq \prth{\sum_{k \in \Zz} \abs{a_k}^{4}}^{1/2} (\Var(G_n))^{-1} = C_{n,2}, \\
C_{n,3} & = \prth{\sum_{k \in \Zz} \abs{a_k}^{5}}^{1/2} (\Var(G_n))^{-5/4} \leq \prth{\sum_{k \in \Zz} \abs{a_k}^{4} (\Var(G_n))^{1/2}}^{1/2} (\Var(G_n))^{-5/4} = C_{n,2}, \\
C_{n,4} & = \prth{\sum_{k \in \Zz} \abs{a_k}^{6}}^{1/2} (\Var(G_n))^{-3/2} \leq \prth{\sum_{k \in \Zz} \abs{a_k}^{5} (\Var(G_n))^{1/2}}^{1/2} (\Var(G_n))^{-3/2} \leq C_{n,2}.
\end{align*}
So, we choose $\gamma_1(t) = \gamma_2(t)  := C e^{ct} \prth{ (1 + t^{1/2} + t) C_n }$ for \eqref{A1} and \eqref{A2}, and $C_n := C_{n,2}$. 
\end{proof}
\section{Proofs II: Subgraph counts in the Erd\H{o}s--R\'{e}nyi random graph}

\noindent There are $\binom{n}{2}$ possible edges in the Erd\H{o}s-R\'{e}nyi random graph.
Hence, we can describe it by using as many Rademacher random variables:
Let $E$ be the set of all possible edges of $G(n,p)$
and let $(X_k)_{k \in E}$ be a set of independent Rademacher random variables,
in which $X_k = 1$ indicates the presence of edge $k$ in $G(n,p)$.
Thus, $\Pp(X_k=1)=p$ for all $k \in E$.
In our calculations, we will make use of the following rescaled versions of $(X_k)_{k \in E}$:
For $k \in E$ we define
$B_k = \frac12 (X_k + 1)$, which is Bernoulli($p$)-distributed,
and $Y_k = (pq)^{-\frac12} (B_k - p)$, which is standardized.
Further, for any subset $A \subset E$,
we shorten the notation for the product of $(B_k)_{k \in A}$
or $(Y_k)_{k \in A}$
by defining $B_A := \prod_{k \in A} B_k$
and $Y_A := \prod_{k \in A} Y_k$,
where $B_\emptyset = Y_\emptyset = 1$ by convention.
For any integrable random variable $Z$, let $Z^c := Z - \E[Z]$ denote the centered version of $Z$.

Let $\G$ be a graph with at least one edge.
Neither the standardized number of copies of $\G$ in $G(n,p)$
nor $\Psimin$ do depend on the number of isolated vertices of $\G$,
see Lemma~4.3 in \cite{ER22}.
Hence, without loss of generality, we may assume that $\G$ does not have isolated vertices.
This way, every copy of $\G$ in $G(n,p)$ can be identified by its set of edges.
Therefore, it will be useful to simplify our notation:

For any set of edges $\graph \subset E$
let $\induced{\graph}$ denote the graph consisting of all edges given by $\graph$ and all necessary vertices.
Let $\vertices{\graph}:=\vertices{\induced{\graph}}$, $\edges{\graph}:=\edges{\induced{\graph}}$,
and $\aut{\graph}:=\aut{\induced{\graph}}$.
We will call $\graph$ to be a \emph{copy} of $\G$ if $\induced{\graph}$ is a copy of $\G$.
Let $\copies := \{ \graph \subset E \,\vert\, \graph$ is a copy of $\G \}$ be the set of all possible copies of $\G$ in $G(n,p)$,
and let $\copiesk{k} := \{ \graph \in \copies \,\vert\, k\in\graph \}$ be the set of all possible copies that contain a specific edge $k \in E$.
Further, for any non-empty $A \subset E$ let
$\copiesk{A} := \bigcup_{k \in A} \copiesk{k}$
be the set of all possible copies that contain at least one of the edges given by $A$.
This last definition will be mainly used in case of $A = \graph \in \copies$ being a copy of $\G$.
In this case, $\copiesk{\graph}$ contains all possible copies of $\G$ that have at least one edge in common with $\graph$.
We will call $\copiesk{\graph}$ to be the \emph{neighborhood} of $\graph$, and the elements of $\copiesk{\graph}$ to be the \emph{neighbors} of $\graph$.
Due to symmetry, the cardinality of the neighborhood $\copiesk{\graph}$ does not depend on the choice of $\graph \in \copies$.
We will denote this cardinality by $D$.

The number of copies of $\G$ in $G(n,p)$ is given by
$\sum_{\graph \in \copies} B_\graph$.
Its standardization is
\begin{align*}
	W
	&:= \frac{1}{\sigma} \sum_{\graph \in \copies} (B_\graph - \E[B_\graph])
	 = \frac{1}{\sigma} \sum_{\graph \in \copies} B_\graph^c
	,
\end{align*}
where $\sigma^2 := \Var(\sum_{\graph \in \copies} B_\graph)$.
From Lemma~4.2 in \cite{ER22} we know that
\begin{align}
	\sigma^2
	&\geq
	\frac{q}{2 \cdot \vertices{\G}! \cdot \aut{\G}} \cdot \frac{n^{2\vertices{\G}} p^{2\edges{\G}}}{\Psimin}
	\label{sg:ineq:sigma}
\end{align}
for $n \geq 4\vertices{\G}^2$.

The main goal of this section is to prove Theorem~\ref{sg:theo:main} and Corollary~\ref{sg:cor:main},
which have been presented in our introductory section.
These proofs are postponed to the end of this section.
Before, we present several lemmas that deal with some basic yet very important properties of our random variables.
These results will be useful repeatedly in the proof of Theorem~\ref{sg:theo:main}.

We start by exploring the behavior of our most important operators,
the discrete gradient $D_k$ with $k \in E$
and the pseudo-inverse Ornstein-Uhlenbeck operator $L^{-1}$.
\begin{lemma} \label{sg:lem:LDk}
	For any subset $A \subset E$ and any edge $k \in E$
	there is
	\begin{flalign*}
		(i)&&
			D_k B_A^c
			&=
			\sqrt{pq}
			\cdot \1{k \in A}
			\cdot B_{A \backslash \{ k \} },
			&&\\
		(ii)&&
			- D_k L^{-1} B_A^c
			&=
			\sqrt{pq}
			\cdot \1{k \in A}
			\cdot \sum_{\{k\} \subset \alpha \subset A }
				\frac{p^{\vert A \vert - \vert \alpha \vert}}{\vert A \vert \cdot \binom{\vert A \vert - 1}{\vert \alpha \vert - 1}}
				B_{\alpha \backslash \{ k \} }.
	\end{flalign*}
	In particular, all expressions above are non-negative.
\end{lemma}
\begin{remark} \label{sg:rem:1}
	In the application of Lemma~\ref{sg:lem:LDk} $(ii)$ it will be useful to keep in mind that
	\begin{align*}
		\sum_{\{k\} \subset \alpha \subset A }
				\frac{1}{\vert A \vert \cdot \binom{\vert A \vert - 1}{\vert \alpha \vert - 1}}
		&=
		\sum_{\alpha \subset A \backslash \{k\} }
				\frac{1}{\vert A \vert \cdot \binom{\vert A \vert - 1 }{\vert \alpha \vert}}
		=
		\sum_{i=0}^{\vert A \vert - 1} \bigg( \frac{1}{\vert A \vert}
		\sum_{ \substack{\alpha \subset A \backslash \{k\} \\ \vert\alpha\vert = i} } \frac{1}{\binom{\vert A \vert - 1}{i}} \bigg)
		=
		\sum_{i=0}^{\vert A \vert - 1} \frac{1}{\vert A \vert}
		=
		1
	\end{align*}
	for any non-empty subset $A \subset E$ and any edge $k \in A$.
\end{remark}
\begin{proof}[Proof of Lemma~\ref{sg:lem:LDk}]
	First, we note that $D_k B_A^c = D_k (B_A - \E[B_A]) = D_k B_A$.
	If $k \not\in A$, $B_A$ is independent of $X_k$ so that $D_k B_A = 0$.
	If $k \in A$, there is $B_A = B_{ A \backslash \{ k \} } \cdot B_k = B_{ A \backslash \{ k \} } \cdot \frac12(X_k+1)$
	so that
	$D_k B_A = \sqrt{pq} \cdot ( B_{ A \backslash \{ k \} } \cdot 1 - B_{ A \backslash \{ k \} } \cdot 0)$.
	This proves $(i)$.
	
	To prove the second part of our statement,
	we need to represent $B_A$ with regard to the standardized random variables $Y_\ell = (pq)^{-\frac12} (B_\ell - p)$, $\ell \in E$.
	By expansion of the product, we see that
	\begin{align*}
		B_A^c
		&= B_A - \E[B_A]
		= \prod_{\ell \in A} \big( (pq)^{\frac12} Y_\ell + p \big) - p^{\vert A \vert}
		= \sum_{\emptyset \neq \alpha \subset A} p^{\vert A \vert - \frac12 \vert\alpha\vert} q^{\frac12 \vert\alpha\vert} Y_\alpha
		.
	\end{align*}
	Now we are able to describe how the pseudo-inverse Ornstein--Uhlenbeck operator works on this expression.
	We get
	\begin{align*}
		-L^{-1} B_A^c
		&=\sum_{\emptyset \neq \alpha \subset A}
			\frac{
				p^{\vert A \vert - \frac12 \vert\alpha\vert}
				q^{\frac12 \vert\alpha\vert}
			}{\vert\alpha\vert}
			Y_\alpha
		.
	\end{align*}
	Having applied the operator,
	we now want to transform the result back to a representation using the Bernoulli variables $(B_\ell)_{\ell \in E}$.
	Thus,
	\begin{align*}
		-L^{-1} B_A^c
		&=\sum_{\emptyset \neq \alpha \subset A}
			\bigg(
			\frac{
				p^{\vert A \vert - \frac12 \vert\alpha\vert}
				q^{\frac12 \vert\alpha\vert}
			}{\vert\alpha\vert}
			\cdot \prod_{\ell \in \alpha}
				\Big( (pq)^{-\frac12} (B_\ell - p) \Big)
			\bigg)
		=\sum_{\emptyset \neq \alpha \subset A}
			\sum_{\hat\alpha \subset \alpha}
				\frac{
					p^{\vert A \vert - \vert\hat\alpha\vert}
				}{\vert\alpha\vert}
				(-1)^{\vert \alpha \vert - \vert \hat\alpha \vert}
				B_{\hat\alpha}
			.
	\end{align*}
	The double sum can be rearranged
	so that we can first choose $\hat\alpha$ as an arbitrary subset of $A$,
	and then fill up the rest $\tilde\alpha := \alpha\backslash\hat\alpha \subset A\backslash\hat\alpha$.
	However, using this rearrangement,
	we have to explicitly exclude the case in which $\hat\alpha = \tilde\alpha = \emptyset$.
	This leads to
	\begin{align*}
		-L^{-1} B_A^c
		&=\sum_{\hat\alpha \subset A}
			\sum_{\substack{\tilde\alpha\subset A\backslash\hat\alpha\\
				\hat\alpha \cup \tilde\alpha \neq \emptyset}}
				\frac{
					p^{\vert A \vert - \vert\hat\alpha\vert}
				}{\vert\hat\alpha\vert + \vert\tilde\alpha\vert}
				(-1)^{\vert \tilde\alpha \vert}
				B_{\hat\alpha}
			\\
		&= \sum_{\emptyset\neq\tilde\alpha\subset A}
				\frac{
					p^{\vert A \vert}
				}{\vert\tilde\alpha\vert}
				(-1)^{\vert \tilde\alpha \vert}
				B_\emptyset
			+			
			\sum_{\emptyset \neq \hat\alpha \subset A}
			\sum_{\tilde\alpha\subset A\backslash\hat\alpha}
				\frac{
					p^{\vert A \vert - \vert\hat\alpha\vert}
				}{\vert\hat\alpha\vert + \vert\tilde\alpha\vert}
				(-1)^{\vert \tilde\alpha \vert}
				B_{\hat\alpha}
			\\
		&= p^{\vert A \vert}
			\sum_{i = 1}^{\vert A \vert}
				\binom{\vert A \vert}{i}
				\frac{
					(-1)^{i}
				}{i}
			+			
			\sum_{\emptyset \neq \hat\alpha \subset A}
			\Bigg(
				p^{\vert A \vert - \vert\hat\alpha\vert}
				B_{\hat\alpha}
			\sum_{i=0}^{\vert A \vert - \vert\hat\alpha\vert}
				\binom{\vert A \vert - \vert\hat\alpha\vert}{i}
				\frac{
					(-1)^{i}
				}{\vert\hat\alpha\vert + i}
			\Bigg)
			.
	\end{align*}
	The first term above is not random but deterministic
	and hence not of interest,
	as it will disappear after the application of the discrete gradient.
	To handle the second term, we use the Beta-function,
	which is usually defined via $\operatorname{Beta}(x,y) = \int_0^1 t^{x-1} (1-t)^{y-1} dt$
	for all $x,y \in \C$ that have a positive real part.
	It is well known that in case of $x,y \in \N$
	there is
	$\sum_{i=0}^{x-1} \binom{x-1}{i} \frac{(-1)^i}{y+i}
	= \operatorname{Beta}(x,y)
	= (x+y-1)^{-1} \cdot \binom{x+y-2}{y-1}^{-1}$.
	Application of these results on
	$x=\vert A \vert - \vert\hat\alpha\vert + 1$ and $y=\vert\hat\alpha\vert$ yields
	\begin{align*}
		\sum_{i=0}^{\vert A \vert - \vert\hat\alpha\vert}
				\binom{\vert A \vert - \vert\hat\alpha\vert}{i}
				\frac{
					(-1)^{i}
				}{\vert\hat\alpha\vert + i}
		&= \frac{1}{\vert A \vert \cdot \binom{\vert A \vert - 1}{\vert \hat\alpha \vert - 1}},
	\end{align*}
	so that
	\begin{align*}
		-L^{-1} B_A^c
		&= p^{\vert A \vert}
			\sum_{i = 1}^{\vert A \vert}
				\binom{\vert A \vert}{i}
				\frac{
					(-1)^{i}
				}{i}
			+			
			\sum_{\emptyset \neq \hat\alpha \subset A}
				\frac{p^{\vert A \vert - \vert \hat\alpha \vert}}{\vert A \vert \cdot \binom{\vert A \vert - 1}{\vert \hat\alpha \vert - 1}}
				B_{\hat\alpha}
			.
	\end{align*}
	When applying the discrete gradient,
	the deterministic term cancels itself out,
	while the second term can be handled the same way as shown in the first part of this proof.
	We arrive at
	\begin{align*}
		-D_k L^{-1} B_A^c
		&= \sqrt{pq} \cdot
			\sum_{\emptyset \neq \hat\alpha \subset A}
				\frac{p^{\vert A \vert - \vert \hat\alpha \vert}}{\vert A \vert \cdot \binom{\vert A \vert - 1}{\vert \hat\alpha \vert - 1}}
				B_{\hat\alpha \backslash \{k\} }
				\1{k \in \hat\alpha}\\
		&=
			\sqrt{pq}
			\cdot \1{k \in A}
			\cdot \sum_{\{k\} \subset \alpha \subset A }
				\frac{p^{\vert A \vert - \vert \alpha \vert}}{\vert A \vert \cdot \binom{\vert A \vert - 1}{\vert \alpha \vert - 1}}
				B_{\alpha \backslash \{ k \} }
			.
	\end{align*}
	This proves the second part of our statement.
\end{proof}

Next, we want to improve our understanding of the correlation of our Bernoulli random variables $\{B_A\}_{A \subset E}$.
\begin{lemma}\label{sg:lem:correlation}
	For every $i \in \N$ let $A_i \subset E$ be a set of edges.
	And let $I \subset \N$ be a finite index set.
	Then the following statements hold:
	\makeatletter\tagsleft@true\makeatother
	\begin{alignat}{2}
		0 &\; \leq \; \E[B_{A_1}]\E[B_{A_2}] &&\; \leq \; \E[ B_{A_1 \cup A_2} ], \tag{$i$}\\
		0 &\; \leq \; \E[ B_{A_1}^c B_{A_2}^c ] &&\; \leq \; \E[B_{A_1 \cup A_2}], \tag{$ii$}\\
		0 &\; \leq \; \E\Big[ \Big\vert \prod_{i \in I} B_{A_i}^c \Big\vert \Big] &&\; \leq \; 2^{\vert I \vert} \E[ B_{ \bigcup_{i \in I} A_i}]. \tag{$iii$}
	\end{alignat}
	\makeatletter\tagsleft@false\makeatother
\end{lemma}
\begin{proof}
	Obviously, $\E[B_{A_i}] = p^{\vert A_i \vert}$ is non-negative for all $i \in \N$.
	Further,
	\begin{align*}
			\E[B_{A_1}]\E[B_{A_2}]
		&=	p^{\vert A_1 \vert + \vert A_2 \vert}
		\leq	p^{\vert A_1 \cup A_2 \vert}
		=	\E[ B_{A_1 \cup A_2} ].
	\end{align*}
	This proves $(i)$.
	To varify $(ii)$, see
	\begin{align*}
			\E[ B_{A_1}^c B_{A_2}^c ]
		&=	\E \big[ \big( B_{A_1} - \E[B_{A_1}] \big) \cdot \big( B_{A_2} - \E[B_{A_2}] \big) \big]
		=	\E[B_{A_1 \cup A_2}] - \E[B_{A_1}]\E[B_{A_1}]
		.
	\end{align*}
	Hence, $(ii)$ follows from $(i)$.
	Finally, by expansion of the product, we see
	\begin{align*}
			\E\Big[ \Big\vert \prod_{i \in I} B_{A_i}^c \Big\vert \Big]
		&\leq
			\sum_{\hat I \subset I} \bigg( \E\Big[ \prod_{i \in \hat I} B_{A_i} \Big] \cdot \prod_{i \in I \backslash \hat I} \E[ B_{A_i} ] \bigg)
		\leq
			\sum_{\hat I \subset I} \E[ B_{ \bigcup_{i \in I} A_i}]
		=
			2^{\vert I \vert} \E[ B_{ \bigcup_{i \in I} A_i}],
	\end{align*}
	where we again used our result from $(i)$.
	This finishes the proof.
\end{proof}

We will later have to deal with sums of powers of $p$,
in which the structure of the sum is determined via the neighborhood structure of the copies of $\G$.
These sums together with the variance $\sigma^2$ are the main reason for the appearance of $\Psimin$ in our results,
and will now be examined in detail.
\begin{definition}\label{sg:defn:connected}
	Let $m \in \N$ be a natural number,
	and let $M=(\graph_i)_{i=1}^m \subset \copies$ be an ordered set of copies,
	in which each copy is a neighbor of at least one of its predecessors,
	i.e. $\graph_1 \in \copies$ and $\graph_i \in \bigcup_{j=1}^{i-1} \copiesk{\graph_j}$ for all $2 \leq i \leq m$.
	We then call $M$ to be a \emph{set of connected copies} of size $m$.
\end{definition}
\begin{lemma}\label{sg:lem:connected}
	Let $m$ and $\hat m$ be natural numbers,
	and let $\bar m = m + \hat m$.
	Then
	\begin{flalign*}
		(i)&&
		\sum_{\mathclap{ M : \vert M \vert = m \text{\,con.\,cop.}}}
			p^{\vert \bigcup_{\graph \in M} \graph \vert}
		&\leq
			c_m
			\cdot
			\frac{n^{m \cdot \vertices{\G}} \cdot p^{m \cdot \edges{\G}}}{\Psimin^{m-1}}
		,&&\\
		(ii)&&
		\sum_{\mathclap{ \substack{ M : \vert M \vert = m  \text{\,con.\,cop.}
						\\ \hat M : \vert \hat M \vert = \hat m  \text{\,con.\,cop.}}}}
			p^{\vert \bigcup_{\graph \in M \cup \hat M} \graph \vert}
		&\leq
			c_{\bar m}
			\cdot
			\frac{n^{\bar m \cdot \vertices{\G}} \cdot p^{ \bar m \cdot \edges{\G}}}{\Psimin^{\bar m -2} \cdot \min\{ \Psimin , 1 \} }
		,&&
	\end{flalign*}
	where the sums run over all sets of connected copies of size $m$ or $\hat m$, respectively,
	and where the constants are given by
	$
		c_k := (\vertices{\G}!)^{k-1} \cdot 2^{\frac12 k (k-1) \cdot \edges{\G}} \cdot \aut{\G}^{-k}
	$
	for $k \in \N$.
\end{lemma}
This lemma summarizes a counting technique that is commonly used in the field of random subgraph counting,
see e.g. (3) in \cite{R88}, (3.10) in \cite{BKR89}, and (3.10) in \cite{JLR00}.
However, in contrast to these references, we formulate the result of this counting strategy as an independent lemma
for an arbitrary number of connected copies of $\G$.
\begin{proof}[Proof of Lemma~\ref{sg:lem:connected}]
	To prove $(i)$, let us first assume that $m=1$.
	In this case, every set of connected copies $M$ consists of only one copy $\graph_1$
	that can be arbitrarily chosen among all copies in $\copies$.
	Therefore, we are interested in the size of $\copies$:
	There are $\binom{n}{\vertices{\G}}$ possibilities to choose $\vertices{\G}$ vertices.
	On every set of $\vertices{\G}$ vertices, we can find $\frac{\vertices{\G}!}{\aut{\G}}$ copies of $\G$.
	So there are
	$\frac{\vertices{\G}!}{\aut{\G}} \cdot \binom{n}{\vertices{\G}}
	\leq \frac{n^{\vertices{\G}}}{\aut{\G}}$
	possible choices for $\graph_1 \in \copies$.
	Hence,
	\begin{align*}
		\sum_{\mathclap{ \substack{ M : \vert M \vert = 1 \\ \text{\,con.\,cop.} }}}
			p^{\vert \bigcup_{\graph \in M} \graph \vert}
		&=
		\sum_{\mathclap{ \graph \in \copies }}
			p^{\vert \graph \vert}
		\leq
		\frac{n^{\vertices{\G}}}{\aut{\G}} \cdot p^{e_\G}
		=
			c_1
			\cdot
			\frac{n^{1 \cdot \vertices{\G}} \cdot p^{1 \cdot \edges{\G}}}{\Psimin^{0}}
		.
	\end{align*}
	Next, we assume that $(i)$ holds true for some fixed $m \in \N$
	and we want to prove that $(i)$ still holds if $m$ is increased by $1$.
	If $\hat M$ is a set of connected copies of size $m+1$,
	then the first $m$ elements of $\hat M$ form a set of connected copies of size $m$.
	Hence,
	\begin{align*}
		\sum_{\mathclap{ \substack{ \hat M : \vert \hat M \vert = m+1 \\ \text{\,con.\,cop.} }}}
			p^{\vert \bigcup_{\graph \in \hat M} \graph \vert}
		&=\;\;\;
		\sum_{\mathclap{ \substack{ M : \vert M \vert = m \\ \text{\,con.\,cop.} }}}
			p^{\vert \bigcup_{\graph \in M} \graph \vert}
		\;
		\sum_{\mathclap{ \graph_{m+1} \in \bigcup_{\graph \in M} \copiesk{\graph} }}
			p^{\vert \graph_{m+1} \backslash \bigcup_{\graph \in M} \graph \vert}
		.
	\end{align*}
	Given $M$ a set of connected copies of size $m$,
	$\graph_{m+1}$ has to be a neighbor of at least one $\graph \in M$.
	In particular, the intersection $\graph_{m+1} \cap \bigcup_{\graph \in M} \graph$ may not be empty.
	To understand the resulting structure, we fix a non-empty subset
	$ H \subset \bigcup_{\graph \in M} \graph$
	and we count the possible choices for $\graph_{m+1}$,
	whose intersection with $\bigcup_{\graph \in M} \graph$ equals $H$.
	Assuming that $H$ is small enough that it can be completed to a copy of $\G$,
	there are $\binom{n-\vertices{H}}{\vertices{\G}-\vertices{H}}$ possibilities
	to complete the $\vertices{H}$ vertices of $H$ to a set of $\vertices{\G}$ vertices.
	On such a set there are $\frac{\vertices{\G}!}{\aut{\G}}$ copies of $\G$.
	Hence, we can find at most
	$\frac{\vertices{\G}!}{\aut{\G}} \cdot \binom{n-\vertices{H}}{\vertices{\G}-\vertices{H}}
	\leq \frac{\vertices{\G}!}{\aut{\G}} \cdot n^{\vertices{\G} - \vertices{H}}$
	possible copies $\graph_{m+1} \in \copies$ with $H = \graph_{m+1} \cap \bigcup_{\graph \in M} \graph$.
	Further note that
	$\vert \graph_{m+1} \backslash \bigcup_{\graph \in M} \graph \vert
	= \vert \graph_{m+1} \backslash H \vert
	= \edges{\G} - \edges{H}$.
	Thus,
	\begin{align*}
		\sum_{\mathclap{ \substack{ M : \vert M \vert = m \\ \text{\,con.\,cop.} }}}
			p^{\vert \bigcup_{\graph \in M} \graph \vert}
		\;
		\sum_{\mathclap{ \graph_{m+1} \in \bigcup_{\graph \in M} \copiesk{\graph} }}
			p^{\vert \graph_{m+1} \backslash \bigcup_{\graph \in M} \graph \vert}
		&\leq\;\;\;
		\sum_{\mathclap{ \substack{ M : \vert M \vert = m \\ \text{\,con.\,cop.} }}}
			p^{\vert \bigcup_{\graph \in M} \graph \vert}
		\;
		\sum_{\mathclap{ H \underset{\G}{\subset} \bigcup_{\graph \in M} \copiesk{\graph} }}
			\vertices{\G}! \cdot \aut{\G}^{-1} \cdot
			n^{\vertices{\G} - \vertices{H}} \cdot
			p^{\edges{\G} - \edges{H}}	
		,
	\end{align*}
	where $H \underset{\G}{\subset} \bigcup_{\graph \in M} \copiesk{\graph}$ denotes
	that $H$ is a non-empty subset of $\bigcup_{\graph \in M} \copiesk{\graph}$
	while additionally being isomorphic to a subset of $\G$.
	There are less than $2^{m\cdot \edges{\G}}$ sets with these properties.
	For each of them there is $n^{\vertices{H}} p^{\edges{H}} \geq \Psimin$,
	so that
	\begin{align*}
		\sum_{\mathclap{ \substack{ M : \vert M \vert = m \\ \text{\,con.\,cop.} }}}
			p^{\vert \bigcup_{\graph \in M} \graph \vert}
		\;
		\sum_{\mathclap{ H \underset{\G}{\subset} \bigcup_{\graph \in M} \copiesk{\graph} }}
			\vertices{\G}! \cdot \aut{\G}^{-1} \cdot
			n^{\vertices{\G} - \vertices{H}} \cdot
			p^{\edges{\G} - \edges{H}}	
		&\leq\;\;\;
		\sum_{\mathclap{ \substack{ M : \vert M \vert = m \\ \text{\,con.\,cop.} }}}
			p^{\vert \bigcup_{\graph \in M} \graph \vert}
		\cdot
		2^{m\cdot \edges{\G}}
		\cdot
		\frac{\vertices{\G}!}{\aut{\G}}
		\cdot
		\frac{n^{\vertices{\G}} p^{\edges{\G}}}{\Psimin}
		.
	\end{align*}
	As we assumed that $(i)$ holds true for $m$, we arrive at
	\begin{align*}
		\sum_{\mathclap{ \substack{ M : \vert M \vert = m \\ \text{\,con.\,cop.} }}}
			p^{\vert \bigcup_{\graph \in M} \graph \vert}
		\cdot
		2^{m\cdot \edges{\G}}
		\cdot
		\frac{\vertices{\G}!}{\aut{\G}}
		\cdot
		\frac{n^{\vertices{\G}} p^{\edges{\G}}}{\Psimin}
		&\leq
			c_m
			\cdot
			\frac{n^{m \cdot \vertices{\G}} \cdot p^{m \cdot \edges{\G}}}{\Psimin^{m-1}}
		\cdot
		2^{m\cdot \edges{\G}}
		\cdot
		\frac{\vertices{\G}!}{\aut{\G}}
		\cdot
		\frac{n^{\vertices{\G}} p^{\edges{\G}}}{\Psimin}
		\\
		&=
			c_{m+1}
			\cdot
			\frac{n^{(m+1) \cdot \vertices{\G}} \cdot p^{(m+1) \cdot \edges{\G}}}{\Psimin^{(m+1)-1}}
		.
	\end{align*}
	This finishes the proof of $(i)$.
	To prove the second part, we want to adapt our strategy from above.
	Let $M=(\graph_i)_{i=1}^m \subset \copies$ be a set of connected copies of size $m$,
	and $\hat M=(\hat \graph_i)_{i=1}^{\hat m} \subset \copies$ be a set of connected copies of size $\hat m$.
	Further, let
	$\bar M := ( \graph_1, \dots, \graph_m, \hat \graph_1, \dots, \hat \graph_{\hat m})$
	be the set created by joining $M$ and $\hat M$.
	If we could be sure that $\bar M$ is a set of connected copies, we could directly apply $(i)$.
	However, we can not generally assume that $\bar M$ is a set of connected copies.
	In fact, $\bar M$ is a set of connected copies of size $m + \hat m$
	if and only if $\hat \graph_1 \in \bigcup_{i=1}^{m} \copiesk{\graph_i}$.
	Hence, we have to modify our strategy from $(i)$ only with respect to $\hat H = \hat \graph_1 \cap \bigcup_{i=1}^{m} \copiesk{\graph_i}$,
	which can be empty in this case.
	Therefore, the lower bound
	$n^{\vertices{\hat H}} \cdot p^{\edges{ \hat H}} \geq \Psimin$
	does not necessarily hold.
	We have to take the case $\hat H = \emptyset$ into account, which results in
	$n^{\vertices{\hat H}} \cdot p^{\edges{ \hat H}} \geq \min\{\Psimin, 1\}$.
	This proves the second part of the statement.	
\end{proof}

Before we turn to the main proofs, we want to exemplify two bounding strategies in connection with the moment generating function.
We will repeatedly use these strategies, later.
\begin{lemma} \label{sg:lem:momgenfunc}
	Let $A_1, A_2 \subset \copies$ be sets of copies of $\G$,
	and let $F$ be a non-negative functional of $\{B_{\graph}\}_{\graph \in A_2}$,
	so that $F$ is a random variable independent from $\{B_{\graph}\}_{\graph \in \copies \backslash \hat A_2}$,
	where $\hat A_2 := \bigcup_{\graph \in A_2} \copiesk{\graph}$ is the set of all copies of $\G$ that depend on at least one element from $A_2$,
	in particular $A_2 \subset \hat A_2$.
	Then for all $t \geq 0$
	\begin{align*}
			\E\Big[
				F
				\cdot
				e^{
					\frac{t}{\sigma}
					\sum_{\graph \in \copies \backslash A_1}
					B_{\graph}^c
				}
			\Big]
		&\leq
			\E[ F ]
			\cdot
			e^{
				\frac{t}{\sigma}
				\vert A_1 \cup \hat A_2 \vert
			}
			\cdot
			\E [ e^{tW} ]
			.
	\intertext{In particular, there is}
			\E\Big[
				e^{
					\frac{t}{\sigma}
					\sum_{\graph \in \copies \backslash A_1}
					B_{\graph}^c
				}
			\Big]
		&\leq
			e^{
				\frac{t}{\sigma}
				\vert A_1\vert
			}
			\cdot
			\E [ e^{tW} ].
	\end{align*}
\end{lemma}
\begin{proof}
	There is
	\begin{align*}
			\E\Big[
				F
				\cdot
				e^{
					\frac{t}{\sigma}
					\sum_{\graph \in \copies \backslash A_1}
					B_{\graph}^c
				}
			\Big]
		&=
			\E\Big[
				F
				\cdot
				e^{
					\frac{t}{\sigma}
					\sum_{\graph \in \hat A_2 \backslash A_1}
					B_{\graph}^c
				}
				\cdot
				e^{
					\frac{t}{\sigma}
					\sum_{\graph \in \copies \backslash (A_1 \cup \hat A_2)}
					B_{\graph}^c
				}
			\Big]
		\\
		&\leq
			\E\Big[
				F
				\cdot
				e^{
					\frac{t}{\sigma}
					\sum_{\graph \in \hat A_2 \backslash A_1}
					(1-p^{\edges{\G}})
				}
				\cdot
				e^{
					\frac{t}{\sigma}
					\sum_{\graph \in \copies \backslash (A_1 \cup \hat A_2)}
					B_{\graph}^c
				}
			\Big]
		\\
		&=
			\E[ F ]
			\cdot
			e^{
				\frac{t}{\sigma}
				\vert \hat A_2 \backslash A_1 \vert
				(1-p^{\edges{\G}})
			}
			\cdot
			\E\Big[
				e^{
					\frac{t}{\sigma}
					\sum_{\graph \in \copies \backslash (A_1 \cup \hat A_2)}
					B_{\graph}^c
				}
			\Big]
		.
	\end{align*}
	Further,
	\begin{align*}
			\E\Big[
				e^{
					\frac{t}{\sigma}
					\sum_{\graph \in \copies \backslash (A_1 \cup \hat A_2)}
					B_{\graph}^c
				}
			\Big]
		&=
			\E\Big[
				e^{
					-\frac{t}{\sigma}
					\sum_{\graph \in A_1 \cup \hat A_2}
					B_{\graph}^c
				}
				\cdot
				e^{
					tW
				}
			\Big]
		\leq
			e^{
				\frac{t}{\sigma}
				\vert A_1 \cup \hat A_2\vert
				p^{\edges{\G}}
			}
			\cdot
			\E [ e^{tW} ].
	\end{align*}
	This proves the first inequality.
	The second inequality follows from the special case of $F = 1$ with $A_2 = \emptyset$.
\end{proof}

We are now prepared to prove Theorem~\ref{sg:theo:main} and Corollary~\ref{sg:cor:main}.
\begin{proof}[Proof of Theorem~\ref{sg:theo:main}]
	This proof is based on our theoretical results presented in Theorem~\ref{Theorem:MD_Rademacher-short}.
	Since $W$ depends on only finitely many Rademacher random variables
	we may apply this theorem as long as we find suitable functions
	$\gamma_1$ and $\gamma_2$.
	In the application of this theorem,
	the quantity $U = (U_k)_{k\in E}$ with
	\begin{align}
		U_k
		:={}& \frac{1}{\sqrt{pq}} \cdot D_kW \cdot \vert D_k L^{-1} W \vert \nonumber\\
		={}& \frac{1}{\sqrt{pq}} \cdot D_kW \cdot (-D_k L^{-1} W), \label{sg:const:Uk}
	\end{align}
	$k \in E$,
	is of special importance,
	where 
	$\vert D_k L^{-1} W \vert = - D_k L^{-1} W $
	holds true due to the results from Lemma~\ref{sg:lem:LDk}.
	
	Due to the application of the discrete gradient,
	$U_k$ does not depend on $Y_k$.
	For this special case it is known by Corollary~9.9 in \cite{Pr08}
	that $\delta(U) = \sum_{k \in E} Y_k U_k$.
	We further note that
	$\E[ \sqrt{pq}\sum_{k\in E} U_k ] = 1$,
	which is implied by (2.13) in \cite{KRT17} in case of $f$ being the identity.
	
	These considerations allow us to rephrase Theorem~\ref{Theorem:MD_Rademacher}
	with the notation introduced in \eqref{sg:const:Uk}:
	To apply this theorem, we need to construct functions $\gamma_1$ and $\gamma_2$ so that
	\makeatletter\tagsleft@true\makeatother
	\begin{align}
		\E\Big[
			\Big\vert
				\sqrt{pq}\sum_{k\in E} U_k^c
			\Big\vert
			\cdot e^{tW}
		\Big]
		&\leq
			\gamma_1(t) \cdot \E[e^{tW}]
		,
		\tag{A1'}\label{sg:A1}
		\\
		\E\Big[
			\Big\vert
				\sum_{k \in E} Y_k U_k
			\Big\vert
			\cdot e^{tW}
		\Big]
		&\leq
			\gamma_2(t) \cdot \E[e^{tW}]
		,
		\tag{A2'}\label{sg:A2}
	\end{align}
	\makeatletter\tagsleft@false\makeatother
	for $0 \leq t \leq A$, where $A \geq 0$.
	
	For the construction of suitable $\gamma_1$ and $\gamma_2$,
	a further decomposition of $(U_k)_{k\in E}$ will be usefull,
	which can be achieved by application of Lemma~\ref{sg:lem:LDk}:
	We decompose
	\begin{align}
		U_k
		&= \frac{1}{\sqrt{pq}} \cdot D_kW \cdot (-D_k L^{-1} W) \nonumber\\
		&= \frac{1}{\sqrt{pq} \cdot \sigma^2} \cdot
			\sum_{\graph_1,\graph_2 \in \copies}
				D_k B_{\graph_1}^c
				\cdot \big( -D_k L^{-1} B_{\graph_2}^c \big) \nonumber\\
		&= \frac{\sqrt{pq}}{\sigma^2} \cdot
			\sum_{\graph_1,\graph_2 \in \copiesk{k}}
			\sum_{\{k\} \subset \alpha_2 \subset \graph_2 }
				\frac{
					p^{\vert \graph_2 \vert - \vert \alpha_2 \vert}
				}{
					\vert \graph_2 \vert
					\cdot
					\binom{\vert\graph_2\vert - 1}{\vert\alpha_2\vert - 1}
				}
				B_{(\graph_1 \cup \alpha_2) \backslash \{ k \} } \nonumber\\
		&= \frac{\sqrt{pq}}{\sigma^2} \cdot
			\sum_{\graph_1,\graph_2 \in \copiesk{k}}
			\V{k}{1}{2}
			\label{sg:eq:UV}
	\intertext{with}
		\V{k}{1}{2}
		&:= \sum_{\{k\} \subset \alpha_2 \subset \graph_2 }
				\frac{
					p^{\vert \graph_2 \vert - \vert \alpha_2 \vert}
				}{
					\vert \graph_2 \vert
					\cdot
					\binom{\vert\graph_2\vert - 1}{\vert\alpha_2\vert - 1}
				}
				B_{(\graph_1 \cup \alpha_2) \backslash \{ k \} }
		\geq 0.
		\label{sg:const:V}
	\end{align}

	\textbf{Construction of $\gamma_1$}:
	To verify that assumption \eqref{sg:A1} can be fulfilled with some function $\gamma_1$%
	---that is yet to be constructed---%
	we start by applying the Cauchy--Schwarz inequality and using the decomposition of $U_k$ from \eqref{sg:eq:UV}.
	This results in
	\begin{align*}
		\E\Big[
			\Big\vert
				\sqrt{pq}\sum_{k\in E} U_k^c
			\Big\vert
			\cdot e^{tW}
		\Big]
		&\leq
			\E\Big[
				\Big(
					\sqrt{pq}
					\sum_{k\in E} U_k^c
				\Big)^2
				e^{tW}
			\Big]^{\frac12}
			\cdot \E[e^{tW}]^{\frac12}
		\\
		&=
			\Big(
				pq
				\sum_{k,\ell\in E}
				\E\Big[
					U_k^c U_\ell^c
					e^{tW}
				\Big]
			\Big)^{\frac12}
			\cdot \E[e^{tW}]^{\frac12}
		\\
		&=
			\bigg(
				\frac{p^2q^2}{\sigma^4} \cdot
				\sum_{k,\ell\in E}
				\sum_{\graph_1,\graph_2 \in \copiesk{k}}
				\sum_{\graph_3,\graph_4 \in \copiesk{\ell}}
					\E\Big[
						\V{k}{1}{2}^c
						\V{\ell}{3}{4}^c
						e^{tW}
					\Big]
			\bigg)^{\frac12}
			\cdot \E[e^{tW}]^{\frac12}
	\end{align*}
	for all $t \geq 0$.
	Next, we apply the iterated Taylor expansion that was introduced in \eqref{ITE-formula}
	to $e^{tW}$ by choosing
	$x_\graph = \frac{t}{\sigma} B_{\graph}^c$
	and
	$y_\graph = \frac{t}{\sigma}
				\sum_{\graph_5 \in \copiesk{\graph} \backslash
					\copiesk{\graph_1 \cup \graph_2 \cup \graph_3 \cup \graph_4 }}
				B_{\graph_5}^c$
	for $\graph \in I = \copiesk{\graph_1 \cup \graph_2 \cup \graph_3 \cup \graph_4}$,
	and
	$z = \frac{t}{\sigma}
			\sum_{\graph \in \copies \backslash
					\copiesk{\graph_1 \cup \graph_2 \cup \graph_3 \cup \graph_4 }}
				B_{\graph}^c$.
	Due to the underlying dependency structure
	this expansion results in
	\begin{align*}
			\frac{p^2q^2}{\sigma^4} \cdot
			\sum_{k,\ell\in E}
			\sum_{\graph_1,\graph_2 \in \copiesk{k}}
			\sum_{\graph_3,\graph_4 \in \copiesk{\ell}}
				\E\Big[
					\V{k}{1}{2}^c
					\V{\ell}{3}{4}^c
					e^{tW}
				\Big]
		&=
			S_{1,1}(t) + S_{1,2}(t) + S_{1,3}(t) + S_{1,4}(t)
	\end{align*}
	with
	\begin{align*}
		S_{1,1}(t) &:=
			\frac{p^2q^2}{\sigma^4}
			\cdot
			\sum_{\mathclap{\substack{
				k,\ell\in E\\
				\graph_1,\graph_2 \in \copiesk{k}\\
				\graph_3,\graph_4 \in \copiesk{\ell}
			}}}
			\E\Big[
				\V{k}{1}{2}^c
				\V{\ell}{3}{4}^c
			\Big]
			\cdot
			\E\Big[
				e^{
					\frac{t}{\sigma}
					\sum_{\graph \in \copies \backslash
						\copiesk{\graph_1 \cup \graph_2 \cup \graph_3 \cup \graph_4 }}
					B_{\graph}^c
				}
			\Big]
			,\\[1em]
		S_{1,2}(t) &:=
			\frac{p^2q^2}{\sigma^4}
			\cdot
			\frac{t}{\sigma}
			\cdot
			\sum_{\mathclap{\substack{
				k,\ell\in E\\
				\graph_1,\graph_2 \in \copiesk{k}\\
				\graph_3,\graph_4 \in \copiesk{\ell}\\
				\graph_5 \in \copiesk{\graph_1 \cup \graph_2 \cup \graph_3 \cup \graph_4}
			}}}
			\E\Big[
				\V{k}{1}{2}^c
				\V{\ell}{3}{4}^c
				\cdot
				B_{\graph_5}^c
			\Big]
			\cdot
			\E\Big[
				e^{
					\frac{t}{\sigma}
					\sum_{\graph \in \copies \backslash
						\copiesk{\graph_1 \cup \graph_2 \cup \graph_3 \cup \graph_4 \cup \graph_5 }}
					B_{\graph}^c
				}
			\Big]
			,\\[1em]
		S_{1,3}(t) &:=
			\frac{p^2q^2}{\sigma^4}
			\cdot
			\frac{t^2}{\sigma^2}
			\cdot
			\sum_{\mathclap{\substack{
				k,\ell\in E\\
				\graph_1,\graph_2 \in \copiesk{k}\\
				\graph_3,\graph_4 \in \copiesk{\ell}\\
				\graph_5 \in \copiesk{\graph_1 \cup \graph_2 \cup \graph_3 \cup \graph_4} \\
				\graph_6 \in \copiesk{\graph_5} \backslash
					\copiesk{\graph_1 \cup \graph_2 \cup \graph_3 \cup \graph_4}
			}}}
			\E\Bigg[
				\V{k}{1}{2}^c
				\V{\ell}{3}{4}^c
				\cdot
				B_{\graph_5}^c
				B_{\graph_6}^c
				\cdot
				r_1 \Big(
					\frac{t}{\sigma}
					\;\;\;\;\;\;\;\;\;\;\;\;
					\sum_{\mathclap{\graph \in \copiesk{\graph_5} \backslash
						\copiesk{\graph_1 \cup \graph_2 \cup \graph_3 \cup \graph_4 }}}
					B_{\graph}^c
					\;\;\;\;\;\;\;
					\Big)
				\\[-2\baselineskip]&\hspace{9cm}
				\cdot
				e^{
					\frac{t}{\sigma}
					\sum_{\graph \in \copies \backslash
						\copiesk{\graph_1 \cup \graph_2 \cup \graph_3 \cup \graph_4 \cup \graph_5}}
					B_{\graph}^c
				}
			\Bigg]
			,\\[0.5em]
		S_{1,4}(t) &:=
			\frac{p^2q^2}{\sigma^4}
			\cdot
			\frac{t^2}{2\sigma^2}
			\cdot
			\sum_{\mathclap{\substack{
				k,\ell\in E\\
				\graph_1,\graph_2 \in \copiesk{k}\\
				\graph_3,\graph_4 \in \copiesk{\ell}\\
				\graph_5, \graph_6 \in \copiesk{\graph_1 \cup \graph_2 \cup \graph_3 \cup \graph_4 }
			}}}
			\E\Bigg[
				\V{k}{1}{2}^c
				\V{\ell}{3}{4}^c
				\cdot
				B_{\graph_5}^c
				B_{\graph_6}^c
				\cdot
				r_2 \Big(
					\frac{t}{\sigma}
					\;\;\;\;\;\;\;\;
					\sum_{\mathclap{\graph \in
						\copiesk{\graph_1 \cup \graph_2 \cup \graph_3 \cup \graph_4 }}}
					B_{\graph}^c
					\;\;\;
					\Big)
				\\[-1.5\baselineskip]&\hspace{9cm}
				\cdot
				e^{
					\frac{t}{\sigma}
					\sum_{\graph \in \copies \backslash
							\copiesk{\graph_1 \cup \graph_2 \cup \graph_3 \cup \graph_4 }}
					B_{\graph}^c
				}
			\Bigg]
		,
	\end{align*}
	so that
	\begin{align*}
		\E\Big[
			\Big\vert
				\sqrt{pq}\sum_{k\in E} U_k^c
			\Big\vert
			\cdot e^{tW}
		\Big]
		&\leq
			\big(
				S_{1,1}(t) + S_{1,2}(t) + S_{1,3}(t) + S_{1,4}(t)
			\big)^{\frac12}
			\cdot \E[e^{tW}]^{\frac12}
		.
	\end{align*}
	The advantage of this expansion based on \cite{Z19} is
	that $S_{1,1}$ and $S_{1,2}$ each possess a comfortable dependency structure
	that allowed us to split the expectations,
	while $S_{1,3}$ and $S_{1,4}$ are of second order and therefore allow more rough estimates
	without a loss in the resulting rate,
	as we will see later in this proof.
	
	We now have to derive bounds for $S_{1,1}$, $S_{1,2}$, $S_{1,3}$, and $S_{1,4}$.
	
	To derive a bound for $S_{1,1}$,
	we first fix $k,\ell\in E$, $\graph_1,\graph_2 \in \copiesk{k}$, and $\graph_3,\graph_4 \in \copiesk{\ell}$.
	If $\graph_1 \cup \graph_2$ is disjoint from $\graph_3 \cup \graph_4$,
	then $\V{k}{1}{2}$ and $\V{\ell}{3}{4}$ are independent, so that
	$\E\Big[ \V{k}{1}{2}^c \V{\ell}{3}{4}^c \Big] = 0$.
	Otherwise, using our knowledge from Lemma~\ref{sg:lem:correlation}
	about the correlation behavior of the random variables $\{B_A\}_{A \subset E}$
	we see that $\E [ \V{k}{1}{2}^c \V{\ell}{3}{4}^c ] \geq 0$ and
	\begin{align*}
		& \E\Big[
			\V{k}{1}{2}^c
			\V{\ell}{3}{4}^c
		\Big]\\
		&= \sum_{\substack{
					\{k   \} \subset \alpha_2 \subset \graph_2 \\
					\{\ell\} \subset \alpha_4 \subset \graph_4
				}}
				\frac{
					\E[
							B_{\graph_2 \backslash \alpha_2}
					]
				}{
					\vert \graph_2 \vert
					\cdot
					\binom{\vert\graph_2\vert - 1}{\vert\alpha_2\vert - 1}
				}
				\frac{
					\E[
							B_{\graph_4 \backslash \alpha_4}
					]
				}{
					\vert \graph_4 \vert
					\cdot
					\binom{\vert\graph_4\vert - 1}{\vert\alpha_4\vert - 1}
				}
			\E\Big[
					B_{(\graph_1 \cup \alpha_2) \backslash \{ k \} }^c
					B_{(\graph_3 \cup \alpha_4) \backslash \{ \ell \} }^c
			\Big]\\
		&\leq \sum_{\substack{
					\{k   \} \subset \alpha_2 \subset \graph_2 \\
					\{\ell\} \subset \alpha_4 \subset \graph_4
				}}
				\frac{
					1
				}{
					\vert \graph_2 \vert
					\cdot
					\binom{\vert\graph_2\vert - 1}{\vert\alpha_2\vert - 1}
				}
				\frac{
					1
				}{
					\vert \graph_4 \vert
					\cdot
					\binom{\vert\graph_4\vert - 1}{\vert\alpha_4\vert - 1}
				}
			\E\Big[
					B_{
						\big( (\graph_1 \cup \alpha_2) \backslash \{ k \} \big)
							\cup
						(\graph_2 \backslash \alpha_2)
							\cup
						\big( (\graph_3 \cup \alpha_4) \backslash \{ \ell \} \big)
							\cup
						(\graph_4 \backslash \alpha_4)
					}
			\Big]\\
		&\leq \sum_{\substack{
					\{k   \} \subset \alpha_2 \subset \graph_2 \\
					\{\ell\} \subset \alpha_4 \subset \graph_4
				}}
				\frac{
					1
				}{
					\vert \graph_2 \vert
					\cdot
					\binom{\vert\graph_2\vert - 1}{\vert\alpha_2\vert - 1}
				}
				\frac{
					1
				}{
					\vert \graph_4 \vert
					\cdot
					\binom{\vert\graph_4\vert - 1}{\vert\alpha_4\vert - 1}
				}
				p^{\vert \graph_1 \cup \graph_2 \cup \graph_3 \cup \graph_4 \vert - 2}
		\;=\;		p^{\vert \graph_1 \cup \graph_2 \cup \graph_3 \cup \graph_4 \vert - 2},
	\end{align*}
	where in the last step we used what we already noted in Remark~\ref{sg:rem:1}.
	The second inequality from Lemma~\ref{sg:lem:momgenfunc} implies
	\begin{align*}
			\E\Big[
				e^{
					\frac{t}{\sigma}
					\sum_{\graph \in \copies \backslash
						\copiesk{\graph_1 \cup \graph_2 \cup \graph_3 \cup \graph_4 }}
					B_{\graph}^c
				}
			\Big]
		&\leq
			e^{
				\frac{t}{\sigma}
				\vert \copiesk{\graph_1 \cup \graph_2 \cup \graph_3 \cup \graph_4 } \vert
			}
			\cdot
			\E [ e^{tW} ]
		\leq
			e^{
				\frac{4Dt}{\sigma}
			}
			\cdot
			\E [ e^{tW} ]
		.
	\end{align*}
	Putting these partial results together yields $S_{1,1}(t) \geq 0$ and
	\begin{align*}
		S_{1,1}(t)
		&\leq
			\frac{p^2q^2}{\sigma^4}
			\cdot
			\sum_{\mathclap{\substack{
				k,\ell\in E\\
				\graph_1,\graph_2 \in \copiesk{k}\\
				\graph_3,\graph_4 \in \copiesk{\ell}
			}}}
			\1{(\graph_1 \cup \graph_2) \cap (\graph_3 \cup \graph_4) \neq \emptyset}
			p^{\vert \graph_1 \cup \graph_2 \cup \graph_3 \cup \graph_4 \vert - 2}
			\cdot
			e^{
				\frac{4Dt}{\sigma}
			}
			\cdot
			\E [ e^{tW} ]
			\\
		&\leq
			\frac{q^2}{\sigma^4}
			\cdot
			2\sum_{\mathclap{\substack{
				\graph_1 \in \copies\\
				\graph_2 \in \copiesk{\graph_1}\\
				\graph_3 \in \copiesk{\graph_1 \cup \graph_2}\\
				\graph_4 \in \copiesk{\graph_3}
			}}}
			\;\;\;\;\;\;
			\sum_{\mathclap{\substack{ k \in \graph_1 \cap \graph_2 \\ \ell \in \graph_3 \cap \graph_4 }}}
			p^{\vert \graph_1 \cup \graph_2 \cup \graph_3 \cup \graph_4 \vert}
			\cdot
			e^{
				\frac{4Dt}{\sigma}
			}
			\cdot
			\E [ e^{tW} ]
			,
	\end{align*}
	where we used
	that due to symmetry the sum over $\1{(\graph_1 \cup \graph_2) \cap (\graph_3 \cup \graph_4) \neq \emptyset}$
	is smaller than or equal to $2$ times the sum over
	$\1{(\graph_1 \cup \graph_2) \cap \graph_3 \neq \emptyset} = \1{\graph_3 \in \copiesk{\graph_1 \cup \graph_2}}$.
	We further reordered the indices to point out
	that the first sum runs over a subset of all sets of connected copies of size $4$
	in the sense of Definition~\ref{sg:defn:connected},
	while the second sum runs over not more than $\edges{\G}^2$ summands.
	Lemma~\ref{sg:lem:connected} therefore yields
	\begin{align*}
		S_{1,1}(t)
		&\leq
			\frac{q^2}{\sigma^4}
			\cdot
			2
			\cdot
			\frac{ (\vertices{\G}!)^{3} \cdot 2^{6\edges{\G}} }{\aut{\G}^4}
			\cdot
			\frac{n^{4\vertices{\G}} \cdot p^{4\edges{\G}}}{\Psimin^{3}}
			\cdot
			\edges{\G}^2
			\cdot
			e^{
				\frac{4Dt}{\sigma}
			}
			\cdot
			\E [ e^{tW} ]
			\\
		&\leq
			c_{1,1}
			\cdot
			\frac{1}{\Psimin}
			\cdot
			e^{
				\frac{4Dt}{\sigma}
			}
			\cdot
			\E [ e^{tW} ]
	\end{align*}
	with $\displaystyle c_{1,1} :=
		\frac{ 2^{3 + 6\edges{\G}} \cdot (\vertices{\G}!)^5 \cdot \edges{\G}^2}{\aut{\G}^2}$.
	For the last step we used the lower bound for $\sigma$ presented in \eqref{sg:ineq:sigma},
	which we may use since we assume that $n \geq 4\vertices{\G}^2$.
	
	The calculations for $S_{1,2}$ are largely analogous,
	except for we have to use part $(iii)$ instead of part $(ii)$ from Lemma~\ref{sg:lem:correlation}.
	We find
	\begin{align*}
		\vert S_{1,2}(t) \vert
		&\leq
			\frac{p^2q^2}{\sigma^4}
			\cdot
			\frac{t}{\sigma}
			\cdot
			2\sum_{\mathclap{\substack{
				k,\ell\in E\\
				\graph_1,\graph_2 \in \copiesk{k}\\
				\graph_3,\graph_4 \in \copiesk{\ell}\\
				\graph_5 \in \copiesk{\graph_3 \cup \graph_4}
			}}}
			\E\Big[ \Big\vert
				\V{k}{1}{2}^c
				\V{\ell}{3}{4}^c
				\cdot
				B_{\graph_5}^c
			\big\vert \Big]
			\cdot
			\E\Big[
				e^{
					\frac{t}{\sigma}
					\sum_{\graph \in \copies \backslash
						\copiesk{\graph_1 \cup \graph_2 \cup \graph_3 \cup \graph_4 \cup \graph_5 }}
					B_{\graph}^c
				}
			\Big]\\
		&\leq
			\frac{p^2q^2}{\sigma^4}
			\cdot
			\frac{t}{\sigma}
			\cdot
			2\sum_{\mathclap{\substack{
				k,\ell\in E\\
				\graph_1,\graph_2 \in \copiesk{k}\\
				\graph_3,\graph_4 \in \copiesk{\ell}\\
				\graph_5 \in \copiesk{\graph_3 \cup \graph_4}
			}}}
			\1{(\graph_1 \cup \graph_2) \cap (\graph_3 \cup \graph_4 \cup \graph_5) \neq \emptyset}
			2^3
			p^{\vert \graph_1 \cup \graph_2 \cup \graph_3 \cup \graph_4 \cup \graph_5 \vert - 2}
			\cdot
			e^{
				\frac{t}{\sigma}
				\vert \copiesk{\graph_1 \cup \graph_2 \cup \graph_3 \cup \graph_4 \cup \graph_5 } \vert
			}
			\cdot
			\E [ e^{tW} ]
			\\
		&\leq
			\frac{q^2}{\sigma^4}
			\cdot
			\frac{t}{\sigma}
			\cdot
			2^4\bigg(\;
				2\sum_{\mathclap{\substack{
					\graph_1 \in \copies\\
					\graph_2 \in \copiesk{\graph_1}\\
					\graph_3 \in \copiesk{\graph_1 \cup \graph_2}\\
					\graph_4 \in \copiesk{\graph_3}\\
					\graph_5 \in \copiesk{\graph_3 \cup \graph_4}
				}}}
				\;\;\;\;\;\;
				\sum_{\mathclap{\substack{ k \in \graph_1 \cap \graph_2 \\ \ell \in \graph_3 \cap \graph_4 }}}
				p^{\vert \graph_1 \cup \graph_2 \cup \graph_3 \cup \graph_4 \cup \graph_5 \vert}
			+
				2\sum_{\mathclap{\substack{
					\graph_1 \in \copies\\
					\graph_2 \in \copiesk{\graph_1}\\
					\graph_5 \in \copiesk{\graph_1 \cup \graph_2}\\
					\graph_3 \in \copiesk{\graph_5}\\
					\graph_4 \in \copiesk{\graph_3}
				}}}
				\;\;\;\;\;\;
				\sum_{\mathclap{\substack{ k \in \graph_1 \cap \graph_2 \\ \ell \in \graph_3 \cap \graph_4 }}}
				p^{\vert \graph_1 \cup \graph_2 \cup \graph_3 \cup \graph_4 \cup \graph_5 \vert}
			\bigg)
			\cdot
			e^{
				\frac{5Dt}{\sigma}
			}
			\cdot
			\E [ e^{tW} ]
			\\
		&\leq
			\frac{q^2}{\sigma^4}
			\cdot
			\frac{t}{\sigma}
			\cdot
			2^6
			\cdot
			\frac{ (\vertices{\G}!)^{4} \cdot 2^{10\edges{\G}} }{\aut{\G}^5}
			\cdot
			\frac{n^{5\vertices{\G}} \cdot p^{5\edges{\G}}}{\Psimin^{4}}
			\cdot
			\edges{\G}^2
			\cdot
			e^{
				\frac{5Dt}{\sigma}
			}
			\cdot
			\E [ e^{tW} ]
			\\
		&\leq
			c_{1,2}
			\cdot
			\frac{t}{\sqrt{q\Psimin}}
			\cdot
			\frac{1}{\Psimin}
			\cdot
			e^{
				\frac{5Dt}{\sigma}
			}
			\cdot
			\E [ e^{tW} ]
	\end{align*}
	with $\displaystyle c_{1,2} :=
		\frac{ 2^{\frac{17}2 + 10\edges{\G}} \cdot (\vertices{\G}!)^{\frac{13}2} \cdot \edges{\G}^2}{\aut{\G}^\frac52}$.
	
	To construct a bound for $S_{1,3}$ we have to slightly modify our approach.
	Due to symmetry, we can rename the variables if necessary to make sure that $\graph_5 \in \copiesk{\graph_3 \cup \graph_4}$.
	However, due to the influence of the remainder function $r_1$
	the relevant expectation can be non-zero even if $\graph_1 \cup \graph_2$ and $\graph_3 \cup \graph_4 \cup \graph_5 \cup \graph_6$ are disjoint.
	This disadvantage can later be compensated for by using the additional $\sigma^{-2}$.
	We start by taking the absolute to get rid of the remainder $r_1$.
	Recall that $t \geq 0$ and $r_1(x) \leq e^{\vert x \vert}$ for all $x \in \R$,
	so that
	\begin{align*}
			\Big\vert
			r_1 \Big(
				\frac{t}{\sigma}
				\;\;\;\;\;\;\;\;\;\;\;\;\;
				\sum_{\mathclap{\graph \in \copiesk{\graph_5} \backslash
					\copiesk{\graph_1 \cup \graph_2 \cup \graph_3 \cup \graph_4 }}}
				B_{\graph}^c
				\;\;\;\;\;\;\;
				\Big)
			\Big\vert
		&\leq
			e^{
				\frac{Dt}{\sigma}
			}
		.
	\end{align*}
	Hence, by application of
	Lemma~\ref{sg:lem:momgenfunc} and
	Lemma~\ref{sg:lem:correlation},
	there is
	\begin{align*}
		\vert S_{1,3}(t) \vert
		 &\leq
			\frac{p^2q^2}{\sigma^4}
			\cdot
			\frac{t^2}{\sigma^2}
			\cdot
			2\sum_{\mathclap{\substack{
				k,\ell\in E\\
				\graph_1,\graph_2 \in \copiesk{k}\\
				\graph_3,\graph_4 \in \copiesk{\ell}\\
				\graph_5 \in \copiesk{\graph_3 \cup \graph_4} \\
				\graph_6 \in \copiesk{\graph_5} \backslash
					\copiesk{\graph_1 \cup \graph_2 \cup \graph_3 \cup \graph_4}
			}}}
			\E\bigg[
				\big\vert
				\V{k}{1}{2}^c
				\V{\ell}{3}{4}^c
				\cdot
				B_{\graph_5}^c
				B_{\graph_6}^c
				\big\vert
				\cdot
				e^{
					\frac{Dt}{\sigma}
				}
				\cdot
				e^{
					\frac{t}{\sigma}
					\sum_{\graph \in \copies \backslash
						\copiesk{\graph_1 \cup \graph_2 \cup \graph_3 \cup \graph_4 \cup \graph_5}}
					B_{\graph}^c
				}
			\bigg]
			\\
		 &\leq
			\frac{p^2q^2}{\sigma^4}
			\cdot
			\frac{t^2}{\sigma^2}
			\cdot
			2\sum_{\mathclap{\substack{
				\graph_1 \in \copies\\
				\graph_2 \in \copiesk{\graph_1}\\
				\graph_3 \in \copies\\
				\graph_4 \in \copiesk{\graph_3}\\
				\graph_5 \in \copiesk{\graph_3 \cup \graph_4} \\
				\graph_6 \in \copiesk{\graph_5} \backslash
					\copiesk{\graph_1 \cup \graph_2 \cup \graph_3 \cup \graph_4}
			}}}
			\;\;\;\;\;\;
			\sum_{\mathclap{\substack{ k \in \graph_1 \cap \graph_2 \\ \ell \in \graph_3 \cap \graph_4 }}}
			\E\Big[
				\big\vert
				\V{k}{1}{2}^c
				\V{\ell}{3}{4}^c
				\cdot
				B_{\graph_5}^c
				B_{\graph_6}^c
				\big\vert
			\Big]
			\cdot
			e^{
				\frac{Dt}{\sigma}
			}
			\cdot
			e^{
				\frac{6Dt}{\sigma}
			}
			\cdot
			\E[e^{tW}]
			\\
		 &\leq
			\frac{q^2}{\sigma^4}
			\cdot
			\frac{t^2}{\sigma^2}
			\cdot
			2\sum_{\mathclap{\substack{
				\graph_1 \in \copies\\
				\graph_2 \in \copiesk{\graph_1}\\
				\graph_3 \in \copies\\
				\graph_4 \in \copiesk{\graph_3}\\
				\graph_5 \in \copiesk{\graph_3 \cup \graph_4} \\
				\graph_6 \in \copiesk{\graph_5}
			}}}
			\edges{\G}^2
			2^{4}
			p^{\vert \graph_1 \cup \graph_2 \cup \graph_3 \cup \graph_4 \cup \graph_5 \cup \graph_6 \vert}
			\cdot
			e^{
				\frac{7Dt}{\sigma}
			}
			\cdot
			\E[e^{tW}]
		.
	\end{align*}
	As already mentioned, we can not be sure that
	$\graph_1 \cup \graph_2$ and $\graph_3 \cup \graph_4 \cup \graph_5 \cup \graph_6$ are not disjoint.
	Therefore, in the sum above $\graph_3$ runs over the the complete set $\copies$,
	so that $(\graph_1,\graph_2)$ and $(\graph_3,\graph_4,\graph_5,\graph_6)$ are sets of connected copies
	while in general their union $(\graph_1,\graph_2,\graph_3,\graph_4,\graph_5,\graph_6)$ is not.
	This case is covered by the second inequality from Lemma~\ref{sg:lem:connected},
	so that
	\begin{align*}
		\vert S_{1,3}(t) \vert
		 &\leq
			\frac{q^2}{\sigma^4}
			\cdot
			\frac{t^2}{\sigma^2}
			\cdot
			2^{5}
			\cdot
			\frac{ (\vertices{\G}!)^{5} \cdot 2^{15\edges{\G}} }{\aut{\G}^6}
			\cdot
			\frac{n^{6\vertices{\G}} \cdot p^{6\edges{\G}}}{\Psimin^{4} \cdot \min\{ \Psimin , 1 \} }
			\cdot
			\edges{\G}^2
			\cdot
			e^{
				\frac{7Dt}{\sigma}
			}
			\cdot
			\E[e^{tW}]
			\\
		&\leq
			c_{1,3}
			\cdot
			\frac{t^2}{q \min\{ \Psimin , 1 \}}
			\cdot
			\frac{1}{\Psimin}
			\cdot
			e^{
				\frac{7Dt}{\sigma}
			}
			\cdot
			\E [ e^{tW} ]
	\end{align*}
	with $\displaystyle c_{1,3} :=
		\frac{ 2^{8 + 15\edges{\G}} \cdot (\vertices{\G}!)^8 \cdot \edges{\G}^2}{\aut{\G}^3}$.
		
	The calculations for $S_{1,4}$ are largely analogous to $S_{1,3}$ again.
	We get
	\begin{align*}
		\vert S_{1,4}(t) \vert
		&\leq
			\frac{p^2q^2}{\sigma^4}
			\cdot
			\frac{t^2}{2\sigma^2}
			\cdot
			\sum_{\mathclap{\substack{
				k,\ell\in E\\
				\graph_1,\graph_2 \in \copiesk{k}\\
				\graph_3,\graph_4 \in \copiesk{\ell}\\
				\graph_5, \graph_6 \in \copiesk{\graph_1 \cup \graph_2 \cup \graph_3 \cup \graph_4 }
			}}}
			\E\bigg[
				\big\vert
				\V{k}{1}{2}^c
				\V{\ell}{3}{4}^c
				\cdot
				B_{\graph_5}^c
				B_{\graph_6}^c
				\big\vert
				\cdot
				e^{
					\frac{4Dt}{\sigma}
				}
				\cdot
				e^{
					\frac{t}{\sigma}
					\sum_{\graph \in \copies \backslash
						\copiesk{\graph_1 \cup \graph_2 \cup \graph_3 \cup \graph_4}}
					B_{\graph}^c
				}
			\bigg]
			\\
		 &\leq
			\frac{p^2q^2}{\sigma^4}
			\cdot
			\frac{t^2}{2\sigma^2}
			\cdot
			\sum_{\mathclap{\substack{
				\graph_1 \in \copies\\
				\graph_2 \in \copiesk{\graph_1}\\
				\graph_3 \in \copies\\
				\graph_4 \in \copiesk{\graph_3}\\
				\graph_5, \graph_6 \in \copiesk{\graph_1 \cup \graph_2 \cup \graph_3 \cup \graph_4 }
			}}}
			\;\;\;\;\;\;
			\sum_{\mathclap{\substack{ k \in \graph_1 \cap \graph_2 \\ \ell \in \graph_3 \cap \graph_4 }}}
			\E\Big[
				\big\vert
				\V{k}{1}{2}^c
				\V{\ell}{3}{4}^c
				\cdot
				B_{\graph_5}^c
				B_{\graph_6}^c
				\big\vert
			\Big]
			\cdot
			e^{
				\frac{4Dt}{\sigma}
			}
			\cdot
			e^{
				\frac{6Dt}{\sigma}
			}
			\cdot
			\E[e^{tW}]
			\\
		 &\leq
			\frac{q^2}{\sigma^4}
			\cdot
			\frac{t^2}{2\sigma^2}
			\bigg(
				2\sum_{\mathclap{\substack{
					\graph_1 \in \copies\\
					\graph_2 \in \copiesk{\graph_1}\\
					\graph_3 \in \copies\\
					\graph_4 \in \copiesk{\graph_3}\\
					\graph_5, \graph_6 \in \copiesk{\graph_3 \cup \graph_4 }
				}}}
				\edges{\G}^2
				2^{4}
				p^{\vert \graph_1 \cup \graph_2 \cup \graph_3 \cup \graph_4 \cup \graph_5 \cup \graph_6 \vert}
			+
				2\sum_{\mathclap{\substack{
					\graph_1 \in \copies\\
					\graph_2 \in \copiesk{\graph_1}\\
					\graph_3 \in \copies\\
					\graph_4 \in \copiesk{\graph_3}\\
					\graph_5 \in \copiesk{\graph_1 \cup \graph_2 }\\
					\graph_6 \in \copiesk{\graph_3 \cup \graph_4 }
				}}}
				\edges{\G}^2
				2^{4}
				p^{\vert \graph_1 \cup \graph_2 \cup \graph_3 \cup \graph_4 \cup \graph_5 \cup \graph_6 \vert}
			\bigg)
			e^{
				\frac{10Dt}{\sigma}
			}
			\E[e^{tW}]
			\\
		 &\leq
			\frac{q^2}{\sigma^4}
			\cdot
			\frac{t^2}{\sigma^2}
			\cdot
			2^{5}
			\cdot
			\frac{ (\vertices{\G}!)^{5} \cdot 2^{15\edges{\G}} }{\aut{\G}^6}
			\cdot
			\frac{n^{6\vertices{\G}} \cdot p^{6\edges{\G}}}{\Psimin^{4} \cdot \min\{ \Psimin , 1 \} }
			\cdot
			\edges{\G}^2
			\cdot
			e^{
				\frac{10Dt}{\sigma}
			}
			\cdot
			\E[e^{tW}]
			\\
		&\leq
			c_{1,4}
			\cdot
			\frac{t^2}{q \min\{ \Psimin , 1 \}}
			\cdot
			\frac{1}{\Psimin}
			\cdot
			e^{
				\frac{10Dt}{\sigma}
			}
			\cdot
			\E [ e^{tW} ]
	\end{align*}
	with $\displaystyle c_{1,4} := c_{1,3} =
		\frac{ 2^{8 + 15\edges{\G}} \cdot (\vertices{\G}!)^8 \cdot \edges{\G}^2}{\aut{\G}^3}$.
	
	Having bounded $S_{1,1}$, $S_{1,2}$, $S_{1,3}$, and $S_{1,4}$,
	we note that $c_{1,3} = \max \{ c_{1,1}, c_{1,2}, c_{1,3}, c_{1,4} \}$
	because of $\aut{\G} \leq \vertices{\G}!$.
	In conclusion,
	for all $t \geq 0$ it holds that
	\begin{align*}
		\E\Big[
			\Big\vert
				\sqrt{pq}\sum_{k\in E} U_k^c
			\Big\vert
			\cdot e^{tW}
		\Big]
		&\leq
			\big(
				S_{1,1}(t) + S_{1,2}(t) + S_{1,3}(t) + S_{1,4}(t)
			\big)^{\frac12}
			\cdot \E[e^{tW}]^{\frac12}
			\\
		&\leq
			\bigg(
				c_{1,1}
				\cdot
				\frac{1}{\Psimin}
				\cdot
				e^{
					\frac{4Dt}{\sigma}
				}
				\cdot
				\E [ e^{tW} ]
			\\&\hspace{1em}+
				c_{1,2}
				\cdot
				\frac{t}{\sqrt{q\Psimin}}
				\cdot
				\frac{1}{\Psimin}
				\cdot
				e^{
					\frac{5Dt}{\sigma}
				}
				\cdot
				\E [ e^{tW} ]
			\\&\hspace{1em}+
				c_{1,3}
				\cdot
				\frac{t^2}{q \min\{ \Psimin , 1 \}}
				\cdot
				\frac{1}{\Psimin}
				\cdot
				e^{
					\frac{7Dt}{\sigma}
				}
				\cdot
				\E [ e^{tW} ]
			\\&\hspace{1em}+
				c_{1,4}
				\cdot
				\frac{t^2}{q \min\{ \Psimin , 1 \}}
				\cdot
				\frac{1}{\Psimin}
				\cdot
				e^{
					\frac{10Dt}{\sigma}
				}
				\cdot
				\E [ e^{tW} ]
			\bigg)^{\frac12}
			\cdot \E[e^{tW}]^{\frac12}
			\\
		&\leq
			\sqrt{c_{1,3}} \cdot
			\bigg(
				1
				+
				\frac{t}{\sqrt{\Psimin}}
				+
				2
				\frac{t^2}{\min\{ \Psimin , 1 \}}
			\bigg)^\frac12
			\frac{1}{\sqrt{q\Psimin}}
			\cdot
			e^{
				\frac{5Dt}{\sigma}
			}
			\cdot
			\E [ e^{tW} ]
			\\
		&\leq
			\gamma_1(t)
			\cdot
			\E [ e^{tW} ]
	\end{align*}
	with
	\begin{align}
		\gamma_1(t) :={}&
			c_{A1}
			\cdot
			\bigg(
				1 + \frac{t}{\min\{\sqrt{\Psimin},1\}}
			\bigg)
			\frac{1}{\sqrt{q\Psimin}}
			\cdot
			e^{
				\frac{5Dt}{\sigma}
			}
			\label{sg:fct:g1}
			,
	\end{align}
	where
	$\displaystyle c_{A1}
	:= 
		\sqrt{2c_{1,3}}
	=
		\frac{ 2^{\frac92 + \frac{15}2\edges{\G}} \cdot (\vertices{\G}!)^4 \cdot \edges{\G}}{\aut{\G}^\frac32}$.
	
	\textbf{Construction of $\gamma_2$}:
	To verify that assumption \eqref{sg:A2} can be fulfilled with some function $\gamma_2$
	we again start by applying the Cauchy--Schwarz inequality
	and using the decomposition of $U_k$ from \eqref{sg:eq:UV},
	which results in
	\begin{align*}
		\E\Big[
			\Big\vert
				\sum_{k \in E} Y_k U_k
			\Big\vert
			\cdot e^{tW}
		\Big]
		&\leq
			\E\Big[
				\Big(
					\sum_{k \in E} Y_k U_k
				\Big)^2
				e^{tW}
			\Big]^{\frac12}
			\cdot \E[e^{tW}]^{\frac12}
		\\
		&=
			\Big(
				\sum_{k,\ell \in E} 
					\E\Big[ Y_k Y_\ell U_k U_\ell e^{tW}\Big]
			\Big)^{\frac12}
			\cdot \E[e^{tW}]^{\frac12}
		\\
		&=
			\bigg(
				\frac{1}{\sigma^4} \cdot
				\sum_{k,\ell\in E}
				\sum_{\substack{\graph_1,\graph_2 \in \copiesk{k} \\ \graph_3,\graph_4 \in \copiesk{\ell}}}
					\E\Big[
						B_k^c
						B_\ell^c
						\cdot
						\V{k}{1}{2}
						\V{\ell}{3}{4}
						\cdot
						e^{tW}
					\Big]
			\bigg)^{\frac12}
			\cdot \E[e^{tW}]^{\frac12}
	\end{align*}
	for all $t \geq 0$.
	We split the sum over $k,\ell \in E$ into two parts.
	The first part
	\begin{align*}
		S_{2,1}(t) &:=
			\frac{1}{\sigma^4}
			\cdot
			\sum_{\mathclap{\substack{
				k\in E\\
				\graph_1,\graph_2,\graph_3,\graph_4 \in \copiesk{k}
			}}}
				\E\Big[
					(B_k^c)^2
					\cdot
					\V{k}{1}{2}
					\V{k}{3}{4}
					\cdot
					e^{tW}
				\Big]
	\intertext{
	contains all terms in which there is $k = \ell$.
	In the second part, in which there is $k \neq \ell$,
	we decompose $e^{tW} = 
				e^{
					\frac{t}{\sigma}
					\sum_{\graph \in \copiesk{\{k,\ell\}}}
					B_{\graph}^c
				}
				\cdot
				e^{
					\frac{t}{\sigma}
					\sum_{\graph \in \copies \backslash	\copiesk{\{k,\ell\}}}
					B_{\graph}^c
				}$
	and we use the Taylor expansion on
	$e^{ \frac{t}{\sigma} \sum_{\graph \in \copiesk{\{k,\ell\}}} B_{\graph}^c} $.
	This yields the three terms
	}
		S_{2,2}(t) &:=
			\frac{1}{\sigma^4}
			\cdot
			\;\;
			\sum_{\mathclap{\substack{
				k,\ell\in E, k \neq \ell\\
				\graph_1,\graph_2 \in \copiesk{k}\\
				\graph_3,\graph_4 \in \copiesk{\ell}
			}}}
			\;
			\E\Big[
				B_k^c
				B_\ell^c
				\cdot
				\V{k}{1}{2}
				\V{\ell}{3}{4}
				\cdot
				e^{
					\frac{t}{\sigma}
					\sum_{\graph \in \copies \backslash	\copiesk{\{k,\ell\}}}
					B_{\graph}^c
				}
			\Big]
			,\\[1em]
		S_{2,3}(t) &:=
			\frac{1}{\sigma^4}
			\cdot
			\frac{t}{\sigma}
			\cdot
			\;\;
			\sum_{\mathclap{\substack{
				k,\ell\in E, k \neq \ell\\
				\graph_1,\graph_2 \in \copiesk{k}\\
				\graph_3,\graph_4 \in \copiesk{\ell}\\
				\graph_5 \in \copiesk{\{k,\ell\}}
			}}}
			\;
			\E\Bigg[
				B_k^c
				B_\ell^c
				\cdot
				\V{k}{1}{2}
				\V{\ell}{3}{4}
				\cdot
				B_{\graph_5}^c
				\cdot
				e^{
					\frac{t}{\sigma}
					\sum_{\graph \in \copies \backslash	\copiesk{\{k,\ell\}}}
					B_{\graph}^c
				}
			\Bigg]
			,\\[1em]
		S_{2,4}(t) &:=
			\frac{1}{\sigma^4}
			\cdot
			\frac{t^2}{2\sigma^2}
			\cdot
			\;
			\sum_{\mathclap{\substack{
				k,\ell\in E\, k \neq \ell\\
				\graph_1,\graph_2 \in \copiesk{k}\\
				\graph_3,\graph_4 \in \copiesk{\ell}\\
				\graph_5, \graph_6 \in \copiesk{\{k,\ell\}}
			}}}
			\;
			\E\Bigg[
				B_k^c
				B_\ell^c
				\cdot
				\V{k}{1}{2}
				\V{\ell}{3}{4}
				\cdot
				B_{\graph_5}^c
				B_{\graph_6}^c
				\cdot
				r_2 \Big(
					\frac{t}{\sigma}
					\;\;\;\;
					\sum_{\mathclap{\graph \in \copiesk{\{k,\ell\}}}}
					B_{\graph}^c
					\Big)
				\cdot
				e^{
					\frac{t}{\sigma}
					\sum_{\graph \in \copies \backslash	\copiesk{\{k,\ell\}}}
					B_{\graph}^c
				}
			\Bigg]
		,
	\end{align*}
	so that
	\begin{align*}
		\E\Big[
			\Big\vert
				\sum_{k \in E} Y_k U_k
			\Big\vert
			\cdot e^{tW}
		\Big]
		&\leq
			\big(
				S_{2,1}(t) + S_{2,2}(t) + S_{2,3}(t) + S_{2,4}(t)
			\big)^{\frac12}
			\cdot \E[e^{tW}]^{\frac12}
		.
	\end{align*}
	
	To derive a bound for $S_{2,1}$,
	we first note that $\V{k}{1}{2}$ and $\V{k}{3}{4}$ are non-negative
	according to \eqref{sg:const:V}, so that
	$S_{2,1}(t) \geq 0$.
	An upper bound for $S_{2,1}(t)$ is derived by
	application of the first inequality from Lemma~\ref{sg:lem:momgenfunc}.
	We then use that $\V{k}{1}{2}$ and $\V{k}{3}{4}$ are independent from $B_k$
	according to \eqref{sg:const:V},
	and we finish by applying
	Lemma~\ref{sg:lem:correlation}, Remark~\ref{sg:rem:1}, and Lemma~\ref{sg:lem:connected}.
	This yields
	\begin{align*}
		S_{2,1}(t)
		&\leq
			\frac{1}{\sigma^4}
			\cdot
			\sum_{\mathclap{\substack{
				k\in E\\
				\graph_1,\graph_2,\graph_3,\graph_4 \in \copiesk{k}
			}}}
				\E\Big[
					(B_k^c)^2
					\cdot
					\V{k}{1}{2}
					\V{k}{3}{4}
				\Big]
				\cdot
				e^{
					\frac{t}{\sigma}
					\vert \copiesk{ \{k\} \cup \graph_1 \cup \graph_2 \cup \graph_3 \cup \graph_4 } \vert
				}
				\cdot
				\E [ e^{tW} ]\\
		&\leq
			\frac{1}{\sigma^4}
			\cdot
			\sum_{\mathclap{\substack{
				k\in E\\
				\graph_1,\graph_2,\graph_3,\graph_4 \in \copiesk{k}
			}}}
				\E\big[
					(B_k^c)^2
				\big]
				\cdot				
				\E\big[
					\V{k}{1}{2}
					\V{k}{3}{4}
				\big]
				\cdot
				e^{
					\frac{4Dt}{\sigma}
				}
				\cdot
				\E [ e^{tW} ]
			\\
		&\leq
			\frac{1}{\sigma^4}
			\cdot
			\sum_{\mathclap{\substack{
				k\in E\\
				\graph_1,\graph_2,\graph_3,\graph_4 \in \copiesk{k}
			}}}
				pq
				\cdot	
				\;\;\;
				\sum_{\substack{
						\{k\} \subset \alpha_2 \subset \graph_2 \\
						\{k\} \subset \alpha_4 \subset \graph_4
					}}
					\frac{
						\E[
								B_{\graph_2 \backslash \alpha_2}
						]
					}{
						\vert \graph_2 \vert
						\binom{\vert\graph_2\vert - 1}{\vert\alpha_2\vert - 1}
					}
					\frac{
						\E[
								B_{\graph_4 \backslash \alpha_4}
						]
					}{
						\vert \graph_4 \vert
						\binom{\vert\graph_4\vert - 1}{\vert\alpha_4\vert - 1}
					}
				\E\Big[
						B_{(\graph_1 \cup \alpha_2) \backslash \{ k \} }
						B_{(\graph_3 \cup \alpha_4) \backslash \{ k\} }
				\Big]
				e^{
					\frac{4Dt}{\sigma}
				}
				\E [ e^{tW} ]
			\\
		&\leq
			\frac{1}{\sigma^4}
			\cdot
			\sum_{\mathclap{\substack{
				\graph_1 \in \copies\\
				\graph_2,\graph_3,\graph_4 \in \copiesk{\graph_1}
			}}}
			\;\;\;\;\;
			\sum_{\mathclap{k\in\copiesk{\graph_1}}}
				pq
				\cdot	
				p^{\vert \graph_1 \cup \graph_2 \cup \graph_3 \cup \graph_4 \vert - 1}
				\cdot
				e^{
					\frac{4Dt}{\sigma}
				}
				\cdot
				\E [ e^{tW} ]
			\\
		&\leq
			\frac{q}{\sigma^4}
			\cdot
			\edges{\G}
			\cdot	
			\frac{ (\vertices{\G}!)^{3} \cdot 2^{6\edges{\G}} }{\aut{\G}^4}
			\cdot
			\frac{n^{4\vertices{\G}} \cdot p^{4\edges{\G}}}{\Psimin^{3}}
			\cdot
			e^{
				\frac{4Dt}{\sigma}
			}
			\cdot
			\E [ e^{tW} ]
			\\
		&\leq
			c_{2,1}
			\cdot
			\frac{1}{q\Psimin}
			\cdot
			e^{
				\frac{4Dt}{\sigma}
			}
			\cdot
			\E [ e^{tW} ]
	\end{align*}
	with $\displaystyle c_{2,1} :=
		\frac{ 2^{2 + 6\edges{\G}} \cdot (\vertices{\G}!)^5 \cdot \edges{\G}}{\aut{\G}^2}$.
	For the last step we again used the lower bound for $\sigma$ presented in \eqref{sg:ineq:sigma}.
	
	Regarding $S_{2,2}$, we first focus on the random variables inside the expectation.
	Let $k,\ell\in E$ with $k \neq \ell$,
	$\graph_1,\graph_2 \in \copiesk{k}$ and $\graph_3,\graph_4 \in \copiesk{\ell}$.
	$\V{k}{1}{2}$ and
	$e^{	\frac{t}{\sigma}
		\sum_{\graph \in \copies \backslash	\copiesk{\{k,\ell\}}}
		B_{\graph}^c
	}$
	are independent of $B_k^c$.
	If $k \not\in \graph_3 \cup \graph_4$, then $\V{\ell}{3}{4}$ is independent of $B_k^c$, too.
	In this case, we could split the expectation and obtain $\E[B_k^c]=0$ as a factor.
	In the opposite case, in which $k \in \graph_3 \cup \graph_4$,
	we only have to regard those summands of $\V{\ell}{3}{4}$ that contain $B_k$ as a factor.
	However, then there is $B_k^c \cdot B_k = q \cdot B_k$.
	Hence, $S_{2,2}(t) \geq 0$.
	Further, application of our main technical
	Lemmas~\ref{sg:lem:momgenfunc}, \ref{sg:lem:correlation}, \ref{sg:lem:connected},
	as well as inequality \eqref{sg:ineq:sigma}
	leads to
	\begin{align*}
		S_{2,2}(t)
		&\leq
			\frac{1}{\sigma^4}
			\cdot
			\sum_{\mathclap{\substack{
				k,\ell\in E, k \neq \ell\\
				\graph_1,\graph_2 \in \copiesk{k}\\
				\graph_3,\graph_4 \in \copiesk{\ell}
			}}}
			\;
			\E\Big[
				q^2
				\cdot
				B_k
				B_\ell
				\cdot
				\V{k}{1}{2}
				\V{\ell}{3}{4}
				\cdot
				e^{
					\frac{t}{\sigma}
					\sum_{\graph \in \copies \backslash	\copiesk{\{k,\ell\}}}
					B_{\graph}^c
				}
			\Big]
			\cdot
			\1{k \in \graph_3 \cup \graph_4}
			\1{\ell \in \graph_1 \cup \graph_2}
			\\
		&\leq
			\frac{q^2}{\sigma^4}
			\cdot
			\sum_{\mathclap{\substack{
				k,\ell\in E, k \neq \ell\\
				\graph_1,\graph_2 \in \copiesk{k}\\
				\graph_3,\graph_4 \in \copiesk{\ell}
			}}}
			\;
			\E\Big[
				B_k
				B_\ell
				\cdot
				\V{k}{1}{2}
				\V{\ell}{3}{4}
			\Big]
			\cdot
			e^{
				\frac{t}{\sigma}
				\vert \copiesk{\{k,\ell\} \cup \graph_1 \cup \graph_2 \cup \graph_3 \cup \graph_4 } \vert
			}
			\cdot
			\E [ e^{tW} ]
			\cdot
			\1{k \in \graph_3 \cup \graph_4}
			\1{\ell \in \graph_1 \cup \graph_2}
			\\
		&\leq
			\frac{q^2}{\sigma^4}
			\cdot
			2^2 \sum_{\mathclap{\substack{
				k,\ell\in E, k \neq \ell\\
				\graph_1,\graph_3 \in \copiesk{k} \cap \copiesk{\ell}\\
				\graph_2 \in \copiesk{k}\\
				\graph_4 \in \copiesk{\ell}
			}}}
			p^{\vert \graph_1 \cup \graph_2 \cup \graph_3 \cup \graph_4 \vert}
			\cdot
			e^{
				\frac{4Dt}{\sigma}
			}
			\cdot
			\E [ e^{tW} ]
			\\
		&\leq
			\frac{q^2}{\sigma^4}
			\cdot
			4\sum_{\mathclap{\substack{
				\graph_1 \in \copies\\
				\graph_2,\graph_3,\graph_4 \in \copiesk{\graph_1}
			}}}
			\;\;\;\;\;\;\;\;\;\;\;\;\;
			\sum_{\mathclap{\substack{
				k \in \graph_1 \cap \graph_2 \cap \graph_3\\
				\ell \in \graph_1 \cap \graph_3 \cap \graph_4
			}}}
			p^{\vert \graph_1 \cup \graph_2 \cup \graph_3 \cup \graph_4 \vert}
			\cdot
			e^{
				\frac{4Dt}{\sigma}
			}
			\cdot
			\E [ e^{tW} ]
			\\
		&\leq
			\frac{q^2}{\sigma^4}
			\cdot
			4
			\cdot
			\edges{\G}^2
			\cdot	
			\frac{ (\vertices{\G}!)^{3} \cdot 2^{6\edges{\G}} }{\aut{\G}^4}
			\cdot
			\frac{n^{4\vertices{\G}} \cdot p^{4\edges{\G}}}{\Psimin^{3}}
			\cdot
			e^{
				\frac{4Dt}{\sigma}
			}
			\cdot
			\E [ e^{tW} ]
			\\
		&\leq
			c_{2,2}
			\cdot
			\frac{1}{\Psimin}
			\cdot
			e^{
				\frac{4Dt}{\sigma}
			}
			\cdot
			\E [ e^{tW} ]
	\end{align*}
	with $\displaystyle c_{2,2} :=
		\frac{ 2^{4 + 6\edges{\G}} \cdot (\vertices{\G}!)^5 \cdot \edges{\G}^2}{\aut{\G}^2}$.
	
	Similar calculations yield
	\begin{align*}
		\vert S_{2,3}(t) \vert
		&\leq
			\frac{t}{\sigma^5}
			\cdot
			2\sum_{\mathclap{\substack{
				k,\ell\in E, k \neq \ell\\
				\graph_1,\graph_2 \in \copiesk{k}\\
				\graph_3,\graph_4 \in \copiesk{\ell}\\
				\graph_5 \in \copiesk{\ell}
			}}}
			\;\;\;
			\Bigg(\Bigg\vert\E\Bigg[
				B_k^c
				B_\ell^c
				\cdot
				\V{k}{1}{2}
				\V{\ell}{3}{4}
				\cdot
				B_{\graph_5}
				\cdot
				e^{
					\frac{t}{\sigma}
					\sum_{\graph \in \copies \backslash	\copiesk{\{k,\ell\}}}
					B_{\graph}^c
				}
			\Bigg]\Bigg\vert
			\\[-1.5\baselineskip] &\hspace{2.5cm}
			+
			\Bigg\vert\E\Bigg[
				B_k^c
				B_\ell^c
				\cdot
				\V{k}{1}{2}
				\V{\ell}{3}{4}
				\cdot
				\E[B_{\graph_5}]
				\cdot
				e^{
					\frac{t}{\sigma}
					\sum_{\graph \in \copies \backslash	\copiesk{\{k,\ell\}}}
					B_{\graph}^c
				}
			\Bigg]\Bigg\vert\Bigg)
			\\
		&\leq
			\frac{t}{\sigma^5}
			\cdot
			2\sum_{\mathclap{\substack{
				k,\ell\in E, k \neq \ell\\
				\graph_1,\graph_2 \in \copiesk{k}\\
				\graph_3,\graph_4,\graph_5 \in \copiesk{\ell}
			}}}
			\;\;\;
			\Bigg(
			\Bigg\vert\E\Bigg[
				q^2
				\cdot
				B_k B_\ell
				\cdot
				\V{k}{1}{2}
				\V{\ell}{3}{4}
				\cdot
				B_{\graph_5}
				\cdot
				e^{
					\frac{t}{\sigma}
					\sum_{\graph \in \copies \backslash	\copiesk{\{k,\ell\}}}
					B_{\graph}^c
				}
			\Bigg]\Bigg\vert
			\cdot \1{k \in \graph_3 \cup \graph_4 \cup \graph_5}
			\\[-\baselineskip] &\hspace{2.5cm}
			+
			\Bigg\vert\E\Bigg[
				q^2
				\cdot
				B_k B_\ell
				\cdot
				\V{k}{1}{2}
				\V{\ell}{3}{4}
				\cdot
				\E[B_{\graph_5}]
				\cdot
				e^{
					\frac{t}{\sigma}
					\sum_{\graph \in \copies \backslash	\copiesk{\{k,\ell\}}}
					B_{\graph}^c
				}
			\Bigg]\Bigg\vert
			\cdot \1{k \in \graph_3 \cup \graph_4}
			\Bigg)
			\\
		&\leq
			\frac{q^2 t}{\sigma^5}
			\cdot
			2\sum_{\mathclap{\substack{
				k,\ell\in E, k \neq \ell\\
				\graph_1,\graph_2 \in \copiesk{k}\\
				\graph_3,\graph_4,\graph_5 \in \copiesk{\ell}
			}}}
			p^{\vert \graph_1 \cup \graph_2 \cup \graph_3 \cup \graph_4 \cup \graph_5 \vert}
			\cdot
			\Big(
			e^{
				\frac{5Dt}{\sigma}
			}
			\cdot \1{k \in \graph_3 \cup \graph_4 \cup \graph_5}
			+
			e^{
				\frac{4Dt}{\sigma}
			}
			\cdot \1{k \in \graph_3 \cup \graph_4}
			\Big)
			\cdot
			\E [ e^{tW} ]
			\\
		&\leq
			\frac{q^2 t}{\sigma^5}
			\cdot
			2\cdot
			5\sum_{\mathclap{\substack{
				\graph_1 \in \copies\\
				\graph_2,\graph_3 \in \copiesk{\graph_1} \\
				\graph_4,\graph_5 \in \copiesk{\graph_3}
			}}}
			\;\;\;\;\;\;\;\;\;\;\;
			\sum_{\mathclap{\substack{
				k \in \graph_1 \cap \graph_2 \cap \graph_3\\
				\ell \in \graph_3 \cap \graph_4 \cap \graph_5
			}}}
			p^{\vert \graph_1 \cup \graph_2 \cup \graph_3 \cup \graph_4 \cup \graph_5 \vert}
			\cdot
			e^{
				\frac{5Dt}{\sigma}
			}
			\cdot
			\E [ e^{tW} ]
			\\
		&\leq
			\frac{10q^2 t}{\sigma^5}
			\cdot
			\edges{\G}^2
			\cdot	
			\frac{ (\vertices{\G}!)^{4} \cdot 2^{10\edges{\G}} }{\aut{\G}^5}
			\cdot
			\frac{n^{5\vertices{\G}} \cdot p^{5\edges{\G}}}{\Psimin^{4}}
			\cdot
			e^{
				\frac{5Dt}{\sigma}
			}
			\cdot
			\E [ e^{tW} ]
			\\
		&\leq
			c_{2,3}
			\cdot
			\frac{t}{\sqrt{q\Psimin}}
			\cdot
			\frac{1}{\Psimin}
			\cdot
			e^{
				\frac{5Dt}{\sigma}
			}
			\cdot
			\E [ e^{tW} ]
	\end{align*}
	with $\displaystyle c_{2,3} :=
		\frac{ 5 \cdot 2^{\frac72 + 10\edges{\G}} \cdot (\vertices{\G}!)^{\frac{13}2} \cdot \edges{\G}^2}{\aut{\G}^\frac52}$,
		and
	\begin{align*}
		\vert S_{2,4}(t) \vert
		&\leq
			\frac{t^2}{2\sigma^6}
			\cdot
			\sum_{\mathclap{\substack{
				k,\ell\in E\, k \neq \ell\\
				\graph_1,\graph_2 \in \copiesk{k}\\
				\graph_3,\graph_4 \in \copiesk{\ell}\\
				\graph_5,\graph_6 \in \copiesk{\{k,\ell\}}
			}}}
			\;
			\E\Bigg[
				\big\vert
				B_k^c
				B_\ell^c
				\cdot
				\V{k}{1}{2}
				\V{\ell}{3}{4}
				\cdot
				B_{\graph_5}^c
				B_{\graph_6}^c
				\big\vert
				\cdot
				e^{
					\frac{2Dt}{\sigma}
				}
				\cdot
				e^{
					\frac{t}{\sigma}
					\sum_{\graph \in \copies \backslash	\copiesk{\{k,\ell\}}}
					B_{\graph}^c
				}
			\Bigg]\\
		&\leq
			\frac{t^2}{2\sigma^6}
			\cdot
			\sum_{\mathclap{\substack{
				k,\ell\in E\, k \neq \ell\\
				\graph_1,\graph_2 \in \copiesk{k}\\
				\graph_3,\graph_4 \in \copiesk{\ell}\\
				\graph_5,\graph_6 \in \copiesk{\{k,\ell\}}
			}}}
			\;
			\E\Bigg[
				\big\vert
				B_k^c
				B_\ell^c
				\big\vert
				\cdot
				\V{k}{1}{2}
				\V{\ell}{3}{4}
				\cdot
				2B_{\graph_5}
				\cdot
				2B_{\graph_6}
			\Bigg]
			\cdot
			e^{
				\frac{2Dt}{\sigma}
			}
			\cdot
			e^{
				\frac{6Dt}{\sigma}
			}
			\cdot
			\E [ e^{tW} ]
			\\
		&\leq
			\frac{2t^2}{\sigma^6}
			\cdot
			\sum_{\mathclap{\substack{
				k,\ell\in E\, k \neq \ell\\
				\graph_1,\graph_2 \in \copiesk{k}\\
				\graph_3,\graph_4 \in \copiesk{\ell}\\
				\graph_5,\graph_6 \in \copiesk{\{k,\ell\}}
			}}}
			\;
			\E\big[
				\big\vert
				B_k^c
				\big\vert
			\big]
			\E\big[
				\big\vert
				B_\ell^c
				\big\vert
			\big]
			\E\Big[
				\V{k}{1}{2}
				\V{\ell}{3}{4}
				\cdot
				B_{\graph_5}
				B_{\graph_6}
			\Big\vert
				B_k=B_\ell=1
			\Big]
			\cdot
			e^{
				\frac{8Dt}{\sigma}
			}
			\cdot
			\E [ e^{tW} ]
			\\
		&\leq
			\frac{2t^2}{\sigma^6}
			\cdot
			\sum_{\mathclap{\substack{
				k,\ell\in E\, k \neq \ell\\
				\graph_1,\graph_2 \in \copiesk{k}\\
				\graph_3,\graph_4 \in \copiesk{\ell}\\
				\graph_5,\graph_6 \in \copiesk{\{k,\ell\}}
			}}}
			\;
			2pq
			\cdot
			2pq
			\cdot
			p^{\vert \graph_1 \cup \graph_2 \cup \graph_3 \cup \graph_4 \cup \graph_5 \cup \graph_6 \vert -2}
			\cdot
			e^{
				\frac{8Dt}{\sigma}
			}
			\cdot
			\E [ e^{tW} ]
			\\
		&\leq
			\frac{8q^2t^2}{\sigma^6}
			\cdot
			\Bigg(
			\;\;
			2\sum_{\mathclap{\substack{
				\graph_1,\graph_3 \in \copies\\
				\graph_2,\graph_5 \in \copiesk{\graph_1} \\
				\graph_4,\graph_6 \in \copiesk{\graph_3}
			}}}
			\edges{\G}^2
			p^{\vert \graph_1 \cup \graph_2 \cup \graph_3 \cup \graph_4 \cup \graph_5 \cup \graph_6 \vert}
			+
			2\sum_{\mathclap{\substack{
				\graph_1,\graph_3 \in \copies\\
				\graph_2,\graph_5,\graph_6 \in \copiesk{\graph_1} \\
				\graph_4 \in \copiesk{\graph_3}
			}}}
			\edges{\G}^2
			p^{\vert \graph_1 \cup \graph_2 \cup \graph_3 \cup \graph_4 \cup \graph_5 \cup \graph_6 \vert}
			\Bigg)
			\cdot
			e^{
				\frac{8Dt}{\sigma}
			}
			\cdot
			\E [ e^{tW} ]
			\\
		&\leq
			\frac{32q^2t^2}{\sigma^6}
			\cdot
			\edges{\G}^2
			\cdot
			\frac{ (\vertices{\G}!)^{5} \cdot 2^{15\edges{\G}} }{\aut{\G}^6}
			\cdot
			\frac{n^{6\vertices{\G}} \cdot p^{6\edges{\G}}}{\Psimin^{4} \cdot \min\{ \Psimin , 1 \} }
			\cdot
			e^{
				\frac{8Dt}{\sigma}
			}
			\cdot
			\E [ e^{tW} ]
			\\
		&\leq
			c_{2,4}
			\cdot
			\frac{t^2}{q \min\{ \Psimin , 1 \}}
			\cdot
			\frac{1}{\Psimin}
			\cdot
			e^{
				\frac{8Dt}{\sigma}
			}
			\cdot
			\E [ e^{tW} ]
	\end{align*}
	with $\displaystyle c_{2,4} :=
		\frac{ 2^{8 + 15\edges{\G}} \cdot (\vertices{\G}!)^8 \cdot \edges{\G}^2}{\aut{\G}^3}$.
		
	Note that $c_{2,4} = \max \{ c_{2,1}, c_{2,2}, c_{2,3}, c_{2,4} \}$.
	Bringing together the bounds for $S_{2,1}$, $S_{2,2}$, $S_{2,3}$, and $S_{2,4}$
	results in
	\begin{align*}
		\E\Big[
			\Big\vert
				\sum_{k \in E} Y_k U_k
			\Big\vert
			\cdot e^{tW}
		\Big]
		&\leq
			\big(
				S_{2,1}(t) + S_{2,2}(t) + S_{2,3}(t) + S_{2,4}(t)
			\big)^{\frac12}
			\cdot \E[e^{tW}]^{\frac12}
			\\
		&\leq
			\bigg(
				c_{2,1}
				\cdot
				\frac{1}{q\Psimin}
				\cdot
				e^{
					\frac{4Dt}{\sigma}
				}
				\cdot
				\E [ e^{tW} ]
			\\&\hspace{1em}+
				c_{2,2}
				\cdot
				\frac{1}{\Psimin}
				\cdot
				e^{
					\frac{4Dt}{\sigma}
				}
				\cdot
				\E [ e^{tW} ]
			\\&\hspace{1em}+
				c_{2,3}
				\cdot
				\frac{t}{\sqrt{q\Psimin}}
				\cdot
				\frac{1}{\Psimin}
				\cdot
				e^{
					\frac{5Dt}{\sigma}
				}
				\cdot
				\E [ e^{tW} ]
			\\&\hspace{1em}+
				c_{2,4}
				\cdot
				\frac{t^2}{q \min\{ \Psimin , 1 \}}
				\cdot
				\frac{1}{\Psimin}
				\cdot
				e^{
					\frac{8Dt}{\sigma}
				}
				\cdot
				\E [ e^{tW} ]
			\bigg)^{\frac12}
			\cdot \E[e^{tW}]^{\frac12}
			\\
		&\leq
			\sqrt{c_{2,4}}
			\cdot
			\bigg(
				2 + \frac{t}{\sqrt{\Psimin}} + \frac{t^2}{\min\{ \Psimin , 1 \}}
			\bigg)^\frac12
			\frac{1}{\sqrt{q\Psimin}}
			\cdot
			e^{
				\frac{4Dt}{\sigma}
			}
			\cdot
			\E [ e^{tW} ]
			\\
		&\leq
			\gamma_2(t)
			\cdot
			\E [ e^{tW} ]
	\end{align*}
	with
	\begin{align}
		\gamma_2(t) :={}&
			c_{A2}
			\cdot
			\bigg(
				1 + \frac{t}{\min\{\sqrt{\Psimin},1\}}
			\bigg)
			\frac{1}{\sqrt{q\Psimin}}
			\cdot
			e^{
				\frac{4Dt}{\sigma}
			}
			,
			\label{sg:fct:g2}
	\end{align}
	where
	$\displaystyle c_{A2}
	:= 
		\sqrt{2c_{2,4}}
	=
		\frac{ 2^{\frac92 + \frac{15}2\edges{\G}} \cdot (\vertices{\G}!)^4 \cdot \edges{\G}}{\aut{\G}^\frac32}
	= c_{A1}$
	for all $t \geq 0$.
	
	\textbf{Conclusion}:
	We have constructed $\gamma_1(t)$ and $\gamma_2(t)$ for $t \geq 0$, see \eqref{sg:fct:g1} and \eqref{sg:fct:g2},
	so that the assumptions \eqref{sg:A1} and \eqref{sg:A2}
	that we formulated at the beginning of this proof of Theorem ~\ref{sg:theo:main} on p.~\pageref{sg:A1}
	are fulfilled.
	Finally, note that
	\begin{align*}
		\gamma_1(t) + \gamma_2(t)
		\leq 2c_\G \cdot s(t)
	\end{align*}
	with $s(t)$ and $c_\G$ as defined in the statement of Theorem~\ref{sg:theo:main}.
	Theorem~\ref{Theorem:MD_Rademacher-short} yields the desired result.
\end{proof}

\begin{proof}[Proof of Corollary~\ref{sg:cor:main}]
	If we restrict $t$ to be smaller than 
	$c_1 \cdot \frac{n^2p^{\edges{\G}}\sqrt{q}}{\sqrt{\Psimin}}$,
	where $c_1 > 0$ is an arbitrary positive number,
	we can further simplify our results.
	Under this restriction,
	due to $n^2p^{\edges{\G}} \leq \Psimin$, there is
	$t \leq c_1 \cdot \frac{n^2p^{\edges{\G}}\sqrt{q}}{\sqrt{\Psimin}} \leq c_1 \cdot \sqrt{q\Psimin}$ and hence
	\begin{align*}
		\frac{t}{\min\{\sqrt{\Psimin},1\}}
		&\leq \max \Big\{
				c_1 \cdot \sqrt{q}
				\, , \,
				t
			 \Big\}
		\leq c_1 \cdot \sqrt{q} + t
		.
	\end{align*}
	
	On the other hand,
	since we are working in the case of $n \geq 4\vertices{\G}^2$,
	we can use explicit bounds for $D$ and $\sigma$
	from Lemma~4.2 in \cite{ER22}
	to show that
	$\frac{D}{\sigma} \leq \hat c_\G \cdot \frac{\sqrt{\Psimin}}{n^2p^{\edges{\G}}\sqrt{q}}$,
	where
	$\hat c_\G = \sqrt{2} \cdot \frac{ \sqrt{\vertices{\G}!} \cdot \vertices{\G}^2 \cdot \edges{\G} }{ \sqrt{\aut{\G}} }$
	is a constant that only depends on $\G$.
	This implies that
	$e^{\frac{5Dt}{\sigma}} \leq e^{5c_1 \hat c_\G}$
	for $0 \leq t \leq c_1 \cdot \frac{n^2p^{\edges{\G}}\sqrt{q}}{\sqrt{\Psimin}}$.
	
	Additionally, in the case we focus on, there is
	$e^{c_\G \cdot t^2 \cdot s(t)} \leq e^{c_\G c_2}$.
	
	Application of these bounds to the result of Theorem~\ref{sg:theo:main} leads to
	\begin{align*}
		\bigg\vert
			\frac{\Pp(W>t)}{1-\Phi(t)} -1
		\bigg\vert
		&\leq
			50 c_\G
			\cdot e^{c_\G c_2}
			\cdot (1+t^2)
			\cdot (
				1
				+
				c_1 \sqrt{q}
				+
				t
			)
			\cdot \frac{1}{\sqrt{q\Psimin}}
			\cdot e^{5c_1 \hat c_\G}\\
		&\leq
			50 c_\G
			\cdot e^{5c_1 \hat c_\G + c_2 c_\G}
			\cdot (1+c_1 \sqrt{q})
			\cdot (1+t^2)
			\cdot (1+t)
			\cdot \frac{1}{\sqrt{q\Psimin}}
		,
	\end{align*}
	To finish the proof, note that $(1+t^2)(1+t) \leq 2(1+t^3)$.
\end{proof}

\section{Acknowledgement}
\noindent This research has been supported in part by the German Research Foundation (DFG) under project number 459731056.

\end{document}